\documentclass[10pt, reqno]{amsart}
\usepackage{amscd, amssymb,latexsym,amsmath, amscd, amsmath, comment, hyperref}
\usepackage[all]{xy}
\usepackage{anysize}
\usepackage[sc]{mathpazo} 

\usepackage{bbm}

\usepackage{color}
\marginsize{2.0cm}{2.0cm}{2.0cm}{1.9cm}

\numberwithin{equation}{section}

\newtheorem{theorem}{Theorem}[section] 
\newtheorem{lemma}[theorem]{Lemma}
\newtheorem{corollary}[theorem]{Corollary}
\newtheorem{proposition}[theorem]{Proposition}
\theoremstyle{definition}

\theoremstyle{remark}
\newtheorem{remark}{Remark}

\begin{document}
	
        \title[Modular representations of $\mathrm{GL}_2({\mathbb F}_q)$]{Modular representations of $\mathrm{GL}_2({\mathbb F}_q)$ using calculus}
	\author{Eknath Ghate and Arindam Jana}


        \address{School of Mathematics, Tata Institute of Fundamental Research, Homi Bhabha Road, Mumbai - 400005, India.}	
	\email{eghate@math.tifr.res.in}
	
	\email{arindamj@iiserbpr.ac.in}

	\date{}

        \begin{abstract}
          We show that certain modular induced representations of 
          $\mathrm{GL}_2({\mathbb F}_q)$ can be written as cokernels of
          operators acting on symmetric power representations of  $\mathrm{GL}_2({\mathbb F}_q)$. When
          the induction is from the Borel subgroup, respectively the anisotropic torus, the operators involve multiplication 
          by newly defined twisted Dickson polynomials, respectively, twisted Serre  operators. Our 
          isomorphisms are explicitly defined using differential operators.
          As a corollary, we improve some
          periodicity results for quotients in the theta filtration.
	\end{abstract}
	\maketitle
	
	

	\section{Introduction}

        Let $G$ be the general linear group $\mathrm{GL}_2$ and let $p$ be a prime.
        Let $V_r$ for $r\geq0$ be the $r$-th symmetric power representation of the standard two-dimensional
        representation of $G(\mathbb F_p)$. It
        is modeled on homogeneous polynomials of degree $r$ over $\mathbb F_p$ in two variables $X$ and $Y$ with the usual
        action of $G(\mathbb F_p)$. Let $\theta = X^pY - XY^p$ be the theta or Dickson polynomial on which $G(\mathbb{F}_p)$ acts
        by determinant.
        Let $V_r^{(m+1)}$ for $m \geq 0$ be the sub-representation of $V_r$ consisting of polynomials
        divisible by $m+1$ copies of $\theta$. The sequence $V_r^{(m+1)}$ is
        called the theta filtration of $V_r$.
          
	\subsection{Principal series}

        It is a classical fact going back to Glover \cite[(4.2)]{glo78} that $\frac{V_r}{V_r^*}$ is periodic in $r$ with period $p-1$ where
	$V_r^* = V_r^{(1)}$. This is proved nowadays  by noting that 
	\begin{eqnarray}
          \label{ps-Glover}
          \frac{V_r}{V_r^*} \simeq \mathrm{ind}_{B(\mathbb{F}_p)}^{G(\mathbb{F}_p)} d^r
        \end{eqnarray}
	is a principal series representation of $G(\mathbb{F}_p)$ obtained by inducing the character $d^r$ of the Borel subgroup
        $B({\mathbb F}_p) = \{ \left( \begin{smallmatrix} a & b \\ 0 & d \end{smallmatrix} \right) \}$ to $G({\mathbb F}_p)$
        \cite[Lemma 2.4]{roz14}, and by
	noting that the character depends only on $r$ modulo $(p-1)$. Similar periodicity results have been investigated for
	higher quotients in the theta filtration of $V_r$.  Indeed, it was shown in \cite[Lemma 4.1]{GV22}
        that for $0 \leq m \leq p-1$,
	the quotient $\frac{V_r}{V_r^{(m+1)}}$ is periodic in $r$ modulo $p(p-1)$ by constructing an embedding
	\begin{eqnarray}
          \label{dual}
          \frac{V_r}{V_r^{(m+1)}} \hookrightarrow \mathrm{ind}_{B(\mathbb{F}_p[\epsilon])}^{G(\mathbb{F}_p[\epsilon])} d^r,
        \end{eqnarray}
	where $\mathbb{F}_p[\epsilon]$ is the ring of generalized  dual numbers (with $\epsilon^{m+1} = 0$), noting that $d^r$ only
        depends on $r$ modulo $p(p-1)$,
	and by showing that the image of \eqref{dual} is independent of $r$ modulo $p(p-1)$.
	
	The map \eqref{dual} is no longer surjective when $m > 0$. In this paper, 
	instead of working with generalized dual numbers and {\it characters} of the inducing subgroup, we work with the induction
	of {\it higher dimensional representations} of the inducing subgroup and obtain {\it isomorphisms} between $\frac{V_r}{V_r^{(m+1)}}$
        and induced spaces. We have:
        \begin{theorem}
           \label{ps}
           Let $0 \leq m \leq p-1$ and $r\geq (m+1)(p+1)-1$. Then, we have the following explicit isomorphisms:
           \begin{enumerate}
              \item If $p \nmid {r \choose m}$,  then 
                \begin{eqnarray}
                  \label{ps-non-split}
                  \frac{V_r}{V_r^{(m+1)}} \simeq \mathrm{ind}_{B(\mathbb{F}_p)}^{G(\mathbb{F}_p)} ( V_m \otimes d^{r-m} ), 
                \end{eqnarray}
                where  $V_m$ is the representation of $B(\mathbb{F}_p)$ obtained by restriction from $G(\mathbb{F}_p)$.
              \item If $p \mid {r \choose m}$ and $m = 1$, then
              	$$\frac{V_r}{V_r^{(2)}} \simeq \mathrm{ind}_{B(\mathbb{F}_p)}^{G(\mathbb{F}_p)} ( V_1^{\mathrm{ss}} \otimes d^{r-1} ),$$
                  where $V_1^{\mathrm{ss}}$ is the split representation of $B(\mathbb{F}_p)$ obtained as
                  the semi-simplification of $V_1$.
           \end{enumerate}
        \end{theorem}
        The map in \eqref{ps-non-split} generalizes the `evaluation of the polynomial at the second row of a matrix'
        map that is used to prove \eqref{ps-Glover} and \eqref{dual}. Additionally, it involves the use
        of a differential operator $\nabla$. Such operators have found
        sporadic use in the literature (see \cite{glo78} and \cite{bg15}), but are used systematically throughout this paper.
        We also need to divide by constants which do not
        vanish if $p \nmid {r \choose m}$ (cf. \eqref{def of psi}).
        
        Consider now the case $p | {r \choose m}$ with $1 \leq m \leq p-1$, that is, $r$ is in one of the congruence classes
        $0, 1,\ldots, m-1$ modulo $p$.
	While it is well known
	that the individual principal series $\frac{V_r^{(i)}}{V_r^{(i+1)}}$ occurring as subquotients in the theta filtration are
	extensions of two Jordan-H\"older factors (which split exactly when $r \equiv 2i$
        mod $(p-1)$), it is not as clear whether the extensions
	between consecutive principal series 
	\begin{eqnarray}
          \label{extns-of-ps}
           0 \rightarrow \dfrac{V_r^{(i+1)}}{V_r^{(i+2)}} \rightarrow \dfrac{V_r^{(i)}}{V_r^{(i+2)}} \rightarrow \dfrac{V_r^{(i)}}{V_r^{(i+1)}}
          \rightarrow 0
        \end{eqnarray}
        for $0 \leq i \leq m-1$ occurring in the theta filtration split. We show that if $p | {r \choose m}$, then there is exactly one $i$
        such that \eqref{extns-of-ps} splits.
	Indeed, since 
	$\frac{V_r^{(i)}}{V_r^{(i+2)}} \simeq \frac{V_{r'}}{V_{r'}^{(2)}} \otimes \det^{i}$ with
	$r' = r - i(p+1) \equiv r-i$ mod $p$ we are reduced to analyzing the case $m = 1$.  If $p \nmid r$, then $\frac{V_r}{V_r^{(2)}}$ does not split
        by \eqref{ps-non-split} and  by (the converse of) the exactness of induction \cite[\S 8, Lemma 6 (5)]{alp}.
        Modifying the above mentioned
	differential operator $\nabla$ by dropping the constant mentioned above (cf. \eqref{def of psi split}), in the second part of the theorem
        we show that if $p | r$, then 
	$$\frac{V_r}{V_r^{(2)}} \simeq \mathrm{ind}_{B(\mathbb{F}_p)}^{G(\mathbb{F}_p)} ad^{r-1} \oplus \mathrm{ind}_{B(\mathbb{F}_p)}^{G(\mathbb{F}_p)}  d^{r}$$
	splits.
        In other words, if $p | r$, then the two-dimensional standard representation $V_1$ in 
        \eqref{ps-non-split}  gets replaced by the
	split representation $V_1^{\mathrm{ss}} = a \oplus d$ of $B(\mathbb{F}_p)$. We deduce that the extensions
        \eqref{extns-of-ps}
        split exactly when $r \equiv i$ mod $p$; moreover, under this condition, we have: 
        \begin{eqnarray}
          \label{split}
          \frac{V_r}{V_r^{(m+1)}} \simeq  \frac{V_r}{V_r^{(i+1)}} \oplus \frac{V_r^{(i+1)}}{V_r^{(m+1)}}.
        \end{eqnarray}
        
        As a corollary of Theorem~\ref{ps}, we obtain a strengthening of the afore-mentioned periodicity result from \cite{GV22}
        in the case that $p \nmid {r \choose m}$ since again the right hand side of \eqref{ps-non-split} only depends on $r$ modulo $(p-1)$ and not
        on $r$ modulo $p$.
        Thus, to obtain periodicity in this case,
        we no longer need to restrict to $r$ in a fixed congruence class modulo $p$, only to those $r$ that avoid collectively
        the congruence classes $0,1,\ldots, m-1$ mod $p$. We obtain:
        \begin{corollary}
          Let $0 \leq m \leq p-1$ and  $r$, $s \geq (m+1)(p+1)-1$ with  $r \equiv s \mod (p-1)$.
          If $p \nmid {r \choose m}, {s \choose m}$, then
            \begin{eqnarray}
              \label{periodicity}
                \frac{V_r}{V_r^{(m+1)}} \simeq \frac{V_s}{V_s^{(m+1)}}.
            \end{eqnarray}  
          \end{corollary}

	
	With future applications in mind, we equally treat the case of $G({\mathbb F}_q) = \mathrm{GL}_2({\mathbb F}_q)$ for an
	arbitrary finite field ${\mathbb F}_q$ with $q = p^f$ elements for $f \geq 1$. Indeed, we prove the following twisted version
        of the isomorphism \eqref{ps-non-split}: 
        \begin{theorem}
           \label{ps-twisted}
           Let $m = \sum_{i=0}^{f-1} m_i p^i$  with $0\leq m_i\leq p-1$ and 
           $V_m = \otimes_{i=0}^{f-1} (V_{m_i} \circ \mathrm{Fr}^i)$.
           Let $r =  \sum_{i=0}^{f-1} r_i p^i$  with $r_i\geq (m_i+1)(q+1)-1$ and $V_r = \otimes_{i=0}^{f-1} (V_{r_i} \circ \mathrm{Fr}^i)$. 
	   Let $d^{r-m}$ be the character $\otimes_{i=0}^{f-1} d^{(r_i-m_i)p^i}$ of $B({\mathbb F}_q)$.
           If $p \nmid {r \choose m} \equiv \prod_{i=0}^{f-1} {r_i \choose m_i} \mod p$, then
           \begin{eqnarray*}
           \dfrac{V_r}{\langle \theta_0^{m_0+1}, \theta_1^{m_1+1}, \cdots, \theta_{f-1}^{m_{f-1}+1} \rangle} \> \simeq \>
                            \mathrm{ind}_{B({\mathbb F}_q)}^{G({\mathbb F}_q)} \left(V_m \otimes d^{r-m} \right).
	   \end{eqnarray*}
        \end{theorem}
	Here the $V_{r_i}$ are modeled on homogeneous polynomials of degree $r_i$ over ${\mathbb F}_q$
        in the variables $X_i$ and $Y_i$ and $V_{r_i} \circ \mathrm{Fr}^i$ means that we twist the
        standard action of $G({\mathbb F}_q)$ on $V_{r_i}$ by the $i$-th power of Frobenius.  The polynomials $$\theta_i = X_{i}Y_{i-1}^p - Y_iX_{i-1}^{p}$$ for
        $0 \leq i \leq f-1$ are what we call {\it twisted Dickson polynomials} (we adopt the convention that $-1=f-1$; alternatively, one could index the
        variables by the group ${\mathbb Z}/f{\mathbb Z}$);
        they do not seem to appear in the literature.

        If $p \nmid {r \choose m}$, then the periodicity of the quotient on the left hand side
        again follows since the right hand side
        only depends on the character $d^r$ which is periodic in $r$ modulo $(q-1)$. We obtain:
        \begin{corollary}
          \label{periodicity-twisted}
          Let $m=\sum_{i=0}^{f-1}m_ip^i$ with $0 \leq m_i \leq p-1$. Let $r=\sum_{i=0}^{f-1}r_ip^i$ and
          $s=\sum_{i=0}^{f-1}s_ip^i$ with $r_i$, $s_i\geq (m_i+1)(q+1)-1$ and  $r \equiv s \mod (q-1)$.
          If $p \nmid {r \choose m}, {s \choose m}$, then
          $$ \dfrac{V_r}{\langle \theta_0^{m_0+1}, \theta_1^{m_1+1}, \cdots, \theta_{f-1}^{m_{f-1}+1} \rangle} \> \simeq \> \dfrac{V_s}{\langle \theta_0^{m_0+1}, \theta_1^{m_1+1}, \cdots, \theta_{f-1}^{m_{f-1}+1} \rangle}. $$
        \end{corollary}
              
        We end this discussion of the case of principal series by noting that the isomorphism in Theorem~\ref{ps-twisted} is expected to
        play an important role in  investigations into
        the reduction problem of two-dimensional crystalline \cite{ars21, bl22, blz04, bha20, bg15, bgr18, bgv25, bg09, bg13, gg15, gk24, gha21, gr23, GV22, gha24, np19} and semi-stable representations \cite{bll23, bm10, chi24, cg24, cgy25, gp19, lp22}
	over arbitrary $p$-adic fields $F$ with residue field ${\mathbb F}_q$ using the compatibility with respect to reduction of the (yet to be
	discovered) $p$-adic and mod $p$ local Langlands correspondences for $G(F)$ for an arbitrary finite extension of ${\mathbb Q}_p$.

        \subsection{Cuspidal case}

        Let $T(\mathbb F_p) = {\mathbb F}_{p^2}^\times \hookrightarrow G({\mathbb F}_p)$ be the anisotropic torus. 
        The theme of writing (generalized) principal series representations as cokernels of theta operators raises the question (asked by Khare)
        as to whether one may similarly write representations induced from the anisotropic torus as cokernels of symmetric power representations.

        Let $D$ be the differential operator $X^p \frac{\partial}{\partial X} + Y^p \frac{\partial}{\partial Y}$ and let
        $\omega_2$ be the identity character $T(\mathbb F_p) = {\mathbb F}_{p^2}^\times \rightarrow {\mathbb F}_{p^2}^\times$. We prove
        the following analog of \eqref{ps}:
        \begin{theorem}
          \label{cuspidal}  
          Let $2 \leq r \leq p-1$. Then there is an explicit isomorphism (cf. \eqref{explicit iso cuspidal})
          $$\frac{V_{r+p-1}}{D(V_r)} \otimes V_{p-1} \simeq \mathrm{ind}_{T(\mathbb{F}_p)}^{G(\mathbb{F}_p)} \> \omega_2^{r}$$
          defined over $\mathbb{F}_{p^2}$.
        \end{theorem}
        One may compare Theorems~\ref{ps} and \ref{cuspidal} by noting that
        if we tensor both sides of \eqref{ps-non-split} by $V_{p-1}$, then the right hand side of that isomorphism gets replaced by
        induction from the {\it split torus} $T^\mathrm{sp}({\mathbb F}_p) \subset B({\mathbb F}_p)$ to $G({\mathbb F}_p)$ of $V_m \otimes d^{r-m}$ (Remark~\ref{split induced}).
        Theorem~\ref{cuspidal} is also true for $r = 1$ (see Remark~\ref{r=p-2,p-1}).
        We prove similar isomorphisms for other values of $r$ by twisting (Corollary~\ref{bigger range})
        using the fact that $D$ preserves the theta filtration in a strong sense (Lemma~\ref{divide}).
        We also prove similar isomorphisms when $D$ is replaced by a higher power $D^{(m+1)}$ (Corollary \ref{higher m}, note the
        analogy with \eqref{ps-non-split}).
        
        A non-explicit version of the isomorphism in Theorem~\ref{cuspidal} can be deduced from the work of Reduzzi \cite{red10}. Let us provide some background
        and explain our contribution. In the discussion that follows, we sometimes think 
        of $\omega_2$ as a character taking values in a characteristic zero field (by taking its Teichm\"uller lift).
        Recall that for each complex character $\chi$ of $T({\mathbb F}_p) = {\mathbb F}^\times_{p^2}$ (with $\chi$ not self-conjugate)
        there is an irreducible cuspidal complex representation $\Theta(\chi)$ of $G({\mathbb F}_p)$. Moreover, $\Theta(\chi)$ is
        a factor of an induced representation: we have
        \begin{eqnarray}
          \label{cuspidal tensor St}
          {\Theta(\chi)} \otimes \mathrm{St} \simeq \mathrm{ind}_{T(\mathbb{F}_p)}^{G(\mathbb{F}_p)} \> \chi,
        \end{eqnarray}
        where $\mathrm{St}$ is the $p$-dimensional complex irreducible Steinberg representation of $G({\mathbb F}_p)$ with
        reduction $\overline{\mathrm{St}} \simeq V_{p-1}$.
        While the group $G({\mathbb F}_p)$ has no mod $p$ cuspidal representations (since, for instance, the Jacquet functor is never $0$
        because there are always invariant elements under the upper unipotent subgroup of $G({\mathbb F}_p)$),
        one may still study the mod $p$ reductions of $\Theta(\omega_2^r)$.
        Following a suggestion of Serre to use the operator $D$,
        Reduzzi \cite{red10} proved that
        the mod $p$ reduction $\overline{\Theta(\omega_2^r)}$ is isomorphic to the cokernel of $D$ on an
        appropriate symmetric power representation, namely:
        \begin{eqnarray}
           \label{reduzzi}
          \frac{V_{r+p-1}}{D(V_r)} \simeq \overline{\Theta(\omega_2^r)}
        \end{eqnarray}
        for $2 \leq r \leq p-1$.
        The proof uses a specific integral model of $\Theta(\omega_2^r)$ arising from the action of
        $G({\mathbb F}_p)$ on the crystalline cohomology of the Deligne-Lusztig variety $XY^{p}-X^{p}Y = Z^{p+1}$ (see Haastert-Jantzen \cite{hj90}).
        Thus, Reduzzi's isomorphism   \eqref{reduzzi} is not at all explicit given that the right hand side involves crystalline cohomology.
        However, by tensoring  \eqref{reduzzi} with $V_{p-1}$ and using the mod $p$ reduction of \eqref{cuspidal tensor St} for $\chi = \omega_2^r$,
        one sees that the isomorphism in Theorem~\ref{cuspidal} must hold, at least abstractly.
        An immediate question that arises is whether one can make this isomorphism explicit, given that the right hand
        side of this isomorphism no longer involves crystalline cohomology. Thus the point
        of Theorem~\ref{cuspidal} is that it contains an {\it explicit} isomorphism (which was found after much computation with special cases).
        Again, the map involves a differential operator $\nabla_\alpha$, where $\alpha$ is an element of ${\mathbb F}_{p^2} \setminus {\mathbb F}_p$,
        which generalizes the operator $\nabla$ used in the principal series case.\footnote{It also involves the difference of a polynomial
        evaluated at two points, which looks like the evaluation of a direct integral in calculus.}

        In fact, Reduzzi \cite{red10} proved\footnote{Technically speaking, Reduzzi does not treat the
        case $r = \frac{p+1}{2}$, though it is covered by
        Theorems~\ref{cuspidal} and \ref{cuspidal-twisted}.} that, more generally, for $G({\mathbb F}_q)$ with $q = p^f$, and for $\omega_{2f}: {\mathbb F}^\times_{q^2} \rightarrow
      {\mathbb F}^\times_{q^2}$ the fundamental (identity) character of level $2f$, one similarly has
        $$\frac{V_{r+q-1}}{D(V_r)} \simeq \overline{\Theta(\omega_{2f}^r)},$$
        for $2 \leq r \leq p-1$, where now $D = X^q \frac{\partial}{\partial X} + Y^q \frac{\partial}{\partial Y}$.  
        We end this paper by proving the following twisted version of this result, which extends Theorem~\ref{cuspidal} to  $G({\mathbb F}_q)$:
        \begin{theorem}
          \label{cuspidal-twisted}
          Let $r=r_0+r_1p+\dots+r_{f-1}p^{f-1},$ where $2\leq r_0\leq p-1$ and $r_j=0$ for all $1\leq j\leq f-1.$
          Then there is an explicit isomorphism over  ${\mathbb F}_{q^2}$:
          \[\frac{\bigotimes_{j=0}^{f-1}V_{r_j+p-1}^{{\rm Fr}^j}}{\langle D_0,\dots, D_{f-1}\rangle} \otimes
            \otimes_{j=0}^{f-1} V_{p-1}^{{\rm Fr}^j} \> \simeq \> {\rm ind}_{T(\mathbb{F}_q)}^{G(\mathbb{F}_q)}\> \omega_{2f}^r.\]
        \end{theorem}
        We remark that $\otimes_{j=0}^{f-1} V_{p-1}^{{\rm Fr}^j} \simeq \overline{\mathrm{St}}$ where $\mathrm{St}$ is
        now the $q$-dimensional Steinberg representation of $G({\mathbb F}_q)$. Also in the statement of the theorem
        we need the following twisted versions
        of Serre's differential operator $D$, namely:
        \[D_0= X_0^pX_1^{p-1}\cdots X_{f-1}^{p-1}\dfrac{\partial}{\partial X_0}+Y_0^pY_1^{p-1}\cdots Y_{f-1}^{p-1}\dfrac{\partial}{\partial Y_0},\]
        and
        \[D_j= X_0^pX_1^{p-1}\cdots X_{j-1}^{p-1}\dfrac{\partial}{\partial X_j}+Y_0^pY_1^{p-1}\cdots Y_{j-1}^{p-1}\dfrac{\partial}{\partial Y_j}, \]
        for all $1\leq j\leq f-1$. Interestingly, the $D_j$ are only $G({\mathbb F}_q)$-linear 
        modulo ${\langle D_1,\dots, D_{j-1}\rangle}$ for all $1 \leq j\leq f$
        (with the convention that $D_f$ is to be thought of as $D_0$, cf. Lemma~\ref{G-linearity of D_j}). Again, the $D_j$
        do not seem to appear in the literature and one might refer to them as {\it twisted Serre operators}. 

	\section{Principal series case}
	
	\subsection{The case of ${\rm GL}_2(\mathbb{F}_p)$}
	Recall $G(\mathbb{F}_p) = {\rm GL}_2(\mathbb{F}_p)$ and $B(\mathbb{F}_p)$ is the subgroup of upper triangular matrices of $G(\mathbb{F}_p).$
        For $r\geq 0,$ let $V_r:={\rm Sym}^r(\mathbb{F}_p^2)$ denote the $r$-th symmetric power of the standard representation of $G(\mathbb{F}_p)$
        over $\mathbb{F}_p.$ We identify $V_r$ with homogeneous polynomials $P(X,Y)$ of degree $r$ in two variables $X$ and $Y$ with coefficients
        in $\mathbb{F}_p$, with action of $g = \left(\begin{smallmatrix} a & b \\ c & d \end{smallmatrix} \right) \in G(\mathbb{F}_p)$ given by
        $$g \cdot P(X,Y) = P(aX+cY,bX+dY).$$
        Consider the Dickson polynomial 
	\[\theta(X, Y):=X^pY-XY^p.\]
	Note that $G(\mathbb{F}_p)$ acts on $\theta(X, Y)$ by the determinant character. So for each $m\geq 0$, we have
	\[V_r^{(m+1)} := \left\{f(X, Y)\in V_r \mid f(X, Y)~\text{is divisible by}~\theta(X, Y)^{m+1}\right\}\]
	is a $G(\mathbb{F}_p)$-stable subspace of $V_r.$ These spaces give a decreasing filtration of submodules of $V_r$:
	\[V_r\supset V_r^{(1)}\supset \dots\supset V_r^{(m+1)} \supset \dots\supset (0).\]
	Let $d^r:B(\mathbb{F}_p)\rightarrow\mathbb{F}_p^{\times}$ denote the character given by $\left( \begin{smallmatrix} a & b \\ 0 & d \end{smallmatrix} \right)
        \mapsto d^r$.
	For $n \geq 0$, $m \in {\mathbb Z}$, define
        $$[n]_m = \begin{cases}
                    1,                      & \text{if } m =0, \\
                    n(n-1) \cdots(n-(m-1)), & \text{if } m>0, \\
                    0,                      & \text{if } m < 0. 
                  \end{cases}	
                  $$

        \begin{lemma}
        Let $k \geq 0$. We have            
        \begin{equation}
          \label{sum}
          \sum\limits_{m=0}^k{k \choose m}[r-j]_{t-l-m}[j]_{l+m}=[r-t+k]_{k}[r-j]_{t-k-l}[j]_l.
        \end{equation}
      \end{lemma}
        \begin{proof}
          Follows by induction on $k$.
              \end{proof}
              
    We introduce the notation: $P(X,Y) \big\vert_{(c, d)} = P(c,d)$ for a polynomial $P$ in two variables $X$, $Y$ and scalars $c$, $d$.        
	\begin{lemma}\label{evaluate}
          Let $z\in \mathbb{F}_p$ and let $P(X, Y)=\sum\limits_{j=0}^ra_jX^{r-j}Y^j\in V_r$ with $a_j\in \mathbb{F}_p$ for all $0\leq j \leq r.$
          Define the differential operators 
                  \[\nabla= a\frac{\partial}{\partial X}+b\frac{\partial}{\partial Y} \quad\text{and}\quad \nabla'= c\frac{\partial}{\partial X}+d\frac{\partial}{\partial Y} ,\]
                for $a, b, c, d \in \mathbb{F}_p$.
		Then, for all $0\leq k \leq t,$ we have
		\begin{eqnarray*}
			\nabla^{t-k}\nabla'^k(P)\bigg\vert_{(zc, zd)}
			&=& z^{r-t}[r-t+k]_{k}\left(a\frac{\partial}{\partial X}+b\frac{\partial}{\partial Y}\right)^{t-k}(P)\bigg\vert_{(c,d)}.
		\end{eqnarray*}
		
	\end{lemma}
	
	\begin{proof}
		Without loss of generality assume that $P=X^{r-j}Y^j.$ We note that
		\begin{equation}\label{pderivatives}
			\frac{\partial^n}{\partial X^{n-i}\partial Y^i}(P)=[r-j]_{n-i}[j]_iX^{r-j-(n-i)}Y^{j-i}.
		\end{equation}		
		Now,
{\small		\begin{align*}
			&\left(a\frac{\partial}{\partial X}+b\frac{\partial}{\partial Y}\right)^{t-k}\left(c\frac{\partial}{\partial X}+d\frac{\partial}{\partial Y}\right)^{k}(P)\bigg\vert_{(zc, zd) } \\
			&=\sum\limits_{l=0}^{t-k}\sum\limits_{m=0}^k{t-k \choose l}{k \choose m}a^{t-k-l}b^lc^{k-m}d^m\frac{\partial^t}{\partial X^{t-l-m}\partial Y^{l+m}}(P)\bigg\vert_{(zc, zd)}\\
			&=\sum\limits_{l=0}^{t-k}\sum\limits_{m=0}^k{t-k \choose l}{k \choose m}z^{r-t}a^{t-k-l}b^lc^{k-m}d^m[r-j]_{t-l-m}[j]_{l+m}c^{r-j-(t-l-m)}d^{j-l-m}\\
			&=\sum\limits_{l=0}^{t-k}z^{r-t}a^{t-k-l}b^lc^{r-j-(t-l-k)}d^{j-l}{t-k \choose l}\left(\sum\limits_{m=0}^k{k \choose m}[r-j]_{t-l-m}[j]_{l+m}\right)\\
			&=z^{r-t}[r-t+k]_{k}\sum\limits_{l=0}^{t-k}{t-k \choose l}a^{t-k-l}b^lc^{r-j-(t-l-k)}d^{j-l}[r-j]_{t-k-l}[j]_l\\
			&=z^{r-t}[r-t+k]_{k}\left(\sum\limits_{l=0}^{t-k}{t-k \choose l}a^{t-k-l}b^l\frac{\partial^{t-k}}{\partial X^{t-k-l}\partial Y^l}\right)(P)\bigg\vert_{(c, d)}\\
			&=z^{r-t}[r-t+k]_{k}\left(a\frac{\partial}{\partial X}+b\frac{\partial}{\partial Y}\right)^{t-k}(P)\bigg\vert_{(c, d)}.
		\end{align*}}
		The second and last but one equalities hold by (\ref{pderivatives}). The fourth equality follows from \eqref{sum}.
	\end{proof}

	\begin{lemma}\label{glinear}
          Let $a,b,c,d,u,v,w,z\in \mathbb{F}_p$ and $P(X, Y)\in V_r.$ Let $P_1 
          =P(U, V),$
          with $U=uX+wY$ and $V=vX+zY.$ Then for $k\geq 0$ we have
		\[\left(a\frac{\partial}{\partial X}+b\frac{\partial}{\partial Y}\right)^k(P_1)\bigg\vert_{(c, d)}=\left((ua+wb)\frac{\partial}{\partial X}+(va+zb)\frac{\partial}{\partial Y}\right)^k(P)\bigg\vert_{(uc+wd,vc+zd)}.\]
	\end{lemma}

	\begin{proof}
          This is just the chain rule.
              \end{proof}

	\begin{lemma}\label{nabla}
          Let $a,b \in \mathbb{F}_p$ and $\nabla= a\frac{\partial}{\partial X}+b\frac{\partial}{\partial Y}.$
          Let $f:=f(X, Y), g:=g(X, Y)\in \mathbb{F}_p[X, Y].$ Then, for all $m \geq 1$, we have
		\[\nabla^m(fg)=\sum\limits_{i=0}^m{m \choose i}\nabla^{m-i}(f)\nabla^i(g).\]
	\end{lemma}
	
	\begin{proof}
	  This is just Leibniz (Leibnitz in Indian calculus texts!) rule.
	\end{proof}

	\begin{lemma}\label{theta}
          Let $a,b,c,d\in \mathbb{F}_p.$ We let $\nabla= a\frac{\partial}{\partial X}+b\frac{\partial}{\partial Y}$
          and $\theta(X, Y)=X^pY-XY^p.$ Then, for $l, k \geq 0$, we have
		\[\nabla^l(\theta(X, Y)^k)\big\vert_{(c,d)}=
		\begin{cases}
			l!\left(\nabla\theta(X, Y)\big\vert_{(c,d)}\right)^l, &\text{if}~k=l,\\
			0, &\text{otherwise}.\\
		\end{cases}\]
	\end{lemma}

	\begin{proof}
          We first show that
		\begin{equation}\label{klarge}
			\nabla^l(\theta^k)=l!{k\choose l}\theta^{k-l}(\nabla\theta)^l,~\forall ~k\geq l.
		\end{equation}
		We prove the result by induction on $l.$ The $l = 0$ case is trivial. Suppose $l=1.$ Then, we have
		\[\nabla(\theta^k)=\left(a\frac{\partial}{\partial X}+b\frac{\partial}{\partial Y}\right)(\theta^k)=ak\theta^{k-1}\frac{\partial\theta}{\partial X}+bk\theta^{k-1}\frac{\partial\theta}{\partial Y}=k\theta^{k-1}(\nabla\theta),\]
	        as desired. Assume that the result is true for $l$. If $k \geq l+1$, then
		\begin{eqnarray*}
	        \nabla^{l+1}(\theta^k) 
			& = & \nabla\left(l!{k\choose l}\theta^{k-l}(\nabla\theta)^l\right) \> = \>  l!{k \choose l}\left(a\frac{\partial}{\partial X}\left(\theta^{k-l}(\nabla\theta)^l\right)+b\frac{\partial}{\partial Y}\left(\theta^{k-l}(\nabla\theta)^l\right)\right)\\
			& = & l!{k\choose l}a\left(\theta^{k-l}\frac{\partial}{\partial X}(\nabla\theta)^l+(k-l)\theta^{k-l-1}\frac{\partial\theta}{\partial X}(\nabla\theta)^l\right)\\
			&  & +l!{k\choose l}b\left(\theta^{k-l}\frac{\partial}{\partial Y}\left(\nabla\theta\right)^l+(k-l)\theta^{k-l-1}\frac{\partial\theta}{\partial Y}(\nabla\theta)^l\right)\\
			& = & l!{k \choose l}\theta^{k-l}\nabla((\nabla\theta)^l)+l!{k \choose l}(k-l)\theta^{k-l-1}(\nabla\theta)^l\left(a\frac{\partial\theta}{\partial X}+b\frac{\partial\theta}{\partial Y}\right)\\
			& = & (l+1)!{k \choose l+1}\theta^{k-(l+1)}(\nabla\theta)^{l+1}.
		\end{eqnarray*}
		The first equality holds by the induction hypothesis. Note that $\nabla\theta=bX^p-aY^p$, so 
		$\nabla((\nabla\theta)^l)=0$, and hence the last equality follows. Thus the identity (\ref{klarge}) follows by induction.
		
		Now, suppose $k<l.$ Then
		\begin{equation}\label{ksmall}
			\nabla^l(\theta^k)=\nabla^{l-k}(\nabla^k(\theta^k))=k!\nabla^{l-k}((\nabla\theta)^k)=k!\nabla^{l-k-1}\nabla((\nabla\theta)^k)=0.
		\end{equation}
		The second equality follows by taking $l=k$ in (\ref{klarge}). The last equality follows because $\nabla((\nabla\theta)^k)=0.$ 
		
		Combining (\ref{klarge}) and (\ref{ksmall}), we have
		\begin{eqnarray*}
			\nabla^l(\theta^k)\big\vert_{(c,d)} & = &
			\begin{cases}
				l!{k \choose l}\theta^{k-l}\big \vert_{(c, d)}\left(\nabla\theta\big\vert_{(c,d)}\right)^l, & \text{if}~k\geq l,\\
				0, &\text{if}~k<l.
			\end{cases}
			\>\> =  \>\>  \begin{cases}
				l!\left(\nabla\theta\big \vert_{(c,d)}\right)^l, & \text{if}~k=l,\\
				0, &\text{otherwise}.  \hspace{1.2cm} \qedhere  
			\end{cases}
		\end{eqnarray*}  
	\end{proof}

        \subsubsection{Non-split case} 
           We prove  Theorem~\ref{ps} (1) from the introduction. 
	\begin{theorem}
		\label{base case}
		Let $0\leq m\leq p-1$ and $r\geq (m+1)(p+1)-1$. If $p\nmid {r \choose m}$, then we have
		\[\frac{V_r}{V_r^{(m+1)}}\simeq{\rm ind}_{B(\mathbb{F}_p)}^{G(\mathbb{F}_p)}\left(V_m\otimes d^{r-m}\right).\]
	\end{theorem}

	\begin{proof}
          \label{base case proof}
          First a word on notation. The space $V_m$ (like $V_r$) is modelled on homogeneous polynomials over ${\mathbb F}_p$ of degree $m$ in two variables, say $X'$ and $Y'$. However,
          to prevent confusion with the variables $X$ and $Y$ used to model $V_r$, we think instead of elements of $V_m$ as tuples $(x_j)_{0\leq j \leq m} \in {\mathbb F}_p^{m+1}$
          with such tuples corresponding to the coefficients of the polynomials in the basis $\{{X'}^{m-j}{Y'^j}\}_{0\leq j \leq m}$.

          Now, define the map
		\[\psi:V_r\rightarrow {\rm ind}_{B(\mathbb{F}_p)}^{G(\mathbb{F}_p)}\left(V_m\otimes d^{r-m}\right)\]   
		by
		$\psi(P(X, Y))=\psi_{P}$ for all $P = P(X, Y)\in V_r,$
		where $\psi_{P}: G(\mathbb{F}_p)\rightarrow V_m\otimes d^{r-m}$ is defined by
		\begin{eqnarray}
                  \label{def of psi}
                  \psi_P\left(\left(\begin{matrix}
			a & b\\
			c & d
                      \end{matrix}\right)\right)=\left(\frac{{m \choose j}}{[r]_{m-j}}\left(a\frac{\partial}{\partial X}+b\frac{\partial}{\partial Y}\right)^{m-j}(P)\bigg\vert_{(c,d)}\right)_{0\leq j \leq m}
                \end{eqnarray}
		for all $\left(\begin{smallmatrix}
			a & b\\
			c & d
                      \end{smallmatrix}\right)\in G(\mathbb{F}_p).$ 
                    We show that $\psi$ induces a $G(\mathbb{F}_p)$-equivariant isomorphism between $\frac{V_r}{V_r^{(m+1)}}$ and
                    ${\rm ind}_{B(\mathbb{F}_p)}^{G(\mathbb{F}_p)}\left(V_m\otimes d^{r-m}\right)$.
                    Note that, by hypothesis the constant is a well-defined
                    non-zero element of $\mathbb{F}_p.$ For some motivation on how $\psi_P$  is constructed for general $m$, it is useful to work out the case
                    $m = 1$, where it is necessary to divide by the constant $r$ in the first coordinate.
		
        Recall that we denote $\nabla = a \frac{\partial}{\partial X} + b\frac{\partial}{\partial Y} $ and $\nabla' = c\frac{\partial}{\partial X}+d\frac{\partial}{\partial Y}$. \\
		
	      \noindent \textbf{$B(\mathbb{F}_p)$-linearity:} We first show that $\psi_P$ is $B(\mathbb{F}_p)$-linear. Let $\gamma=\left(\begin{smallmatrix}
			a & b\\
			c & d
		\end{smallmatrix}\right)\in G(\mathbb{F}_p),~b=\left(\begin{smallmatrix}
			u & v\\
			0 & z
                      \end{smallmatrix}\right)\in B(\mathbb{F}_p)$ and $\underline{x}:=(x_j)_{0\leq j \leq m}\in V_m\otimes d^{r-m}.$
                Then the action of $b$ on $\underline{x}$ is given by 
		\[b\cdot\underline{x} 
                                      =z^{r-m}\cdot\left(\sum\limits_{j=i}^m{j \choose i}u^{m-j}v^{j-i}z^ix_j \right)_{0\leq i \leq m}.\]
                We have
		\begin{eqnarray*}
			 \psi_P(b \cdot \gamma)
			& = & \psi_P\left(\left(\begin{matrix}
				ua+vc & ub+vd\\
				zc & zd
			\end{matrix}\right)\right)
			\>\> = \>\> \left(\frac{{m \choose j}}{[r]_{m-j}}\left( u\nabla+v\nabla' \right)^{m-j}(P)\bigg\vert_{(zc, zd)}\right)_{0\leq j\leq m}\\
			& =  & \left(\frac{{m \choose j}}{[r]_{m-j}}\left({\sum\limits_{k=0}^{m-j}}{m-j\choose k}u^{m-j-k}v^k\nabla^{m-j-k}\nabla'^k\right)(P)\bigg\vert_{(zc, zd)}\right)_{0\leq j\leq m},\\
		\end{eqnarray*}
		which by taking $t=m-j$ in Lemma \ref{evaluate} and by using the fact $[r]_{m-j}=[r]_{m-j-k}[r-(m-j-k)]_k$ equals
		\begin{equation}\label{lhs}
			\left(\sum\limits_{k=0}^{m-j}\frac{{m \choose j}{m-j\choose k}z^{r-(m-j)}}{[r]_{m-j-k}}\left(u^{m-j-k}v^k\nabla^{m-j-k}\right)(P)\bigg\vert_{(c, d)}\right)_{0\leq j\leq m}.
		\end{equation}
		Now,
		\[b\cdot\psi_P(\gamma)
		=\left( \begin{matrix}
			u & v\\
			0 & z
		\end{matrix} \right)\cdot\left(\frac{{m \choose j}}{[r]_{m-j}}\nabla^{m-j}(P)\bigg\vert_{(c, d)}\right)_{0\leq j \leq m} \> = \>
		\left(\sum\limits_{j=i}^m\frac{{j \choose i}{m \choose j}z^{r-(m-i)}}{[r]_{m-j}}u^{m-j}v^{j-i}\nabla^{m-j}(P)\bigg\vert_{(c,d)} \right)_{0\leq i \leq m},\]
		which, by relabeling $j$ as $l$ and $i$ as $j$, and 
		by further replacing $l$ by $j+k$, equals
		\begin{equation}\label{rhs}
			\left(\sum\limits_{k=0}^{m-j}   \frac{{j+k \choose j}{m \choose j+k}z^{r-(m-j)}}{[r]_{m-j-k}}u^{m-j-k}v^{k}\nabla^{m-j-k}(P)\bigg\vert_{(c,d)} \right)_{0\leq j \leq m}.
		\end{equation}
		Observing ${m \choose j}{m-j \choose k}={j+k \choose j}{m \choose j+k}$ and comparing (\ref{lhs}) and (\ref{rhs}), we have $\psi_P(b\cdot \gamma)=b\cdot\psi_P(\gamma).$ So $\psi_P$ is $B(\mathbb{F}_p)$-linear and hence $\psi$ is well defined. \\
		
		\noindent \textbf{$G(\mathbb{F}_p)$-linearity:} Now, we show that $\psi$ is $G(\mathbb{F}_p)$-linear.
		Let $\gamma=\left(\begin{smallmatrix}
			a & b\\
			c & d
		\end{smallmatrix}\right)$, $g=\left(\begin{smallmatrix}
			u & v \\
			w & z
		\end{smallmatrix}\right)\in G(\mathbb{F}_p).$ Then
                $g\cdot P(X, Y) 
                =P(U, V) =:P_1,$
		where $U=uX+wY$ and $V=vX+zY.$ 
		We have
		\begin{eqnarray*}
			 \psi(g\cdot(P(X, Y)))(\gamma)
			& = &\psi\left(P_1\right)\left(\left( \begin{matrix}
				a & b\\
				c & d
			\end{matrix} \right)\right)
			\> = \> \left(\frac{{m \choose j}}{[r]_{m-j}}\left(a\frac{\partial}{\partial X}+b\frac{\partial}{\partial Y}\right)^{m-j}(P_1)\bigg\vert_{(c,d)}\right)_{0\leq j \leq m},
	        \end{eqnarray*}
		which, by Lemma \ref{glinear}, equals	
{\small 		\[ \left(\frac{{m \choose j}}{[r]_{m-j}}\left((ua+wb)\frac{\partial}{\partial X}+(va+zb)\frac{\partial}{\partial Y}\right)^{m-j}(P)\bigg\vert_{(uc+wd,vc+zd)}\right)_{0\leq j \leq m}  
		= \psi(P(X, Y))(\gamma g) \> 
		 = \> (g\cdot \psi(P(X, Y)))(\gamma) \]}
                \!\!\! for all $\gamma\in G(\mathbb{F}_p)$.
		So $\psi(g\cdot P(X, Y))=g\cdot \psi(P(X, Y))$ for all $g\in G(\mathbb{F}_p).$
                Hence $\psi $ is  $G(\mathbb{F}_p)$-linear. \\

		 \noindent \textbf{Kernel:} 
                 Next we show that $\ker\psi=V_r^{(m+1)}$ by induction on $m$. If $m=0$, it is well known that $\ker\psi=V_r^{(1)}$ (e.g.,
                 \cite[Lemma 2.4]{roz14} or Lemma~\ref{qtheta} with $f = 1$).
                 Let $P(X, Y)\in\ker \psi.$
                 By definition of $\psi$, we have $\nabla^{m-j}(P)\big \vert_{(c, d)}=0$ for all $0\leq j\leq m$, 
                 $\left( \begin{smallmatrix}
			  a & b \\
			  c & d
                        \end{smallmatrix}
                \right) \in G(\mathbb{F}_p)$. In particular, this is true for all $1\leq j\leq m.$ So by the induction hypothesis,
                 we have $P(X, Y)\in V_r^{(m)},$ which gives $P(X, Y)=\theta(X, Y)^mQ(X, Y)$ for some $Q(X, Y)\in V_{r-m(p+1)}.$
		Now, taking $j = 0$ and using Lemma~\ref{nabla}, 
		\begin{eqnarray*}
			 0 =  \nabla^m(P)\bigg \vert_{(c, d)} & = &
                                       \left(\sum\limits_{i=0}^m{m \choose i}\nabla^{m-i}(\theta^m)\nabla^i(Q(X, Y))\right)\bigg\vert_{(c, d)}.
                \end{eqnarray*}
                This implies $Q(c, d)=0$ since, by  Lemma \ref{theta}, all terms above die except for the $i = 0$ term.
                Then by the $m=0$ case, we have $\theta \mid Q(X, Y)$, so $P(X, Y) \in V_r^{(m+1)}.$ Thus $\ker \psi\subset V_r^{(m+1)}$.
                On the other hand if $P(X, Y)=\theta^{m+1}Q'(X, Y)$, it is easy to check using Lemmas~\ref{nabla}, \ref{theta}
                that $\nabla^j(P(X, Y))\big\vert_{(c, d)}=0$ for all $0\leq j\leq m$.
                By the definition of $\psi$, we have $P(X, Y)\in\ker \psi$.  Thus $\ker\psi=V_r^{(m+1)}$. \\
                
                \noindent \textbf {Isomorphism: }  We have shown $\psi$ induces a $G(\mathbb{F}_p)$-equivariant injection from
                $\frac{V_r}{V_r^{(m+1)}}$ to  ${\rm ind}_{B(\mathbb{F}_p)}^{G(\mathbb{F}_p)}\left(V_m\otimes d^{r-m}\right)$. \\
                This is an isomorphism since both sides have dimension $(m+1)(p+1)$
                as $r \geq (m+1)(p+1)-1$. 
        \end{proof}
	
	\begin{remark}
                \label{split induced}
                To compare Theorem~\ref{base case} and Theorem~\ref{cuspidal} (to be proved in Section~\ref{section cuspidal}),
                we note that the extra factor of $V_{p-1}$ in the latter cuspidal setting also occurs in the present principal series setting.  
                To see this, we compare the
                induction of $V_{m} \otimes d^{r-m}$ from the Borel subgroup $B({\mathbb F}_p)$ with that
                from the split torus  $S({\mathbb F}_p) := T^{{\mathrm {sp}}}(\mathbb{F}_p) \subset B(\mathbb{F}_p)$ of diagonal matrices.
		We claim that for any representation $V$ of $B({\mathbb F}_p)$, we have
		\[{\rm ind}_{S(\mathbb{F}_p)}^{G(\mathbb{F}_p)} V \big\vert_{S(\mathbb{F}_p)} \simeq {\rm ind}_{B(\mathbb{F}_p)}^{G(\mathbb{F}_p)}V \otimes V_{p-1}.\]
		To see this, first observe that ${\rm ind}_{S(\mathbb{F}_p)}^{B(\mathbb{F}_p)} \mathbbm{1} \simeq V_{p-1}\big\vert_{B(\mathbb{F}_p)}$.
		Indeed, since $S({\mathbb F}_p)$ is a group of order prime to $p$, every representation of $S({\mathbb F}_p)$, in particular the trivial representation
                $\mathbbm 1$, is projective.  
                Since induction of a projective is projective, so is ${\rm ind}_{S(\mathbb{F}_p)}^{B(\mathbb{F}_p)} \mathbbm{1}$.  
                By Frobenius reciprocity $\mathbbm{1} \subseteq {\rm ind}_{S(\mathbb{F}_p)}^{B(\mathbb{F}_p)} \mathbbm{1}$, making the last module
                the projective envelope of $\mathbbm{1}$ (since the smallest dimension of a projective $B({\mathbb F}_p)$-module is $p$ by, e.g., \cite[(2.11)]{glo78}). 
		Also, $V_{p-1}$ is a projective $G({\mathbb F}_p)$-module and restriction of a projective is projective, so 
		$V_{p-1}\big\vert_{B(\mathbb{F}_p)}$ is projective. Since $\mathbbm{1} \subseteq V_{p-1}\big\vert_{B(\mathbb{F}_p)}$ (as the highest monomial),
                by uniqueness of projective envelopes, we have
		${\rm ind}_{S(\mathbb{F}_p)}^{B(\mathbb{F}_p)} \mathbbm{1} \simeq V_{p-1}\big\vert_{B(\mathbb{F}_p)}$. Tensoring with $V$
                yields ${\rm ind}_{S(\mathbb{F}_p)}^{B(\mathbb{F}_p)} V \big\vert_{S(\mathbb{F}_p)} \simeq V \otimes V_{p-1} \big\vert_{B(\mathbb{F}_p)}$.
		Using this, we have $${\rm ind}_{S(\mathbb{F}_p)}^{G(\mathbb{F}_p)} V \big\vert_{S(\mathbb{F}_p)}
		= {\rm ind}_{B(\mathbb{F}_p)}^{G(\mathbb{F}_p)} \left({\rm ind}_{S(\mathbb{F}_p)}^{B(\mathbb{F}_p)} V\big\vert_{S(\mathbb{F}_p)} \right) = 
		{\rm ind}_{B(\mathbb{F}_p)}^{G(\mathbb{F}_p)} \left( V \otimes V_{p-1}\big\vert_{B(\mathbb{F}_p)} \right) 
                = {\rm ind}_{B(\mathbb{F}_p)}^{G(\mathbb{F}_p)} V \otimes V_{p-1},$$
                as claimed. Thus tensoring both sides of the isomorphism in Theorem~\ref{base case} with $V_{p-1}$ and using the above claim with $V = V_m \otimes d^{r-m}$,
                we obtain the
                isomorphism 
		\[\frac{V_r}{V_r^{(m+1)}} \otimes V_{p-1} \simeq{\rm ind}_{S(\mathbb{F}_p)}^{G(\mathbb{F}_p)}\left(V_m\otimes d^{r-m}\right)\]
                for $0\leq m\leq p-1$, $r\geq (m+1)(p+1)-1$, $p\nmid {r \choose m}$. This isomorphism is closer in flavor to that in
                Theorem~\ref{cuspidal}.
		\end{remark}

	\subsubsection{Split case}
        We turn to the case $p\mid\binom{r}{m}$. As explained in the introduction we may assume that $m = 1$.
        In the next theorem, we prove Theorem~\ref{ps} (2). Recall that
        $V_1^{\mathrm{ss}}:=a\oplus d$ denotes the two-dimensional split representation of $B(\mathbb{F}_p).$ 
	\begin{theorem}
          \label{split case}
                Let $r\geq 2p+1$.
		If $p\mid r$, then we have
		\[\dfrac{V_r}{V_r^{(2)}}\simeq {\rm ind}_{B(\mathbb{F}_p)}^{G(\mathbb{F}_p)} \left( V_1^{\mathrm{ss}}\otimes d^{r-1} \right) = \mathrm{ind}_{B(\mathbb{F}_p)}^{G(\mathbb{F}_p)} ad^{r-1} \oplus \mathrm{ind}_{B(\mathbb{F}_p)}^{G(\mathbb{F}_p)}  d^{r}.\]
	\end{theorem}
	
	\begin{proof}
		Define \[\psi^{\mathrm{ss}}: V_r\rightarrow {\rm ind}_{B(\mathbb{F}_p)}^{G(\mathbb{F}_p)}  ( V_1^{\mathrm{ss}}\otimes d^{r-1} )\] by $\psi^{\mathrm{ss}}(P)=\psi^{\mathrm{ss}}_P,$ where $\psi^{\mathrm{ss}}_P:G(\mathbb{F}_p)\rightarrow  V_1^{\mathrm{ss}}\otimes d^{r-1} $ is defined by
		\begin{eqnarray}
                  \label{def of psi split}
                  \psi^{\mathrm{ss}}_P\left(\left(\begin{matrix}
				a & b \\ c & d
                              \end{matrix}\right)\right)=\left(\nabla P\Big\vert_{(c, d)}, \>P(c, d)\right),
                \end{eqnarray}                             
		for $(\begin{smallmatrix}
				a & b \\ c & d
                              \end{smallmatrix}) \in G(\mathbb{F}_p)$, 
                $\nabla = a \frac{\partial}{\partial X} + b\frac{\partial}{\partial Y}$. Note we've dropped the $1/r$ factor in the
                $j = 0$ term of \eqref{def of psi} for $m =1$.
		
                We check the $B(\mathbb{F}_p)$-linearity of $\psi^{\mathrm{ss}}_P.$ 
  Let $b=\left(\begin{smallmatrix}
  	u & v\\ 0 & z
  \end{smallmatrix}\right)\in B(\mathbb{F}_p)$ and $\gamma\in \left(\begin{smallmatrix}
  a & b\\ c & d
\end{smallmatrix}\right)\in G(\mathbb{F}_p).$
Then,
\begin{eqnarray*}
	\label{B-linearity split case}
	\psi_P^{\mathrm{ss}}(b\cdot \gamma)&=& \psi_P^{\mathrm {ss}}\left(\left(\begin{matrix}
		ua+vc & ub+vd\\ zc & zd
	\end{matrix}\right)\right) 
                                           \> = \> \left(u z^{r-1}\nabla(P)\Big\vert_{(c, d)}+v z^{r-1} \nabla'(P)
                                               \Big\vert_{(c, d)}, ~ z^r P(c, d)\right)\\
&=& \left(u z^{r-1}\nabla(P)\Big\vert_{(c, d)}+v z^{r-1}  r P(c, d), ~z^rP(c, d)\right) \> = \>  \left(u z^{r-1}\nabla(P)\Big\vert_{(c, d)}, ~z^rP(c, d)\right) \\
& = &  \left(\begin{matrix}
	u & v\\ 0 & z
\end{matrix}\right)\cdot \left(\nabla(P)\Big\vert_{(c, d)}, ~P(c, d)\right)
\> = \> b\cdot \psi^{\mathrm{ss}}_P(\gamma).
\end{eqnarray*}
Here $\nabla' =  c\frac{\partial}{\partial X}+d\frac{\partial}{\partial Y}$ and the fourth equality follows because $p\mid r.$
Hence $\psi^{\mathrm{ss}}_P$ is $B(\mathbb{F}_p)$-linear.

The proof of the $G(\mathbb{F}_p)$-linearity of $\psi^{\mathrm{ss}}$ and the fact that $\ker\psi^{\mathrm{ss}}=V_r^{(2)}$ follows as in the proof of
                Theorem~\ref{base case}. We conclude as in the proof of Theorem~\ref{base case} by comparing dimensions.
\end{proof}

As discussed in the introduction, we obtain the splitting \eqref{split}:
	\begin{corollary}
		\label{corollary to split case}
		Let $0 \leq m \leq p-1$ and $r\geq (m+1)(p+1)-1$. If $p\mid \binom{r}{m}$, that is, $p\mid r-i$ for some $0\leq i\leq m-1$, then we have
		\[\dfrac{V_r}{V_r^{(m+1)}}\simeq \dfrac{V_r}{V_r^{(i+1)}}\oplus \dfrac{V_r^{(i+1)}}{V_r^{(m+1)}}.\]
	\end{corollary}

	\subsection{The case of ${\rm GL}_2(\mathbb{F}_q)$}\label{general}	
	Let $G(\mathbb{F}_q):={\rm GL}_2(\mathbb{F}_q)$ with $q=p^f$ for $f \geq 1$.
        Let $B(\mathbb{F}_q)$ denote the subgroup of upper triangular matrices of $G(\mathbb{F}_q).$ Let $r\geq 0$ and let $r=r_0+r_1p+\dots+r_{f-1}p^{f-1}$ be the
        $p$-adic expansion of $r$ with $0\leq r_j\leq p-1$ and $0\leq j\leq f-1.$ Let $V_{r_j}^{{\rm Fr}^j}:={\rm Sym}^{r_j}(\mathbb{F}_p^2)\circ {\rm Fr}^{j}$ for
        all $0\leq j\leq f-1,$ where ${\rm Fr}$ denotes the Frobenius morphism. Let  $V_{r_j}^{{\rm Fr}^j}$  be modeled
        on polynomials in $X_j$ and $Y_j$ over ${\mathbb F}_q$
        of degree $r_j$ for all $0\leq j\leq f-1.$ Let
        \begin{eqnarray}
          \label{twisted-Dickson-polys}
          \theta_0:=X_0Y_{f-1}^p-Y_0X_{f-1}^p \quad \text{and} \quad \theta_k:=X_kY_{k-1}^p-Y_kX_{k-1}^p
        \end{eqnarray}
        for all $1\leq k\leq f-1$ denote the twisted Dickson polynomials.

	\begin{lemma}\label{evaluategeneral}
		Let $a, b, c, d, z\in \mathbb{F}_q.$ We write \[\nabla_j = a^{p^j}\frac{\partial}{\partial X_j}+b^{p^j}\frac{\partial}{\partial Y_j} \quad\text{and}\quad \nabla_j'= c^{p^j}\frac{\partial}{\partial X_j}+d^{p^j}\frac{\partial}{\partial Y_j} .\]
		Let $P_j(X_j, Y_j)=\sum\limits_{i_j=0}^{r_j}a_{i_j}X_j^{r_j-i_j}Y^{i_j}\in V_{r_j}^{{\rm Fr}^j}$ with $a_{i_j}\in \mathbb{F}_q$ for all $0\leq i_j \leq r_j.$ 
		Then, for all $0\leq k_j \leq t_j,$ we have
		\[\nabla_j^{t_j-k_j} \nabla_j'^{k_j}(P_j)\bigg\vert_{\left((zc)^{p^j}, (zd)^{p^j}\right)}=z^{(r_j-t_j)p^j}[r_j-t_j+k_j]_{k_j}\nabla_j^{t_j-k_j}(P_j)\bigg\vert_{(c^{p^j},d^{p^j})}.\]
	\end{lemma}
	
	\begin{proof}
		Similar to Lemma \ref{evaluate}.
	\end{proof}
	
	\begin{lemma}\label{glineargeneral}
		Let $a,b,c,d,u,v,w,z\in \mathbb{F}_q$ and $P_j(X_j, Y_j)\in V_{r_j}^{{\rm Fr}^j}.$ Let $P_j':=P_j(U_j, V_j),$ where $U_j=u^{p^j}X_j+w^{p^j}Y_j$ and $V_j=v^{p^j}X_j+z^{p^j}Y_j.$ Then for $k_j\geq 0,$ we have
		
		\begin{eqnarray*}
			\left(a^{p^j}\frac{\partial}{\partial X_j}+b^{p^j}\frac{\partial}{\partial Y_j}\right)^{k_j}(P_j')\bigg\vert_{\left(c^{p^j}, d^{p^j}\right)}
			&=& \left((ua+wb)^{p^j}\frac{\partial}{\partial X_j}+(va+zb)^{p^j}\frac{\partial}{\partial Y_j}\right)^{k_j}(P_j)\bigg\vert_{\left((uc+wd)^{p^j},(vc+zd)^{p^j}\right)}.
		\end{eqnarray*}
	
	\end{lemma}
	
	\begin{proof}
		Similar to Lemma \ref{glinear}.
	\end{proof}

	\begin{lemma}\label{qtheta}
		Let $\theta'=X^qY-Y^qX.$ Let $P(X, Y)=\sum\limits_{i=0}^ra_iX^{r-i}Y^i\in\mathbb{F}_q[X, Y]$ be such that $P(c, d)=0$ for all $\left(\begin{matrix}
			a & b \\
			c & d
		\end{matrix}\right)\in G(\mathbb{F}_q).$ Then $P(X, Y)\in\langle \theta' \rangle.$
	\end{lemma}

	\begin{proof}
          Exercise.
	\end{proof}

	\begin{lemma}\label{qthetaj}
		Let $a,b,c,d\in \mathbb{F}_q.$ We let $\nabla_j=\left(a^{p^j}\frac{\partial}{\partial X_j}+b^{p^j}\frac{\partial}{\partial Y_j}\right)$ and $\theta_j'=X_j^qY_j-X_jY_j^q.$ Then we have
		\[\nabla_j^l(\theta_j'^k)\bigg\vert_{(c^{p^j},d^{p^j})}=
		\begin{cases}
			l!\left(\nabla_j\theta_j'\bigg\vert_{(c^{p^j},d^{p^j})}\right)^l, &\text{if}~k=l,\\
			0, &\text{otherwise}.\\
		\end{cases}\]
	\end{lemma}
	
	\begin{proof}
		Similar to Lemma \ref{theta}.
	\end{proof}

	\begin{lemma}\label{jthcomponent}
		 For $0\leq j\leq f-1$, $0\leq m_j\leq p-1$ and $r_j\geq (m_j+1)(q+1)-1$, let \[ \psi^j:V_{r_j}^{{\rm Fr}^j}\rightarrow{\rm ind}_{B(\mathbb{F}_q)}^{G(\mathbb{F}_q)} \left( V_{m_j}^{{\rm Fr}^j}\otimes d^{(r_j-m_j)p^j} \right) \] 
		 be defined by $\psi^j(P_j(X_j, Y_j))=\psi^j_{P_j(X_j, Y_j)},$ where $\psi^j_{P_j(X_j, Y_j)}:G(\mathbb{F}_q)\rightarrow V_{m_j}^{{\rm Fr}^j}\otimes d^{(r_j-m_j)p^j}$ is given by
		\[\psi^j_{P_j(X_j, Y_j)}\left(\left(\begin{matrix}
			a & b\\
			c & d\\
		\end{matrix}\right)\right)=\left(\frac{{m_j\choose n_j}}{[r_j]_{m_j-n_j}}\nabla_j^{m_j-n_j}(P_j)\bigg\vert_{(c^{p^j}, d^{p^j})}\right)_{0\leq n_j\leq m_j},\]
		for all $\left( \begin{matrix}
			a & b\\
			c & d
		\end{matrix} \right)\in G(\mathbb{F}_q)$ and 
	$\nabla_j = a^{p^j}\frac{\partial}{\partial X_j}+b^{p^j}\frac{\partial}{\partial Y_j}.$ Then  
		\begin{enumerate}
			\item[$(i)$] $\psi_{P_j(X_j, Y_j)}^j$ is $B(\mathbb{F}_q)$-linear,
			\item[$(ii)$]  $\psi^j$ is $G(\mathbb{F}_q)$-linear, 
			\item[$(iii)$]  $\psi^j$ induces an isomorphism from $\frac{V_{r_j}^{\rm{Fr}^j}}{\langle \theta_j'^{m_j+1}\rangle}$ to
                                        ${\rm ind}_{B(\mathbb{F}_q)}^{G(\mathbb{F}_q)} \left( V_{m_j}^{{\rm Fr}^j}\otimes d^{(r_j-m_j)p^j} \right)$.
		\end{enumerate}
	\end{lemma}

	\begin{proof}\label{jthcomponentproof}
		$(i)$ \textbf{$B(\mathbb{F}_q)$-linearity:} Let $b=\left(\begin{smallmatrix}
			u & v\\
			0 & z
		\end{smallmatrix}\right)\in B(\mathbb{F}_q)$  and $\gamma=\left(\begin{smallmatrix}
			a & b\\
			c & d\\
                      \end{smallmatrix}\right)\in G(\mathbb{F}_q).$ Let $\underline{x}':=(x_{n_j}')_{0\leq n_j \leq m_j}\in V_{m_j}^{{\rm Fr}^j}\otimes d^{(r_j-m_j)p^j}.$ Then
                    the action of $b$ on $\underline{x}'$ is given by
		\[ b\cdot\underline{x}'=z^{(r_j-m_j)p^j}\cdot\left(\sum\limits_{n_j=i_j}^{m_j}{n_j \choose i_j}u^{(m_j-n_j)p^j}v^{(n_j-i_j)p^j}z^{i_jp^j}x_{n_j}' \right)_{0\leq i_j\leq m_j}. \]
                As in the proof of Theorem~\ref{base case}, but using Lemma~\ref{evaluategeneral} instead, 
                $\psi^j_{P_j}(b\cdot \gamma)   = \psi^j_{P_j}(b\cdot \gamma)$. 
   Thus $\psi^j_{P_j}$ is $B(\mathbb{F}_q)$-linear. 
		
   $(ii)$ \textbf{$G(\mathbb{F}_q)$-linearity:} This follows as in the proof of Theorem~\ref{base case}
   using Lemma~\ref{glineargeneral} instead.

		$(iii)$ \textbf{Isomorphism: } We now show by induction on $m_j$ that $\ker\psi^j=\langle \theta_j'^{(m_j+1)}\rangle,$ where $\theta_j'=X_j^qY_j-X_jY_j^q.$ Suppose $m_j=0.$ Then $\psi^j(P_j)=\psi^j_{P_j}$ is defined by $\psi^j_{P_j}\left( \left( \begin{smallmatrix}
			a & b\\
			c & d
		\end{smallmatrix} \right) \right)=P_j(c^{p^j}, d^{p^j})$. Clearly $\langle \theta_j' \rangle\subset \ker\psi^j.$ Let $P_j\in \ker\psi^j.$ Then $P_j(c^{p^j}, d^{p^j})=0$ for all $(c, d)\in \mathbb{F}_q\times\mathbb{F}_q\setminus (0, 0).$ Thus, $P_j(c, d)=0$ for all $(c, d)\in \mathbb{F}_q\times\mathbb{F}_q\setminus (0, 0).$ Then by Lemma \ref{qtheta}, we have $P_j\in\langle \theta_j' \rangle.$ So $\ker\psi^j\subset \langle \theta_j' \rangle.$ Thus $\ker\psi^j=\langle \theta_j' \rangle.$
		
		Assume that the result is true for $m_j-1.$ By Lemma \ref{qthetaj}, we have \[\nabla_j^{a_j}(\theta_j'^{m_j+1})\big\vert_{(c^{p^j}, d^{p^j})}=0\] for all 
		$0\leq a_j\leq m_j.$ By definition of $\psi^j,$ we have $\theta_j'^{m_j+1}\in \ker\psi^j.$ So $\langle \theta_j'^{m_j+1} \rangle\subset \ker\psi^j.$ Let $P_j\in \ker\psi^j.$
		Then $\nabla_j^{a_j}(P_j)\big\vert_{(c^{p^j}, d^{p^j})}=0$ for all $0\leq a_j\leq m_j.$ In particular, $\nabla_j^{a_j}(P_j)\big\vert_{(c^{p^j}, d^{p^j})}=0$
                for all $0\leq a_j\leq m_j-1.$ By the induction hypothesis $P_j\in \langle \theta_j'^{m_j} \rangle.$ Write $P_j=Q_j\theta_j'^{m_j}$ with $Q_j\in\mathbb{F}_q[X_j, Y_j].$
                Now, taking $a_j = m_j$ and proceeding exactly as in the proof of Theorem~\ref{base case}, but
                using Lemma~\ref{qthetaj} instead,
                we see $\theta_j'$ divides $Q_j.$ 
                So $\ker\psi^j\subset \langle \theta_j'^{m_j+1} \rangle.$ Thus $\ker\psi^j= \langle \theta_j'^{m_j+1} \rangle.$
                
                So $\psi^j$ induces an injective map  $\frac{V_{r_j}^{\rm{Fr}^j}}{\langle \theta_j'^{m_j+1}\rangle} \rightarrow{\rm ind}_{B(\mathbb{F}_q)}^{G(\mathbb{F}_q)} \left( V_{m_j}^{{\rm Fr}^j}\otimes d^{(r_j-m_j)p^j} \right)$. Note that the dimension of both sides is $(m_j+1)(q+1)$
                as $r_j \geq (m_j+1)(q+1)-1$. So this map is also surjective and an isomorphism.
		%
	\end{proof}

        \begin{remark}
          \label{psij-surjective}
        It is also possible to give a direct proof of the surjectivity of $\psi^j$, at least when $m_j = 0$.
        Let $\sigma$ be a representation of $B(\mathbb{F}_q).$ For $g\in G(\mathbb{F}_q)$, $v\in \sigma,$ let $[g, v]\in {\rm ind}_{B(\mathbb{F}_q)}^{G(\mathbb{F}_q)}\sigma$
        denote the map defined by
	\[[g, v](g')=
	\begin{cases}
		\sigma(g'g)v,  & \text{if}~ g'g\in B(\mathbb{F}_q),\\
		0,             & \text{otherwise}.\\
	\end{cases}
	\]
        A basis of ${\rm ind}_{B(\mathbb{F}_q)}^{G(\mathbb{F}_q)}d^{(r_j-m_j)p^j}$ is given by (cf.  \cite[Lemma 7.2]{bre07}, \cite[Lemma 2.5 (2)]{bp12} )
		\begin{equation}\label{basis}
                  \left\{f_i:=\sum\limits_{\lambda\in\mathbb{F}_q}\lambda^i
			\left(\begin{matrix}
				\lambda & 1\\
				1 & 0
			\end{matrix}\right)[1, 1],~~\phi:=[1, 1]\right\}_{0\leq i\leq q-1},
		\end{equation}
		where $[1, 1]$ denotes the function supported on $B(\mathbb{F}_q)$ and $[1, 1](u)=1$
                for all $u\in \left\{\left(\begin{smallmatrix}
			1 & x\\
			0 & 1
                      \end{smallmatrix}\right)\mid x\in \mathbb{F}_q\right\}.$
                  When, e.g., $j = 0$ and $m_0 = 0$, then one may check
                  $\psi^0$ maps $(-1)^iX_0^{r_0-i}Y_0^i$ to $f_i$ for $0 \leq i \leq q-1$, and maps 
                   $Y^{r_0} - X_0^{q-1}Y_0^{r_0-(q-1)}$ to $\phi$.

        Similarly for $j$, $m_j$ arbitrary, a basis of
         $ V_{m_j} \otimes {\rm ind}_{B(\mathbb{F}_q)}^{G(\mathbb{F}_q)}d^{(r_j-m_j)p^j} \simeq {\rm ind}_{B(\mathbb{F}_q)}^{G(\mathbb{F}_q)}(V_{m_j} \otimes d^{(r_j-m_j)p^j})$
        is given by
		\[\left\{S_j^{m_j-l}T_j^l\otimes f_i,~ S_j^{m_j-l}T_j^l\otimes\phi \> \big | \> 0\leq l\leq m_j, \> 0\leq i\leq q-1\right\},\]
		where $V_{m_j}$ is modeled on polynomials of degree $m_j$ over ${\mathbb F}_q$ in $S_j$, $T_j$ and $f_i$, $\phi$ are in 
                ${\rm ind}_{B(\mathbb{F}_q)}^{G(\mathbb{F}_q)}d^{(r_j-m_j)p^j}$ as in (\ref{basis}). 
                One should similarly be able to write down polynomials mapping to these basis elements under $\psi^j$.
        \end{remark}

	Theorem~\ref{ps-twisted} is a twisted version of the Lemma~\ref{jthcomponent}. However the proof is more involved. To prove it we need a few more lemmas.
        Recall the twisted Dickson polynomials $\theta_j$ for $0 \leq j \leq f-1$ were defined in \eqref{twisted-Dickson-polys}.
        
	\begin{lemma}\label{thetaj}
		Let $a,b,c,d\in\mathbb{F}_q.$ For $0\leq j\leq f-1$ and $0\leq l_j,k_j\leq m_j,$ let $\nabla_j = a^{p^j}\frac{\partial}{\partial X_j}+b^{p^j}\frac{\partial}{\partial Y_j}.$ 
		Then 
{\small		\begin{eqnarray*}
			\left(\prod_{j=0}^{f-1}\nabla_j^{l_j}\right)\left(\prod_{j=0}^{f-1}\theta_j^{k_j}\right)\bigg\vert_{(c, d;\dots;c^{p^{f-1}},d^{p^{f-1}})}
			=  \begin{cases}
				l_0!\cdots l_{f-1}!\prod_{j=0}^{f-1}\left(\nabla_j(\theta_j)\bigg\vert_{(c, d;\dots;c^{p^{f-1}},d^{p^{f-1}})}\right)^{l_j}, & if~ (l_0,\dots,l_{f-1})=(k_0,\dots,k_{f-1}),\\
				0, & otherwise.
			\end{cases}
		\end{eqnarray*}}
	\end{lemma}
	
	\begin{proof}
		We induct on $f.$ Lemma \ref{theta} is the case $f = 1$. Assume the result for $f-1.$ Now, consider
{\small		\begin{align*}
			&\left(\prod_{j=0}^{f-1}\nabla_j^{l_j}\right)\left(\prod_{j=0}^{f-1}\theta_j^{k_j}\right)\bigg\vert_{(c, d;\dots;c^{p^{f-1}},d^{p^{f-1}})}\\
			&=\left(\prod_{j=1}^{f-1}\nabla_j^{l_j}\right)\left(\left(\prod_{j=2}^{f-1}\theta_j^{k_j}\right)\theta_1^{k_1}\nabla_{0}^{l_{0}}\left(\theta_{0}^{k_{0}}\right)\right)\bigg\vert_{(c, d;\dots;c^{p^{f-1}},d^{p^{f-1}})}\\
			&=\left(\prod_{j=1}^{f-1}\nabla_j^{l_j}\right)\left(\prod_{j=1}^{f-1}\theta_j^{k_j}\right)\bigg \vert_{(c, d;\dots;c^{p^{f-1}},d^{p^{f-1}})}\nabla_{0}^{l_{0}}\left(\theta_{0}^{k_{0}}\right)\bigg \vert_{(c,d;\dots;c^{p^{f-1}},d^{p^{f-1}})}\\
			&=\begin{cases}
				l_0!\cdots l_{f-1}!\prod_{j=0}^{f-1}\left(\nabla_j(\theta_j)\bigg \vert_{(c, d;\dots;c^{p^{f-1}},d^{p^{f-1}})}\right)^{l_j}, & if~ (l_0,\dots,l_{f-1})=(k_0,\dots,k_{f-1}),\\
				0, & otherwise.
			\end{cases}
		\end{align*}}
\!\!\!		The first equality follows since $\theta_j$ for $2\leq j\leq f-1$ is independent of $X_0$, $Y_0$, and
                since $\nabla_0(\theta_1)=0$, and
                the last from the induction hypothesis and a twisted analogue of the $j = 0$ version of Lemma~\ref{qthetaj}.
	\end{proof}

	For $0\leq j\leq f-1,$ assume that $r_j\geq p^{f-j}$. Then $r=\sum\limits_{j=0}^{f-1}r_jp^j\geq fq.$ Note that $r$ does not determine the $r_j$
        and the $r_j$ are not the digits of the base $p$ expansion of $r$.
        Let $0\leq i_j\leq r_j$  for $0\leq j\leq f-1$. Set $i = \sum\limits_{j=0}^{f-1}i_jp^j$ and  $\vec{i}=(i_0,\dots,i_{f-1}).$ We write 
	$\sum\limits_{\vec{i}=\vec{0}}^{\vec{r}}:=\sum\limits_{i_0=0}^{r_0}\dots\sum\limits_{i_{f-1}=0}^{r_{f-1}}.$
	
	\begin{lemma}\label{monomial difference}
          For $0\leq j\leq f-1$, let $0 \leq i_j \leq r_j$ 
          and $0 \leq k_j\leq p-1$. Let $k=\sum\limits^{f-1}_{j=0}k_jp^j$.
          Let $P=\sum\limits_{\vec{i}=\vec{0}}^{\vec{r}}a_{\vec{i}}\prod_{j=0}^{f-1} X_j^{r_j-i_j}Y_j^{i_j}\in\mathbb{F}_q[X_0,Y_0,\dots,X_{f-1},Y_{f-1}]$ be such that $P\left(c,d;\dots;c^{p^{f-1}},d^{p^{f-1}}\right)=0$
          for all  $\left( \begin{smallmatrix}
			a & b\\
			c & d\\
		\end{smallmatrix} \right)\in G(\mathbb{F}_q)$.
		Then 
		$a_{\vec{0}}=0 = a_{\vec{r}},$
		and the polynomial $P$ is of the form
{\small		\[P=\sum\limits_{k=1}^{q-1}
		\left(\sum\limits_{\substack{\vec{i}=\vec{0}\\ \vec{i}\neq \vec{k}\\i\equiv k\mod(q-1)}}^{\vec{r}}a_{\vec{i}}\left(\prod_{j=0}^{f-1}X_j^{r_j-i_j}Y_j^{i_j}-\prod_{j=0}^{f-1}X_j^{r_j-k_j}Y_j^{k_j}\right)\right).\]}	
	\end{lemma}
	
	\begin{proof}
		From the given condition we have $P\left(c,d;\dots;c^{p^{f-1}},d^{p^{f-1}}\right)=\sum\limits_{\vec{i}=\vec{0}}^{\vec{r}}a_{\vec{i}}c^{r-i}d^i=0$, $\text{for all}~(c, d)\in\mathbb{F}_q\times\mathbb{F}_q\setminus\{(0,0)\}$.
		In particular, choosing $(c, d)=(1, 0),$ we have $a_{\vec{0}}=0$ and choosing $(c, d)=(0, 1),$ we have $a_{\vec{r}}=0.$ Let $\mathbb{F}_q^{\times}=\{\lambda_l\mid1\leq l\leq q-1 \}.$
		Then taking $(c, d)=(1, \lambda_l),$ we have
{\small		\[P(1,\lambda_l;\dots;1,\lambda_l^{p^{f-1}})=\sum\limits_{k=1}^{q-1}\left(\left(\sum\limits_{\substack{\vec{i}=\vec{0}\\i\equiv k\mod (q-1)}}^{\vec{r}}a_{\vec{i}}\right)\lambda_l^k\right)=0,\]}
		which, by writing $A_k= \sum\limits_{\substack{\vec{i}=\vec{0}\\i\equiv k\mod (q-1)}}^{\vec{r}}a_{\vec{i}}$,
		gives $\sum\limits_{k=1}^{q-1}A_{k}\lambda_l^{k}=0.$
                Since the (essentially) Vandermonde matrix $(\lambda_l^k)$ 
                is invertible, we have
		\begin{equation}\label{ak}
                  A_k
                  =0
                      \end{equation}
                for all $1\leq k\leq q-1$ (see also \cite[Lemme 3.1.6]{bre03}, \cite[Lemma 2.8]{hen19}).
		Now, we have
{\small		\begin{eqnarray*}
			P & = & \sum\limits_{\vec{i}=\vec{0}}^{\vec{r}}a_{\vec{i}}\prod_{j=0}^{f-1} X_j^{r_j-i_j}Y_j^{i_j} \>
			= \> \sum\limits_{k=1}^{q-1}\left(\sum\limits_{\substack{\vec{i}=\vec{0}\\i\equiv k\mod (q-1)}}^{\vec{r}}a_{\vec{i}}\prod_{j=0}^{f-1} X_j^{r_j-i_j}Y_j^{i_j}\right)\\
			&= & \sum\limits_{k=1}^{q-1}\left(\left(\sum\limits_{\substack{\vec{i}=\vec{0}\\\vec{i}\neq \vec{k}\\i\equiv k\mod (q-1)}}^{\vec{r}}a_{\vec{i}}\prod_{j=0}^{f-1} X_j^{r_j-i_j}Y_j^{i_j}\right)+a_{\vec{k}}\prod_{j=0}^{f-1} X_j^{r_j-k_j}Y_j^{k_j}\right)
			 \> = \> \sum\limits_{k=1}^{q-1}\left(\sum\limits_{\substack{\vec{i}=\vec{0}\\\vec{i}\neq \vec{k}\\i\equiv k\mod (q-1)}}^{\vec{r}}a_{\vec{i}}\left(\prod_{j=0}^{f-1} X_j^{r_j-i_j}Y_j^{i_j}-\prod_{j=0}^{f-1} X_j^{r_j-k_j}Y_j^{k_j}\right)\right).
		\end{eqnarray*}}
		\!\!\! The last equality follows from (\ref{ak}).
	\end{proof}

	\begin{lemma}\label{unequal}
		Let $\theta_0:=X_0Y_{f-1}^p-Y_0X_{f-1}^p$ and $\theta_k:=X_kY_{k-1}^p-Y_kX_{k-1}^p$ for all $1\leq k\leq f-1.$ For $0\leq j\leq f-1,$ let $c_j,d_j,g_j,h_j\in\mathbb{N}\cup \{0\}$ be such that $c_j+d_j=r_j$ with $r_j\geq p^{f-j}.$
		Let \[P(X_0,Y_0;\dots;X_{f-1},Y_{f-1})=\prod_{j=0}^{f-1}X_j^{c_j}Y_j^{d_j}-\prod_{j=0}^{f-1}X_j^{g_j}Y_j^{h_j},\] 
		where
		$\sum\limits_{j=0}^{f-1}d_jp^j=\sum\limits_{j=0}^{f-1}h_jp^j+t(q-1)$ for some $t\geq 1.$ 
		Assume that $P$ has no pure term which involves either only $X_j$ or only $Y_j.$
		Then, for $0\leq j\leq f-1$ there exist $c_j'$, $d_j'\in\mathbb{N}\cup \{0\}$ such that 
		\[P(X_0,Y_0;\dots;X_{f-1},Y_{f-1})=\left(\prod_{j=0}^{f-1}X_j^{c_j'}Y_j^{d_j'}-\prod_{j=0}^{f-1}X_j^{g_j}Y_j^{h_j}\right)\mod\langle \theta_{0},\dots,\theta_{f-1}\rangle\] with $\sum\limits_{j=0}^{f-1}d_j'p^j=\sum\limits_{j=0}^{f-1}h_jp^j.$
	\end{lemma}

	\begin{proof}
		We first make two observations. 
		First, if $c_0\geq 1, d_{f-1}\geq p,$ then we may write
		\[\prod_{j=0}^{f-1}X_j^{c_j}Y_j^{d_j}\\
		=\left(\prod_{j=1}^{f-2}X_j^{c_j}Y_j^{d_j}\right)\left(
		X_{0}^{c_0-1}Y_0^{d_0}X_{f-1}^{c_{f-1}}Y_{f-1}^{d_{f-1}-p}(X_0Y_{f-1}^p-Y_0X_{f-1}^p)+X_{0}^{c_0-1}Y_0^{d_0+1}X_{f-1}^{c_{f-1}+p}Y_{f-1}^{d_{f-1}-p}\right)\]
		which implies that
		\begin{equation}\label{mod-theta-o}
			\prod_{j=0}^{f-1}X_j^{c_j}Y_j^{d_j}\\
			=\left(\prod_{j=1}^{f-2}X_j^{c_j}Y_j^{d_j}\right)
			X_{0}^{c_0-1}Y_0^{d_0+1}X_{f-1}^{c_{f-1}+p}Y_{f-1}^{d_{f-1}-p}\mod\theta_{0}.\\
		\end{equation}
		%
                Second, for $1\leq k\leq f-1,$ if $c_k\geq 1, d_{k-1}\geq p,$ then similarly 
		\begin{equation}\label{mod-theta-j-unequal}
			\prod_{j=0}^{f-1}X_j^{c_j}Y_j^{d_j}\\
			=\left(\prod_{\substack{j=0\\j\neq k,k-1}}^{f-1}X_j^{c_j}Y_j^{d_j}\right)
			X_{k-1}^{c_{k-1}+p}Y_{k-1}^{d_{k-1}-p}X_k^{c_k-1}Y_k^{d_k+1}\mod \theta_k.
		\end{equation}     
		Note that the first operation involving $\theta_0$ decreases $\sum_{j=0}^{f-1}d_jp^j$ in \eqref{mod-theta-o} by $q-1$.
                However, the second operation involving $\theta_k$ does not change $\sum_{j=0}^{f-1}d_jp^j$ in \eqref{mod-theta-j-unequal}.

		Now, to prove the result we may assume that $d_{f-1}\geq h_{f-1}$ (replacing $P$ by -$P$). We also assume that $t=1$. If $t>1$,
                we iterate the argument below to reduce to the case $t = 1$. 
		
		\textbf{Case 1:} Suppose $d_{f-1}\geq p.$ 
		
		First assume $c_0\geq 1$. 
                Since $c_0\geq 1$ and $d_{f-1}\geq p,$ by (\ref{mod-theta-o}), we have
		\[\prod_{j=0}^{f-1}X_j^{c_j}Y_j^{d_j}=\left(\prod_{j=1}^{f-2}X_j^{c_j}Y_j^{d_j}\right)X_0^{c_0-1}Y_0^{d_0+1}X_{f-1}^{c_{f-1}+p}Y_{f-1}^{d_{f-1}-p}\mod \theta_0.\]
		Note that, \[(d_0+1)+\sum\limits_{j=1}^{f-2}d_jp^j+(d_{f-1}-p)p^{f-1}=\sum\limits_{j=0}^{f-1}d_jp^j-(q-1)=\sum\limits_{j=0}^{f-1}h_jp^j.\]
		Thus the result follows by taking $c_0'=c_0-1,d_0'=d_0+1;c_{f-1}'=c_{f-1}+p, d_{f-1}'=d_{f-1}-p$ and $c_k'=c_k, d_k'=d_k$ for all $1\leq k\leq f-2.$ 
		
		Next suppose $c_0=0.$ Since there is no pure term in $Y_j$ in $P,$ we choose $k$ to be the least index for which $c_j\neq 0$ in $\prod_{j=0}^{f-1}X_j^{c_j}Y_j^{d_j}.$ Then we have $c_j=0$ for all $j<k$ and $c_k\neq 0.$ This implies that $d_j=r_j\geq p$ for all $j<k$ and $c_k\geq 1.$ 
		Then, using (\ref{mod-theta-j-unequal}),  we have
{\small		\begin{eqnarray*}
			\prod_{j=0}^{f-1}X_j^{c_j}Y_j^{d_j}
			& = &  \left(\prod_{\substack{j=0\\j\neq k,k-1}}^{f-1}X_j^{c_j}Y_j^{d_j}\right)X_{k-1}^{c_{k-1}+p}Y_{k-1}^{d_{k-1}-p}X_k^{c_k-1}Y_k^{d_k+1} \mod \theta_k\\
			& = &  \left(\prod_{\substack{j=0\\j\neq k,k-1,k-2}}^{f-1}X_j^{c_j}Y_j^{d_j}\right)X_k^{c_k-1}Y_k^{d_k+1}X_{k-1}^{c_{k-1}+p-1}Y_{k-1}^{d_{k-1}-p+1}X_{k-2}^{c_{k-2}+p}Y_{k-2}^{d_{k-2}-p} \mod \theta_{k-1}\\
			&& \quad\quad\quad\quad\quad\quad\quad\quad\quad\quad\quad\quad\quad\quad\quad\quad\quad\quad\quad\vdots\\
                       & =& \left(\prod_{j=k+1}^{f-1}X_j^{c_j}Y_j^{d_j}\right)X_k^{c_k-1}Y_k^{d_k+1}\left(\prod_{j=1}^{k-1}X_j^{c_j+p-1}Y_j^{d_j-p+1}\right) X_{0}^{c_{0}+p}Y_{0}^{d_0-p} \mod \theta_{1}.
		\end{eqnarray*}}
	
		Clearly $c_0+p\geq 1$ and we are reduced to the previous paragraph ($c_0 \geq 1$). 
		
		\textbf{Case 2:} Suppose $d_{f-1}<p.$
	
		Then $d_{f-1}=p-i_1=i_0p-i_1,$ where $i_0=1$ and $0<i_1\leq p.$ Since $r_{f-1}\geq p,$ we have $c_{f-1}=r_{f-1}-i_0p+i_1\geq i_1.$ Now, if $d_{f-2}\geq i_1p,$ taking $k=f-1,$ by (\ref{mod-theta-j-unequal}) we have
		\[\prod_{j=0}^{f-1}X_j^{c_j}Y_j^{d_j}= \left(\prod_{j=0}^{f-3}X_j^{c_j}Y_j^{d_j}\right)X_{f-2}^{c_{f-2}+i_1p}Y_{f-2}^{d_{f-2}-i_1p}X_{f-1}^{c_{f-1}-i_1}Y_{f-1}^{d_{f-1}+i_1}\mod\theta_{f-1}^{i_1}.\]
		Note that $d_{f-1}+i_1=p,$ so we are done by Case $1.$ 
		
		Suppose, $d_{f-2}<i_1p.$ We write $d_{f-2}=i_1p-i_2$ with $0<i_2\leq i_1p.$ Since by assumption $r_{f-2}\geq p^2$ and $0<i_1\leq p,$ we have $r_{f-2}\geq p^2\geq i_1p.$ This shows that $c_{f-2}=r_{f-2}-i_1p+i_2\geq i_2.$ Now, if $d_{f-3}\geq i_2p,$ taking $k=f-2,$ by (\ref{mod-theta-j-unequal}) we have 
		\[ \prod_{j=0}^{f-1}X_j^{c_j}Y_j^{d_j}= \left(\prod_{\substack{j=0\\j\neq f-2, f-3}}^{f-1}X_j^{c_j}Y_j^{d_j}\right)X_{f-3}^{c_{f-3}+i_2p}Y_{f-3}^{d_{f-3}-i_2p}X_{f-2}^{c_{f-2}-i_2}Y_{f-2}^{d_{f-2}+i_2}\mod\theta_{f-2}^{i_2}.\]
		Note that $d_{f-2}+i_2=i_1p.$ So we are done by the `$d_{f-2}\geq i_1p$' case above. And so on.

                Thus it is enough to show that this process stops. Suppose not. Then for $0\leq j\leq f-1,$ we have $0<i_{j+1}\leq i_jp$ such that $d_{f-1-j}=i_jp-i_{j+1}.$ In particular, $d_0<i_{f-1}p.$ Now, note that
		\[\sum\limits_{j=0}^{f-1}d_jp^j<i_{f-1}p+\sum\limits_{j=0}^{f-2}d_{f-1-j}p^{f-1-j}=i_{f-1}p+\sum\limits_{j=0}^{f-2}(i_jp-i_{j+1})p^{f-1-j} = i_0q. \]
		Since $i_0=1,$ we have $\sum\limits_{j=0}^{f-1}d_jp^j<q.$ Also, note that  $\sum\limits_{j=0}^{f-1}d_jp^j\geq q-1$ by assumption. Hence,
		\[q-1\leq \sum\limits_{j=0}^{f-1}d_jp^j\leq q-1\implies \sum\limits_{j=0}^{f-1}d_jp^j=q-1\implies \sum\limits_{j=0}^{f-1}h_jp^j=0\implies h_j=0,\]
		for all $0\leq j\leq f-1.$
		This is a contradiction because $P$ does not contain any pure term in $X_j.$ Thus modulo $\langle \theta_1,\dots,\theta_{f-1}\rangle,$ we can always assume that $d_{f-1}\geq p,$ hence by Case $1$ the result follows.
	\end{proof}

	\begin{lemma}\label{equal}
          Let 
          $\theta_k:=X_kY_{k-1}^p-Y_kX_{k-1}^p$ for all $1\leq k\leq f-1.$ For $0\leq j\leq f-1,$ let $c_j,d_j,g_j,h_j\in\mathbb{N}\cup \{0\}$ be such that $c_j+d_j=r_j$ with $r_j\geq p^{f-j}.$ Let \[P(X_0,Y_0;\dots;X_{f-1},Y_{f-1})=\prod_{j=0}^{f-1}X_j^{c_j}Y_j^{d_j}-\prod_{j=0}^{f-1}X_j^{g_j}Y_j^{h_j},\] where $\sum\limits_{j=0}^{f-1}d_jp^j=\sum\limits_{j=0}^{f-1}h_jp^j.$ Then $P(X_0,Y_0;\dots;X_{f-1},Y_{f-1})\in\langle \theta_{1},\dots,\theta_{f-1}\rangle.$
	\end{lemma}
	
	\begin{proof}
          We make the following observation.
          For $1\leq k\leq f-1,$ if $d_k\geq 1, c_{k-1}\geq p,$ then
		\begin{equation}\label{mod-theta-j-equal}
			\prod_{j=0}^{f-1}X_j^{c_j}Y_j^{d_j}\\
			=\left(\prod_{\substack{j=0\\j\neq k,k-1}}^{f-1}X_j^{c_j}Y_j^{d_j}\right)
			X_{k-1}^{c_{k-1}-p}Y_{k-1}^{d_{k-1}+p}X_k^{c_k+1}Y_k^{d_k-1}\mod \theta_k.
		\end{equation}
		Note that the operation involving $\theta_k$ in \eqref{mod-theta-j-equal}
                does not change the sum $\sum_{j=0}^{f-1}d_jp^j.$

		Now, we prove the result by induction on $f.$ If $f=1,$ then $P=0,$ and hence, the result follows. Assume that the result is true for $f-1.$ Without loss of generality we assume that $d_{f-1}\geq h_{f-1}.$ If $d_{f-1}= h_{f-1},$ we are done by the induction hypothesis.

                Suppose $d_{f-1}>h_{f-1}.$ 
                We assume that $d_{f-1}-h_{f-1}=1.$ If this difference is bigger than $1$, then we iterate the proof below until it is $1$. 
		Clearly $d_{f-1}\geq 1.$ If $c_{f-2}\geq p,$ taking $k=f-1$ in (\ref{mod-theta-j-equal}), we have
		\[\prod_{j=0}^{f-1}X_j^{c_j}Y_j^{d_j}=\left(\prod_{j=0}^{f-3}X_j^{c_j}Y_j^{d_j}\right)X_{f-2}^{c_{f-2}-p}Y_{f-2}^{d_{f-2}+p}X_{f-1}^{c_{f-1}+1}Y_{f-1}^{d_{f-1}-1}\mod\theta_{f-1}.\]
		Since the  operation involving $\theta_{f-1}$ does not change the sum $\sum_{j=0}^{f-1}d_jp^j$ and $d_{f-1}-1=h_{f-1},$ we are done by the case `$d_{f-1}=h_{f-1}$'.

		If $c_{f-2}<p,$ we write $c_{f-2}=p-i_1=i_0p-i_1,$ where $i_0=1$ and $0<i_1\leq p.$ Since by assumption $r_{f-2}\geq p,$ we have $d_{f-2}=r_{f-2}-i_0p+i_1\geq i_1.$ 
		Now, if $c_{f-3}\geq i_1p,$ taking $k=f-2$ in (\ref{mod-theta-j-equal}), we have
		\[ \prod_{j=0}^{f-1}X_j^{c_j}Y_j^{d_j}= \left(\prod_{\substack{j=0\\j\neq f-2, f-3}}^{f-1}X_j^{c_j}Y_j^{d_j}\right)X_{f-3}^{c_{f-3}-i_1p}Y_{f-3}^{d_{f-3}+i_1p}X_{f-2}^{c_{f-2}+i_1}Y_{f-2}^{d_{f-2}-i_1}\mod\theta_{f-2}^{i_1}.\]
		Since the operation involving $\theta_{f-2}$ does not change the sum $\sum_{j=0}^{f-1}d_jp^j$ and $c_{f-2}+i_1=p,$ we are done by the `$c_{f-2}\geq p$' case. And so on.
		

                Thus it is enough to show that this process stops. Suppose not. Then for $0\leq j\leq f-2,$ we have $0<i_{j+1}\leq i_jp$ such that $c_{f-2-j}=i_jp-i_{j+1}.$ In particular, $c_0<i_{f-2}p.$ By hypothesis, we have
		\[\sum\limits_{j=0}^{f-1}d_jp^j=\sum\limits_{j=0}^{f-1}h_jp^j\implies \sum\limits_{j=0}^{f-2}d_jp^j+(d_{f-1}-h_{f-1})p^{f-1}=\sum\limits_{j=0}^{f-2}h_jp^j,\]
		which by substituting $d_{f-1}-h_{f-1}=1, d_j=r_j-c_j$ and $h_j\leq r_j$ for all $0\leq j\leq f-2$ gives
		\[\sum\limits_{j=0}^{f-2}(r_j-c_j)p^j+p^{f-1}\leq \sum\limits_{j=0}^{f-2}r_jp^j\implies-\sum\limits_{j=0}^{f-2}c_jp^j+p^{f-1}\leq 0,\]
		which further by substituting $c_j=i_{f-2-j}p-i_{f-1-j}$ implies that
		\[-\sum\limits_{j=0}^{f-2}\left(i_{f-2-j}p-i_{f-1-j}\right)p^j+p^{f-1}\leq 0 \implies 
                -i_0p^{f-1}+i_{f-1}+p^{f-1}\leq 0.\] 
              Since $i_0=1,$ we conclude $i_{f-1}\leq 0.$ But we had $0<i_{f-1}\leq i_{f-2}p.$ Thus we arrive at a contradiction.
              So modulo $\langle \theta_1, \dots,\theta_{f-1}\rangle$ we are always reduced to the `$d_{f-1}=h_{f-1}$' case, and so we are done.
	\end{proof}

	We finally prove Theorem~\ref{ps-twisted} from the introduction.
        
	\begin{theorem}\label{general case}
		Let $m=m_0+m_1p+\dots+m_{f-1}p^{f-1}$ be the $p$-adic expansion of $m$ with $0\leq m_j\leq p-1$. Let $r=r_0+r_1p+\dots+r_{f-1}p^{f-1}$ with 
		$r_j\geq (m_j+1)(q+1)-1$ for all $0\leq j\leq f-1$. If $p\nmid{r_j \choose m_j}$ for all $0\leq j\leq f-1,$ then we have
		\[\frac{\bigotimes_{j=0}^{f-1}V_{r_j}^{{\rm Fr}^j}}{\langle\theta_0^{m_0+1},\dots,\theta_{f-1}^{m_{f-1}+1}\rangle}\simeq {\rm ind}_{B(\mathbb{F}_q)}^{G(\mathbb{F}_q)}\bigotimes_{j=0}^{f-1}\left(V_{m_j}^{{\rm Fr}^j}\otimes d^{(r_j-m_j)p^j}\right).\]
	\end{theorem}

	\begin{proof}\label{general case proof}
		 Define 
		 \[\psi:\bigotimes_{j=0}^{f-1}V_{r_j}^{{\rm Fr}^j}\rightarrow {\rm ind}_{B(\mathbb{F}_q)}^{G(\mathbb{F}_q)}\bigotimes_{j=0}^{f-1}\left(V_{m_j}^{{\rm Fr}^j}\otimes d^{(r_j-m_j)p^j}\right)\] by
		$\psi(\bigotimes_{j=0}^{f-1} P_j(X_j, Y_j))=\bigotimes_{j=0}^{f-1}\psi^j_{P_j(X_j, Y_j)}$, where $P_j(X_j, Y_j)\in V_{r_j}^{{\rm Fr}^j}$ for all $0\leq j\leq f-1$,
                and
                \[\bigotimes_{j=0}^{f-1}\psi^j_{P_j(X_j, Y_j)}:G(\mathbb{F}_q)\rightarrow\bigotimes_{j=0}^{f-1}\left(V_{m_j}^{{\rm Fr}^j}\otimes d^{(r_j-m_j)p^j}\right)\] is given by
		\[\left(\begin{matrix}
			a & b\\
			c & d
		\end{matrix}\right)\mapsto\bigotimes_{j=0}^{f-1}\left(\frac{{m_j\choose n_j}}{[r_j]_{m_j-n_j}}\nabla_j^{m_j-n_j}(P_j)\bigg \vert_{(c^{p^j}, d^{p^j})}\right)_{0\leq n_j\leq m_j}.\]
	    We show that $\psi$ induces a $G(\mathbb{F}_q)$-equivariant isomorphism between the right and left hand sides of the isomorphism in the statement of the theorem.
            
            We first show that $\bigotimes_{j=0}^{f-1}\psi^j_{P_j(X_j, Y_j)}$ is $B(\mathbb{F}_q)$-linear. This follows from the
            $B(\mathbb{F}_q)$-linearity of each $\psi^j_{P_j(X_j, Y_j)}$. Let $b\in B(\mathbb{F}_q)$ and $g\in G(\mathbb{F}_q).$ Then
            {\small
		\[ \left( \bigotimes_{j=0}^{f-1}\psi^j_{P_j(X_j, Y_j)} \right) (b\cdot g)=\bigotimes_{j=0}^{f-1}\psi^j_{P_j(X_j, Y_j)}(b\cdot g)=\bigotimes_{j=0}^{f-1}b\cdot\psi^j_{P_j(X_j, Y_j)}(g)=b\cdot\left(\bigotimes_{j=0}^{f-1}\psi^j_{P_j(X_j, Y_j)}(g)\right)  = b\cdot\left(\bigotimes_{j=0}^{f-1}\psi^j_{P_j(X_j, Y_j)}\right)(g).\] }
		The second equality holds by Lemma \ref{jthcomponent} (i). 
		
		Similarly, for the $G(\mathbb{F}_q)$-linearity of $\psi,$ we note that
		$\psi(g\cdot\bigotimes_{j=0}^{f-1} P_j(X_j, Y_j))=\psi(\bigotimes_{j=0}^{f-1}g\cdot P_j(X_j, Y_j))= \bigotimes_{j=0}^{f-1}\psi^j_{g\cdot P_j(X_j, Y_j)}=\bigotimes_{j=0}^{f-1}\psi^j(g\cdot P_j(X_j, Y_j))$, which by Lemma \ref{jthcomponent} (ii), equals 
                $\bigotimes_{j=0}^{f-1}g\cdot\psi^j(P_j(X_j, Y_j))=g\cdot\bigotimes_{j=0}^{f-1}\psi^j(P_j(X_j, Y_j)) =g \cdot\bigotimes_{j=0}^{f-1}\psi^j_{P_j(X_j, Y_j)}=g\cdot \psi(\bigotimes_{j=0}^{f-1} P_j(X_j, Y_j)).$

                Now there is a natural surjection
                $$\pi: \bigotimes_{j=0}^{f-1} {\rm ind}_{B(\mathbb{F}_q)}^{G(\mathbb{F}_q)} \left( V_{m_j}^{{\rm Fr}^j} \otimes d^{(r_j-m_j)p^j} \right)\twoheadrightarrow
                {\rm ind}_{B(\mathbb{F}_q)}^{G(\mathbb{F}_q)} \left( \bigotimes_{j=0}^{f-1} V_{m_j}^{{\rm Fr}^j} \otimes d^{(r_j-m_j)p^j} \right)$$
                given by $\pi(\otimes_{j=0}^{f-1} F_j) = F$ with $F(g) = \otimes_{j=0}^{f-1} F_j(g)$ for all $g \in G({\mathbb F}_q)$.
                Note that $\pi$ is not necessarily injective (take $f > 1$ and compare dimensions on both sides of $\pi$). 
		By definition of $\psi$, we have that $\psi(\otimes_{j=0}^{f-1} P_j(X_j, Y_j))(g) = \otimes_{j=0}^{f-1}\psi_{P_j(X_j,Y_j)}^j(g) = 
                \pi (\otimes_{j=0}^{f-1}\psi_{P_j(X_j,Y_j)}^j)(g)$ for all $g \in G({\mathbb F}_q)$.
                Hence we have that $\psi(\otimes_{j=0}^{f-1} P_j(X_j, Y_j)) =   \pi (\otimes_{j=0}^{f-1}\psi_{P_j(X_j,Y_j)}^j)$.
                But the last map is equal to $\pi ( (\otimes_{j=0}^{f-1} \psi^j) (\otimes_{j=0}^{f-1} P_j(X_j,Y_j)))
                = (\pi \circ \otimes_{j=0}^{f-1}\psi^j)(\otimes_{j=0}^{f-1} P_j(X_j,Y_j))$. Thus  
                $$\psi = \pi \circ (\otimes_{j=0}^{f-1}\psi^j).$$
                By Lemma \ref{jthcomponent} (iii), each $\psi^j$ is surjective, hence so is $\otimes_{j=0}^{f-1}\psi^j$.
                Since $\pi$ is surjective, so is $\psi$.

		Note that each $\psi_j$ is not injective, hence $\otimes_{j=0}^{f-1}\psi^j$ is not injective, therefore $\psi$ is not injective.
                Thus it remains to compute $\ker \psi$. 
		We show that $\ker\psi_m=\langle\theta_0^{m_0+1},\dots,\theta_{f-1}^{m_{f-1}+1}\rangle,$ where we write $\psi_m$ instead of $\psi$ for emphasis.
                By Lemma \ref{thetaj}, for $0\leq a_j\leq m_j,$ 
		\[\left(\prod_{j=0}^{f-1}\nabla_j^{a_j}\right)(\theta_s^{m_s+1})\bigg \vert_{(c, d;\dots;c^{p^{f-1}},d^{p^{f-1}})}=0,\]
		for all $0\leq s\leq f-1.$ By the definition of $\psi_m,$ we see that
		$\langle\theta_0^{m_0+1},\dots,\theta_t^{m_t+1},\dots,\theta_{f-1}^{m_{f-1}+1}\rangle\subset \ker\psi_m$.
                We now prove the other containment
                \begin{eqnarray}
                  \label{other-containment}
                  \ker\psi_m\subset \langle\theta_0^{m_0+1},\dots,\theta_t^{m_t+1},\dots,\theta_{f-1}^{m_{f-1}+1}\rangle.
                \end{eqnarray}
                This is the trickiest part of the proof of the theorem. We will use  
                the three lemmas proved before the theorem.
                
                The proof is by induction on $\sum_p(m),$ where $\sum_p(m)=m_0+\dots+m_{f-1}$ denotes the sum of the $p$-adic digits in the base $p$ expansion of $m.$
                If $\sum_p(m)=0,$ then we have $m_j=0$ for all $0\leq j\leq f-1.$
                Noting that $r_j\geq (m_j+1)(q+1)-1\geq p^{f-j}$ for $0 \leq j\leq f-1$, we see that \eqref{other-containment} follows immediately
                from Lemmas~\ref{monomial difference},  \ref{unequal} and \ref{equal}.
		
		Now, suppose $\sum_p(m)\geq 1$. Assume that \eqref{other-containment} holds for $m'$ with $\sum_p(m') \leq \sum_p(m)-1$. 
                Pick $t$ such that $m_t\geq 1.$ Let $m'=m_0+\dots+(m_t-1)p^t+\dots+m_{f-1}p^{f-1}.$ Then $\sum_p(m')=\sum_p(m)-1$ and so  
		\begin{eqnarray}
                    \label{other-containment-m'}
                    \ker\psi_{m'} \subset \langle\theta_0^{m_0+1},\dots, \theta_t^{m_t}, \dots, \theta_{f-1}^{m_{f-1}+1}\rangle,
		\end{eqnarray}
                by the induction hypothesis. Now, let $P\in \ker\psi_m.$ So 
		$\left(\prod_{j=0}^{f-1}\nabla_j^{a_j}\right)(P)\big \vert_{(c, d;\dots;c^{p^{f-1}},d^{p^{f-1}})}=0$ for all $0\leq a_j\leq m_j.$
                In particular, $\left(\prod_{j=0}^{f-1}\nabla_j^{a_j}\right)(P)\big \vert_{(c, d;\dots;c^{p^{f-1}},d^{p^{f-1}})}=0$ for all $0\leq a_j\leq m_j$ with $j\neq t$
                and for all $0\leq a_t\leq  m_t-1.$ This shows that $P\in \ker\psi_{m'}.$ 
		By \eqref{other-containment-m'}, we may write
		\[P=Q_0\theta_0^{m_0+1}+\dots+Q_t\theta_t^{m_t}+\dots+Q_{f-1}\theta_{f-1}^{m_{f-1}+1},\] 
		with $Q_j\in \mathbb{F}_q[X_0,Y_0;\dots;X_{f-1},Y_{f-1}]$ for $0\leq j\leq f-1.$ 
		Clearly,
		\[P\in\langle\theta_0^{m_0+1},\dots,\theta_t^{m_t+1},\dots,\theta_{f-1}^{m_{f-1}+1}\rangle
		\iff Q_t\theta_t^{m_t}\in\langle\theta_0^{m_0+1},\dots,\theta_t^{m_t+1},\dots,\theta_{f-1}^{m_{f-1}+1}\rangle.\]
		So without loss of generality, let $P=Q_t\theta_t^{m_t}.$ Since $P\in\ker\psi_m,$ for $0\leq a_j\leq m_j$ and $j\neq t,$ we have 
{\small		\begin{align*}
			&\left(\prod_{j=0,j\neq t}^{f-1}\nabla_j^{a_j}\right)\nabla_t^{m_t}(Q_t\theta_t^{m_t})\bigg \vert_{(c, d;\dots;c^{p^{f-1}},d^{p^{f-1}})}=0\\
			&\implies \left(\prod_{j=0,j\neq t}^{f-1}\nabla_j^{a_j}\right)\left(\sum\limits_{l=0}^{m_t}{m_t\choose l}\nabla_t^{m_t-l}(Q_t)\nabla_t^l(\theta_t^{m_t})\right)\bigg \vert_{(c, d;\dots;c^{p^{f-1}},d^{p^{f-1}})} =0\\
			&\implies\sum_{l=0}^{m_t}{m_t\choose l}\nabla_{t-1}
			^{a_{t-1}}\left(\left(\prod_{j=0,j\neq t,t-1}^{f-1}\nabla_j^{a_j}\right)\nabla_t^{m_t-l}(Q_t)\nabla_t^l(\theta_t^{m_t})\right)\bigg \vert_{(c, d;\dots;c^{p^{f-1}},d^{p^{f-1}})} = 0 \\
			&\implies\sum_{l=0}^{m_t}{m_t\choose l}\left( \sum_{k=0}^{a_{t-1}}{a_{t-1}\choose k}\nabla_{t-1}^{a_{t-1}-k}\left(\left(\prod_{j=0,j\neq t,t-1}^{f-1}\nabla_j^{a_j}\right)\nabla_t^{m_t-l}(Q_t)\right)\nabla_{t-1}^k\left(\nabla_t^l\left(\theta_t^{m_t}\right)\right)\right) \bigg \vert_{(c, d;\dots;c^{p^{f-1}},d^{p^{f-1}})} = 0 \\
			&\implies \left(\prod_{j=0,j\neq t}^{f-1}\nabla_j^{a_j}\right)\nabla_t^0(Q_t)\bigg \vert_{(c, d;\dots;c^{p^{f-1}},d^{p^{f-1}})}=0\\
			&\implies Q_t\in \langle\theta_0^{m_0+1},\dots,\theta_t^{1},\dots,\theta_{f-1}^{m_{f-1}+1}\rangle \> \implies \> Q_t\theta_t^{m_t}\in \langle\theta_0^{m_0+1},\dots,\theta_t^{m_t+1},\dots,\theta_{f-1}^{m_{f-1}+1}\rangle.
		\end{align*}}
                \!\!\! The fourth implication follows from Lemma \ref{thetaj}: if $(k,l) = (0,m_t)$, then  
                \[\nabla_{t-1}^k \nabla_t^{l}(\theta_t^{m_t}) \bigg \vert_{(c, d;\dots;c^{p^{f-1}},d^{p^{f-1}})} = m_t!(ad-bc)^{m_tp^{t}}\neq 0,\]
                and is $0$ for all other $(k,l)$. 
                The penultimate implication holds by the induction hypothesis as the sum of the $p$-adic digits is $\Sigma_p(m)-m_t \leq \Sigma_p(m)-1$.
                So $P\in \langle\theta_0^{m_0+1},\dots,\theta_t^{m_t+1},\dots,\theta_{f-1}^{m_{f-1}+1}\rangle,$
                proving \eqref{other-containment}.
	\end{proof}

	\section{Dual numbers}

        This section is an aside.
        The ring of generalized dual numbers is defined by $\mathbb{F}_p[\epsilon] =\frac{\mathbb{F}_p[X]}{\langle X^{m+1} \rangle}.$
         We make some remarks on two questions that arise in the context of the lack of surjectivity when $m > 0$ of
        the map \eqref{dual} which involves dual numbers  (introduced in \cite[Lemma 4.1]{GV22}).
	Firstly, can one possibly replace the inducing subgroup $B(\mathbb{F}_p[\epsilon])$ in \eqref{dual}
        by another subgroup $B'$ of
        index $p+1$ in $G(\mathbb{F}_p[\epsilon])$ and $d^r$ by a
	surjective character $\chi_r : B' \rightarrow \mathbb{F}_p[\epsilon]^\times$ such that there is an isomorphism
	$$ \frac{V_r}{V_r^{(m+1)}} \stackrel{{\tiny{?}}}{\simeq} \mathrm{ind}_{B'}^{G(\mathbb{F}_p[\epsilon])} \chi_r$$
	which might then be used to study periodicity results? The answer is no, and explains why in this
        paper we turned towards proving the isomorphism in Theorem~\ref{ps}. Secondly,  can one at least describe the image of
        \eqref{dual} in a more conceptual way than in \cite[Lemma 4.1]{GV22}? The answer in some cases is yes 
        (see Proposition~\ref{image} and a longer arXiv version of this paper \url{https://arxiv.org/pdf/2308.10246} for its proof).

        It would be interesting to see how the material in this section connected to announced
        work of Schein and his coauthors on the modular representation theory
        of $\mathrm{GL}_2(R)$ where $R$ is a finite quotient ring of ${\mathcal O}_F$ for $F$ a $p$-adic field, and, e.g., 
        to work of Avni, Onn, Prasad, Vaserstein \cite{AOPV}.

        \subsection{Isomorphisms using dual numbers}	
	
	
	There are two notions of projective space over the generalized dual numbers. The first is standard projective space 
	$${\mathbb P}^1(\mathbb{F}_p[\epsilon]) = \{ [x:y] | (x,y) = 1 \}.$$
        It has cardinality $p(p+1)$ when $m = 1$ and is the cylinder
	obtained by glueing the line at $\infty$, namely $[1 : d\epsilon]$, to the plane $[c : 1]$ \cite{gru06}. The second is
	$\tilde{\mathbb P}^1(\mathbb{F}_p[\epsilon]) = \{ [x:y] | (x,y) \neq 0 \}$. It has cardinality $(p+1)^2$ when $m = 1$.  The group $G(\mathbb{F}_p[\epsilon])$
	acts on the left on both these spaces via $\left(\begin{smallmatrix} a & b \\ c & d \end{smallmatrix} \right) \cdot [x:y] = [ax+by:cx+dy]$.
	
	Let $B'$ be the $\epsilon$-Iwahori subgroup of $G(\mathbb{F}_p[\epsilon])$  obtained as the pre-image of the usual Borel
	$B({\mathbb F}_p)$ under the reduction modulo $\epsilon$ map $G(\mathbb{F}_p[\epsilon]) \rightarrow G(\mathbb{F}_p)$.
	Then $B'$ is the stabilizer of $[\epsilon:0]$ under the action of 
	$G(\mathbb{F}_p[\epsilon])$ on $\tilde{\mathbb P}^1(\mathbb{F}_p[\epsilon])$
	(whereas $B(\mathbb{F}_p[\epsilon])$ is the stabilizer of $[1:0]$ in ${\mathbb P}^1(\mathbb{F}_p[\epsilon])$).
	Clearly $B'$ has index $p+1$.
	
	One may ask if there is a surjective character $\chi : B' \rightarrow \mathbb{F}_p[\epsilon]^\times$ which induces
        a $G(\mathbb{F}_p)$-isomorphism
	$$ \frac{V_r}{V_r^{(m+1)}} \simeq \mathrm{ind}_{B'}^{G(\mathbb{F}_p[\epsilon])} \chi_r$$
	from which the periodicity of the left side would follow if one knew $\chi_r$ only depended on
	$r$ modulo $p(p-1)$.
	
	The answer is no. Indeed, one quickly sees that $B'$ has abelianization
	$\mathbb{F}_p[\epsilon]^\times \times \mathbb{F}_p^\times$ and so the only surjective characters $B' \rightarrow \mathbb{F}_p[\epsilon]^\times$
	it supports are powers of the determinant character. These characters are not genuine characters of $B'$, since they are obtained
	by restricting from $G(\mathbb{F}_p[\epsilon])$, so the induction is not so well-behaved. 
	
	Moreover, one checks that every other subgroup $B''$ of index $p+1$ in $G(\mathbb{F}_p[\epsilon])$ is conjugate to $B'$ so this
	line of reasoning does not bear fruit.

	\subsection{Image of \eqref{dual}} Thus, the best one can hope to do is to characterize the image of the (non-surjective) map (if $m > 0$)
	$$ \frac{V_r}{V_r^{(m+1)}} \hookrightarrow \mathrm{ind}_{B(\mathbb{F}_p[\epsilon])}^{G(\mathbb{F}_p[\epsilon])} d^r$$
	mentioned in the introduction (cf. \cite[Lemma 4.1]{GV22}).
	
	To this end, we consider the right action of $G(\mathbb{F}_p[\epsilon])$ on standard projective space ${\mathbb P}^1(\mathbb{F}_p[\epsilon])$
	defined via  $[x:y] \cdot  \left(\begin{smallmatrix} a & b \\ c & d \end{smallmatrix} \right)  = [ax+cy:bx+dy]$.
	This action is transitive and the stabilizer of $[0:1]$ under this action is $B(\mathbb{F}_p[\epsilon]).$ We have the following decomposition
	\[G(\mathbb{F}_p[\epsilon])=B(\mathbb{F}_p[\epsilon])\left(\begin{matrix}
		1 & 0\\
		c & 1
	\end{matrix}\right)\bigsqcup B(\mathbb{F}_p[\epsilon])\left(\begin{matrix}
		0 & 1\\
		1 & d\epsilon
	\end{matrix}\right),\]
	where $c,d\in \mathbb{F}_p[\epsilon].$
	There is a bijection between $B(\mathbb{F}_p[\epsilon])\backslash G(\mathbb{F}_p[\epsilon])$ and ${\mathbb P}^1(\mathbb{F}_p[\epsilon])$ by sending $B(\mathbb{F}_p[\epsilon])g$ to $[0:1]g$ for $g\in G(\mathbb{F}_p[\epsilon])$.
        
	We say that  $f :{\mathbb P}^1(\mathbb{F}_p[\epsilon])\rightarrow \mathbb{F}_p[\epsilon]$ is {\it smooth} if for all
        $z_0+z'\epsilon\in \mathbb{F}_p[\epsilon]$ with $z_0\in\mathbb{F}_p$ and $z'=z_1+z_2\epsilon+\dots+z_{m-1}\epsilon^{m-1}$ with $z_1,\dots,z_{m-1}\in\mathbb{F}_p$
        and all $0 \leq j \leq m$, there exist constants $f^{(j)} ( [z_0: 1 ] )$ and
        $f^{(j)}( [1: 0] )$ in ${\mathbb F}_p$ such that    
	\[f \left( [z_0+z'\epsilon : 1 ] \right)=\sum\limits_{\substack{j=0}}^{m}\dfrac{(z'\epsilon)^{j}}{j!}f^{(j)}\left( [z_0: 1 ] \right) \quad \text{and} \quad 
	f \left( [1:  z'\epsilon ] \right)=\sum\limits_{\substack{j=0}}^{m}\dfrac{(z'\epsilon)^{j}}{j!}f^{(j)}\left( [1: 0] \right).\]
	
	\begin{proposition}\label{image}
		Let $\psi:V_r\rightarrow {\rm ind}_{B(\mathbb{F}_p[\epsilon])}^{G(\mathbb{F}_p[\epsilon])}d^r$ be
		given by
		$\psi(P(X, Y))=\psi_{P(X, Y)}$ for all $P(X, Y)\in V_r,$
		where $\psi_{P(X, Y)}: G(\mathbb{F}_p[\epsilon])\rightarrow \mathbb{F}_p[\epsilon]$ is defined by
		$\psi_P\left(\left(\begin{smallmatrix}
			a & b\\
			c & d
		\end{smallmatrix}\right)\right)=P(c, d)$
		for $\left(\begin{smallmatrix}
			a & b\\
			c & d
		\end{smallmatrix}\right)\in G(\mathbb{F}_p[\epsilon]).$ If $r \equiv 0$ mod $p(p-1)$, then 
              \[ {\rm Im}~\psi
              = \left\{f:{\mathbb P}^1(\mathbb{F}_p[\epsilon])\rightarrow \mathbb{F}_p[\epsilon]\mid f(\alpha)\in \mathbb{F}_p ~\text{if}~\alpha\in{\mathbb P}^1(\mathbb{F}_p)~\text{and}~f ~\text{is smooth} \right\}.\]
	\end{proposition}

	\section{Cuspidal case}
          \label{section cuspidal}
        In this section, we prove Theorems~\ref{cuspidal} and \ref{cuspidal-twisted}
        which, as explained in the introduction, are the cuspidal analogs of Theorems~\ref{ps} and \ref{ps-twisted}.

        \subsection{The case of ${\rm GL}_2({\mathbb{F}_p})$}
	
	Let $\alpha \in \mathbb{F}_{p^2}$ be such that $\alpha^2\in \mathbb{F}_p$ and $\alpha\notin\mathbb{F}_p.$
        Fix an identification $i:\mathbb{F}_{p^2}^{\times}\simeq T(\mathbb{F}_p)\subset {\rm GL}_2(\mathbb{F}_p)$ given by 
	$u + v\alpha \mapsto \left(\begin{smallmatrix}
		u & v\alpha^2\\
		v & u 
	\end{smallmatrix}\right)$
        for $u$, $v \in {\mathbb F}_p$ not both zero. Recall that $\omega_2 : T(\mathbb{F}_p) \rightarrow \mathbb{F}_{p^2}^{\times}$ is the inverse of $i$.	

                We define some functions in induced spaces.
		For $r\geq 0$ and $0\leq i\leq r+p^2+1,$ let $f_i:G(\mathbb{F}_p)\rightarrow \mathbb{F}_{p^2}$ be  
		\[f_i\left(\left(\begin{matrix}
			a & b\\
			c & d
		\end{matrix}\right)\right)=(a+c\alpha)^{(r+p^2+1)-i}(b+d\alpha)^i,\] for all $g=\left(\begin{smallmatrix}
			a & b\\
			c & d
		\end{smallmatrix}\right)\in G(\mathbb{F}_p).$ Then $f_i$ is $T(\mathbb{F}_p)$-linear and hence
		$f_i\in {\rm ind}_{T(\mathbb{F}_p)}^{G(\mathbb{F}_p)}\omega_2^{r+2}.$ Indeed, for $t=\left(\begin{smallmatrix}
			u  & v\alpha^2\\
			v & u
		\end{smallmatrix}\right)\in T(\mathbb{F}_p),$ we have
		\[f_i\left(t\cdot g\right)=f_i\left(\left(\begin{matrix}
			ua+v\alpha^2c & ub+v\alpha^2d\\
			va+uc & vb+ud
		\end{matrix}\right)\right)\\
		=(u+v\alpha)^{r+p^2+1}(a+c\alpha)^{(r+p^2+1)-i}(b+d\alpha)^i\]
		which equals
		\[(u+v\alpha)^{(r+2)+p^2-1}\cdot f_i(g)
		=(u+v\alpha)^{r+2}\cdot f_i(g)
		=\omega_2^{r+2}(t)\cdot f_i(g)
		=t\cdot f_i(g).\]

       One can check that the functions in $\mathcal{B}=\left\{f_i\mid 0\leq i\leq p^2-p-1 \right\}$ are linearly independent.
       Also, $T(\mathbb{F}_p)$ has index $p^2-p$ in $G(\mathbb{F}_p).$
       So $\mathcal{B}$ forms a basis of ${\rm ind}_{T(\mathbb{F}_p)}^{G(\mathbb{F}_p)}\omega_2^{r+2}.$  We fix this basis in the
       computations to follow. 
	
	For $p^2-1\leq i\leq r+p^2+1,$ we observe that
	\begin{equation}\label{flip}
		f_i=f_{p^2-1+j}=f_j
	\end{equation}
	for some $0\leq j\leq r+2.$ In this case, we say that $f_i$ is a \emph{flip}.
        We shall soon assume that $r \leq p-3 \leq p^2-p-3$ in which case $f_i$ lies in ${\mathcal B}$. 
	
	On the other hand, for $p^2-p\leq i\leq p^2-2,$ we have $f_i=f_{p^2-p+j}$ for some $0\leq j\leq p-2.$ Then we say that $f_i$ is a \emph{flop}.
        It is the last term in the following relation:
      	\begin{equation}\label{flop}
           f_j+f_{j+(p-1)}+f_{j+2(p-1)}+\dots +f_{j+(p-1)(p-1)}+f_{j+p^2-p}=0
         \end{equation}
        where all but the last term lie in ${\mathcal B}$.
	Indeed, we have 
	\[X^{p^2-1}-1=\left(X^{p-1}-1\right)\left(X^{p(p-1)}+X^{(p-1)(p-1)}+\dots+X^{p-1}+1\right).\]
	Then for $A\in \mathbb{F}_{p^2}^{\times} \setminus \mathbb{F}_p^{\times},$ we have
	$A^{p(p-1)}+A^{(p-1)(p-1)}+\dots+A^{p-1}+1=0$.
	So for $\left(\begin{smallmatrix}
		a & b\\
		c & d
	\end{smallmatrix}\right)\in G(\mathbb{F}_p),$ we have 
	\[\left(\dfrac{a+c\alpha}{b+d\alpha}\right)^{p(p-1)}+\left(\dfrac{a+c\alpha}{b+d\alpha}\right)^{(p-1)(p-1)}+\dots+\left(\dfrac{a+c\alpha}{b+d\alpha}\right)^{p-1}+1=0,\]
	which, after multiplying by $(a+c\alpha)^{(r+p+1)-j}(b+d\alpha)^{j+p^2-p}$ on both sides, gives 
        $$(a+c\alpha)^{(r+p^2+1)-j}(b+d\alpha)^{j}+(a+c\alpha)^{(r+p^2-p+2)-j}(b+d\alpha)^{j+(p-1)}
        +\dots+(a+c\alpha)^{(r+p+1)-j}(b+d\alpha)^{j+p^2-p}=0,$$
	which shows that \eqref{flop} holds for $0\leq j\leq p-2$.

	Thus, any flip or flop can be expressed as an easy linear combination of vectors in $\mathcal{B}$ using  \eqref{flip}, \eqref{flop}.
	
	For any polynomial $P(X, Y)$ and $A, B, C, D \in {\mathbb F}_{p^2}$,
we set 
	\[P(X, Y)\Big\vert_{(A, B)}^{(C, D)} := P(C, D)-P(A, B).\]
		
	The following theorem is Theorem~\ref{cuspidal} from the introduction (replacing $r$ by $r+2$).
        \begin{theorem}
                \label{base case cuspidal}
		Let $0\leq r\leq p-3.$ Then there is an explicit isomorphism defined over ${\mathbb F}_{p^2}$:
		\[\dfrac{V_{r+p+1}}{D(V_{r+2})}\otimes V_{p-1}\simeq {\rm ind}_{T(\mathbb{F}_p)}^{G(\mathbb{F}_p)}\omega_2^{r+2},\]
                where  $D:= X^p\frac{\partial}{\partial X}+Y^p\frac{\partial}{\partial Y}$.
	\end{theorem}
	
	\begin{proof}\label{proof of base case cuspidal}
		Let $P\in V_{r+p+1}$ and $Q\in V_{p-1}.$ Define 
		
	        \begin{eqnarray}
		\label{explicit iso}
		\psi:V_{r+p+1}\otimes V_{p-1}\rightarrow {\rm ind}_{T(\mathbb{F}_p)}^{G(\mathbb{F}_p)}\omega_2^{r+2}
		\end{eqnarray}	
		by $\psi(P\otimes Q)=\psi_{P\otimes Q},$ where $\psi_{P\otimes Q}: G(\mathbb{F}_p)\rightarrow \mathbb{F}_{p^2}$ is defined by
		\begin{align*}
			&\psi_{P\otimes Q}\left(\left(\begin{matrix}
				a & b\\
				c & d
			\end{matrix}\right)\right)
                                    =\nabla_{\alpha}^r(P)\Big\vert_{\left(a+c\alpha, \> b+d\alpha\right)}^{\left((a+c\alpha)^p, \> (b+d\alpha)^p\right)}\cdot
                                     Q \Big\vert_{(0,0)}^{\left((a+c\alpha)^p, \> (b+d\alpha)^p\right)}\\
		\end{align*}
		for all $\left(\begin{smallmatrix}
			a & b\\
			c & d
                      \end{smallmatrix}\right)\in G(\mathbb{F}_p),$ and where
                    $$\nabla_{\alpha}= (a+c\alpha)\dfrac{\partial}{\partial X}+(b+d\alpha)\dfrac{\partial}{\partial Y}.$$ 
	       For convenience, we set
               \begin{eqnarray}
                 \label{def of A alpha and B alpha base case}
                 A_\alpha =  a+c\alpha & \text{and} & 
                 B_\alpha =  b+d\alpha.
               \end{eqnarray}
               We think of $A_\alpha$ and $B_\alpha$ as ${\mathbb F}^\times_{p^2}$-valued functions on $G({\mathbb F}_p)$. We remark that the
               definition \eqref{explicit iso}  of $\psi$ reminds one of the evaluation of a direct integral in calculus. It was
               arrived at after much explicit computation with special cases. \\
               
         \noindent \textbf{$T(\mathbb{F}_p)$-linearity:} We show that $\psi_{P\otimes Q}$ is $T(\mathbb{F}_p)$-linear.
         Since $P(X, Y)$ is a homogeneous polynomial of degree $r+p+1$ and $Q(S, T)$ is a homogeneous polynomial of degree $p-1,$
         we have $\nabla_{\alpha}^r(P(X, Y))\cdot Q(S, T)$ is a linear combination of terms of the form
		\[A_\alpha^{r-j}B_\alpha^j\cdot X^{p+1-k}Y^k\cdot S^{p-1-l}T^l\] for $0\leq k\leq p+1,0\leq l\leq p-1,$ and $0\leq j\leq r.$ 
		Now,
		\begin{eqnarray*}
			A_\alpha^{r-j}B_\alpha^j\cdot X^{p+1-k}Y^k\Big\vert_{(A_\alpha, B_\alpha)}^{(A_\alpha^p, B_\alpha^p)}\cdot S^{p-1-l}T^l\Big\vert_{\left(0, 0\right)}^{(A_\alpha^p, B_\alpha^p)}
			& = &  A_\alpha^{r-j}B_\alpha^j\left(A_\alpha^{p^2-kp-lp+1}B_\alpha^{kp+lp}-A_\alpha^{p^2-lp-k+1}B_\alpha^{k+lp}\right),
		\end{eqnarray*}	
		which shows that $\psi_{P\otimes Q}$ is a linear combination of the functions $f_i$ defined above.
                Since these functions are $T(\mathbb{F}_p)$-linear,
                $\psi_{P\otimes Q}$ is also $T(\mathbb{F}_p)$-linear. \\

\noindent	\textbf{$G(\mathbb{F}_p)$-linearity:} We show that $\psi$ is $G(\mathbb{F}_p)$-linear. Let $g=\left(\begin{smallmatrix}
			u & v\\
			w & z
		\end{smallmatrix}\right)\in G(\mathbb{F}_p).$ Then we have
               \[g\cdot (P\otimes Q) = P(U, V)\otimes Q(U', V')=P_1\otimes Q_1 (\text{say}), \]
		where $U =uX+wY, V=vX+zY$ and $U'=uS+wT, V'=vS+zT.$
%
%
		Now,
		\begin{eqnarray*}
		   \psi(g\cdot (P\otimes Q))\left(\left(\begin{matrix}
				a & b\\
				c & d
			\end{matrix}\right)\right)
			&= &\psi(P_1\otimes Q_1)\left(\left(\begin{matrix}
				a & b\\
				c & d
			\end{matrix}\right)\right)\\
			&= &\left((a+c\alpha)\dfrac{\partial}{\partial X}+(b+d\alpha)\dfrac{\partial}{\partial Y}\right)^r(P_1)\Big\vert_{\left(\left(a+c\alpha\right), \left(b+d\alpha\right)\right)}^{\left(\left(a+c\alpha\right)^p, \left(b+d\alpha\right)^p\right)} \cdot \\
			&  &{\hspace{20pt}}Q\left(U'\left(\left(a+c\alpha\right)^p, \left(b+d\alpha\right)^p\right), V'\left(\left(a+c\alpha\right)^p, \left(b+d\alpha\right)^p\right)\right),
		\end{eqnarray*}
		which, by Lemma~\ref{glineargeneral} applied twice
		\begin{eqnarray*}
			&& =\left(\left(uA_\alpha+wB_\alpha\right)\dfrac{\partial}{\partial X}+\left(vA_\alpha+zB_\alpha\right)\dfrac{\partial}{\partial Y}\right)^r(P)\bigg\vert_{\left(uA_\alpha+wB_\alpha, vA_\alpha+zB_\alpha\right)}^{\left(\left(uA_\alpha+wB_\alpha\right)^p,\left(vA_\alpha+zB_\alpha\right)^p\right)} \cdot\\
			&& {\hspace{50pt}}Q\left(\left(uA_\alpha+wB_\alpha\right)^p, \left(vA_\alpha+zB_\alpha\right)^p\right)\\
			&& =\psi_{P\otimes Q}\left(\left(\begin{matrix}
				au+bw  & av+bz\\
				cu+dw  & cv+dz
			\end{matrix}\right)\right)
			  =\left(\begin{matrix}
				u & v\\
				w & z
			\end{matrix}
			\right)\cdot\psi_{P\otimes Q}\left(\left(\begin{matrix}
				a & b\\
				c & d
			\end{matrix}\right)\right)
			 =g\cdot\psi_{P\otimes Q}\left(\left(\begin{matrix}
				a & b\\
				c & d
			\end{matrix}\right)\right).\\
		\end{eqnarray*}
		Thus we have $\psi(g\cdot (P\otimes Q))=g\cdot\psi_{P\otimes Q}.$ Hence $\psi$ is $G(\mathbb{F}_p)$-linear. \\
		
                \noindent	\textbf{Kernel of $\psi$:} Next we show that $\ker\psi=D(V_{r+2})\otimes V_{p-1}.$ We first show that
                $D(V_{r+2})\otimes V_{p-1}\subset \ker\psi$. 
                Let $P=\sum\limits_{\substack{i=0}}^{r+2}a_iX^{r+2-i}Y^i\in V_{r+2}$ and $Q\in V_{p-1}.$
                Then $D(P)=\sum\limits_{\substack{i=0}}^{r+2}a_i\left((r+2-i)X^{r+p+1-i}Y^i+iX^{r+2-i}Y^{i+p-1}\right)$.
		\textbf{Claim:} for all $\left(\begin{smallmatrix}
			a & b\\
			c & d
                      \end{smallmatrix}\right)\in G(\mathbb{F}_p),$ we have $\nabla_{\alpha}^r(D(P))\Big\vert_{\left(a+c\alpha, b+d\alpha\right)}^{\left((a+c\alpha)^p, (b+d\alpha)^p\right)}=0$.
                    
\noindent		Indeed,
{\small		\begin{eqnarray*}
			\nabla_{\alpha}^r(D(P))
			& = & \left(A_\alpha\dfrac{\partial}{\partial X}+B_\alpha\dfrac{\partial}{\partial Y}\right)^r\left(\sum\limits_{\substack{i=0}}^{r+2}a_i\left((r+2-i)X^{r+p+1-i}Y^i+iX^{r+2-i}Y^{i+p-1}\right)\right)\\
			& = &\sum\limits_{\substack{i=0}}^{r+2}\sum\limits_{\substack{k=0}}^r\binom{r}{k}A_\alpha^{r-k}B_\alpha^k\dfrac{\partial^r}{\partial X^{r-k}\partial Y^k}a_i\left((r+2-i)X^{r+p+1-i}Y^i+iX^{r+2-i}Y^{i+p-1}\right)\\
			& = & \sum\limits_{\substack{i=0}}^{r+2}\sum\limits_{\substack{k=0}}^r\binom{r}{k}a_i (r+2-i)A_\alpha^{r-k}B_\alpha^k[r+p+1-i]_{r-k}[i]_kX^{p+1-(i-k)}Y^{i-k}\\
			&   & +\sum\limits_{\substack{i=0}}^{r+2}\sum\limits_{\substack{k=0}}^r\binom{r}{k}a_i i A_\alpha^{r-k}B_\alpha^k[r+2-i]_{r-k}[i+p-1]_kX^{k+2-i}Y^{i+p-1-k},
		\end{eqnarray*}}
\!\!\!		which, by observing $[i]_k$ equals $0$ if $i<k$ and equals $i!$ if $i=k$,
{\small		\begin{eqnarray*}
		\qquad	&= & a_0(r+2)A_\alpha^r[r+p+1]_rX^{p+1} +\sum\limits_{\substack{i=1}}^{r+1}\sum\limits_{\substack{k=0}}^r\binom{r}{k}a_i (r+2-i)A_\alpha^{r-k}B_\alpha^k[r+p+1-i]_{r-k}[i]_kX^{p+1-(i-k)}Y^{i-k}\\
			& &  +a_{r+2}(r+2)B_\alpha^r[r+p+1]_rY^{p+1} +\sum\limits_{\substack{i=1}}^{r+1}\sum\limits_{\substack{k=0}}^r\binom{r}{k}a_i i A_\alpha^{r-k}B_\alpha^k[r+2-i]_{r-k}[i+p-1]_kX^{k+2-i}Y^{i+p-1-k},
		\end{eqnarray*}}
\!\!\!		which again using $[r+p+1-i]_{r-k}=0 \mod p$ for $k<i-1,$ $[i]_k=0$ for $k>i$ and $[r+2-i]_{r-k}=0$ for $k<i-2$ and $[i+p-1]_k=0 \mod p$ for $k>i-1,$ 
{\small		\begin{eqnarray*}
			&= &  a_0(r+2)A_\alpha^r[r+p+1]_rX^{p+1}
		              +\left(\sum\limits_{\substack{i=1}}^{r+1}\binom{r}{i}a_i(r+2-i)A_\alpha^{r-i}B_\alpha^i[r+p+1-i]_{r-i}[i]_i\right)X^{p+1}\\
			&  & +\left(\sum\limits_{\substack{i=1}}^{r+1}\binom{r}{i-1}a_i(r+2-i)A_\alpha^{r-i+1}B_\alpha^{i-1}[r+p+1-i]_{r-i+1}[i]_{i-1}\right)X^pY\\
			&  & +a_{r+2}(r+2)B_\alpha^r[r+p+1]_rY^{p+1} 
			     +\left(\sum\limits_{\substack{i=1}}^{r+1}\binom{r}{i-2}a_i i A_\alpha^{r-i+2}B_\alpha^{i-2}[r+2-i]_{r+2-i}[i+p-1]_{i-2}\right)Y^{p+1}\\
			&  & +\left(\sum\limits_{\substack{i=1}}^{r+1}\binom{r}{i-1}a_i i A_\alpha^{r-i+1}B_\alpha^{i-1}[r+2-i]_{r-i+1}[i+p-1]_{i-1}\right)XY^p.
		\end{eqnarray*}}
\!\!\!		Since $[r+p+1-i]_{r-i+1}=(r-i+1)!$ modulo $p$ and $[i+p-1]_{i-1}=(i-1)!$ modulo $p,$ the above expression
{\small		\begin{eqnarray*}
			&= & a_0(r+2)A_\alpha^r[r+p+1]_rX^{p+1}
			     +\left(\sum\limits_{\substack{i=1}}^{r+1}\binom{r}{i}a_i(r+2-i)A_\alpha^{r-i}B_\alpha^i[r+p+1-i]_{r-i}[i]_i\right)X^{p+1}\\
			&  & +a_{r+2}(r+2)B_\alpha^r[r+p+1]_rY^{p+1}
			     +\left(\sum\limits_{\substack{i=1}}^{r+1}\binom{r}{i-2}a_i i A_\alpha^{r-i+2}B_\alpha^{i-2}[r+2-i]_{r+2-i}[i+p-1]_{i-2}\right)Y^{p+1}\\
			&  & +\left(\sum\limits_{\substack{i=1}}^{r+1}\binom{r}{i-1}a_i(r+2-i)! i!A_\alpha^{r-i+1}B_\alpha^{i-1}\right)\left(X^pY+XY^p\right).
                \end{eqnarray*}}
\!\!\!                Note that since $A_\alpha, B_\alpha\in \mathbb{F}_{p^2},$ we have
		\begin{eqnarray*}
                  X^{p+1}\Big\vert_{(A_\alpha, B_\alpha)}^{(A_\alpha^p, B_\alpha^p)}=A_\alpha^{p(p+1)}-A_\alpha^{p+1}=0 \quad 
		\text{and} \quad
                  Y^{p+1}\Big\vert_{(A_\alpha, B_\alpha)}^{(A_\alpha^p, B_\alpha^p)}=B_\alpha^{p(p+1)}-B_\alpha^{p+1}=0.
                \end{eqnarray*}
                Also,
		\[X^pY+XY^p \Big\vert_{(A_\alpha, B_\alpha)}^{\left(A_\alpha^p, B_\alpha^p\right)}=A_\alpha B_\alpha^p+A_\alpha^pB_\alpha-A_\alpha^pB_\alpha-A_\alpha B_\alpha^p=0.\]
		Thus we have,
		\[\nabla_{\alpha}^r(D(P))\Big\vert_{\left(a+c\alpha, b+d\alpha\right)}^{\left((a+c\alpha)^p, (b+d\alpha)^p\right)}=\nabla_{\alpha}^r(D(P))\Big\vert_{(A_\alpha, B_\alpha)}^{(A_\alpha^p, B_\alpha^p)}
                =0,\]
		proving the claim, and so $P\otimes Q\in \ker\psi.$ Thus $D(V_{r+2})\otimes V_{p-1}\subset \ker\psi.$
		
		Next, we show that $\ker\psi\subset D(V_{r+2})\otimes V_{p-1}$.
		We prove this inclusion by changing $r$ to $r-2,$ i.e., we show
                \[\ker\psi\subset D(V_{r})\otimes V_{p-1},\] for $2\leq r\leq p-1,$ where  \[\psi:V_{r+p-1}\otimes V_{p-1}\rightarrow {\rm ind}_{T(\mathbb{F}_p)}^{G(\mathbb{F}_p)}\omega_2^{r}\]
		such that $\psi(P\otimes Q)=\psi_{P\otimes Q}$ ,where $\psi_{P\otimes Q}: G(\mathbb{F}_p)\rightarrow \mathbb{F}_{p^2}$ is defined by
		\begin{eqnarray}
                    \label{explicit iso cuspidal}
                    \psi_{P\otimes Q}\left(\left(\begin{matrix}
			a & b\\
			c & d
                      \end{matrix}\right)\right)=\nabla_{\alpha}^{r-2}(P)\Big\vert_{\left(a+c\alpha, b+d\alpha\right)}^{\left((a+c\alpha)^p, (b+d\alpha)^p\right)}\cdot Q\left(\left(a+c\alpha\right)^p, \left((b+d\alpha\right)^p\right)
               \end{eqnarray}
		for all $\left(\begin{smallmatrix}
			a & b\\
			c & d
		\end{smallmatrix}\right)\in G(\mathbb{F}_p).$ 
              We use the latter notation since we reduce the rest of the proof of Theorem~\ref{base case cuspidal}
              to some determinant computations in a longer arXiv version of this paper \url{https://arxiv.org/pdf/2308.10246} where $r$ is replaced by $r-2$.		
              
              Let
              \begin{eqnarray}
                \label{poly}
                P\otimes Q=\sum\limits_{\substack{i=0}}^{r+p-1}\sum\limits_{\substack{j=0}}^{p-1}a_{i,j}X^{r+p-1-i}Y^i\otimes S^{p-1-j}T^j\in \ker\psi,
              \end{eqnarray}
              where $V_{p-1}$ is modeled on polynomials in $S$,$T$. Then we have
		\begin{equation}\label{condition}
			\sum\limits_{\substack{j=0}}^{p-1}\left(\sum\limits_{\substack{i=0}}^{r+p-1}a_{i,j}\left(\left(A_\alpha\dfrac{\partial}{\partial X}+B_\alpha\dfrac{\partial}{\partial Y}\right)^{r-2}\left(X^{r+p-1-i}Y^i\right)\Big\vert_{(A_\alpha, B_\alpha)}^{(A_\alpha^p, B_\alpha^p)}\right)\right)A_\alpha^{p^2-p-jp}B_\alpha^{jp}=0.
		\end{equation}

		Now, 
{\small		\begin{align}
		    \label{inner sum}
			&\notag\sum\limits_{\substack{i=0}}^{r+p-1}a_{i,j}\left(\left(A_\alpha\dfrac{\partial}{\partial X}+B_\alpha\dfrac{\partial}{\partial Y}\right)^{r-2}\left(X^{r+p-1-i}Y^i\right)\Big\vert_{(A_\alpha, B_\alpha)}^{(A_\alpha^p, B_\alpha^p)}\right)\\
			&\quad =\notag\sum\limits_{\substack{i=0}}^{r+p-1}a_{i,j}\left(\sum\limits_{\substack{k=0}}^{r-2}\binom{r-2}{k}A_\alpha^{r-2-k}B_\alpha^k\dfrac{\partial^{r-2}}{\partial X^{r-2-k}\partial Y^k}\left(X^{r+p-1-i}Y^i\right)\Big\vert_{(A_\alpha, B_\alpha)}^{(A_\alpha^p, B_\alpha^p)}\right)\\
			&\quad =\notag\sum\limits_{\substack{i=0}}^{r+p-1}a_{i,j}\left(\sum\limits_{\substack{k=0}}^{r-2}\binom{r-2}{k}A_\alpha^{r-2-k}B_\alpha^k[r+p-1-i]_{r-2-k}[i]_k\left(X^{p+1-i+k}Y^{i-k}\right)\Big\vert_{(A_\alpha, B_\alpha)}^{(A_\alpha^p, B_\alpha^p)}\right)\\
			&\quad =\notag\sum\limits_{\substack{i=0}}^{r+p-1}a_{i,j}\left(\sum\limits_{\substack{k=0}}^{r-2}\binom{r-2}{k}[r+p-1-i]_{r-2-k}[i]_k\left(A
			_\alpha^{r-ip+kp+p-k-1}B_\alpha^{ip-kp+k}-A_\alpha^{r+p-1-i}B_\alpha^i\right)\right)\\
			&\quad =\sum\limits_{\substack{i=1}}^{r+p-2}a_{i,j}\left(\sum\limits_{\substack{k=0}}^{r-2}\binom{r-2}{k}[r+p-1-i]_{r-2-k}[i]_k\left(A
			_\alpha^{r-ip+kp+p-k-1}B_\alpha^{ip-kp+k}-A_\alpha^{r+p-1-i}B_\alpha^i\right)\right).
		\end{align}}
\!\!\!                In the last equality, we dropped the terms for $i=0, r+p-1$ as these are zero. Indeed, for $i=0,$ since $[i]_k=0$ for $k\neq 0,$ only
                the $k=0$ term survives in the sum over $k$ for which  
		\[A
		_\alpha^{r-ip+kp+p-k-1}B_\alpha^{ip-kp+k}-A_\alpha^{r+p-1-i}B_\alpha^i=A_\alpha^{r+p-1}-A_\alpha^{r+p-1}=0.\]
		Similarly, one can check that the terms for $i=r+p-1$ are also zero.
		The expression \eqref{inner sum}
		\begin{equation}\label{coefficient}
			= \sum\limits_{\substack{i=1}}^{r+p-2}a_{i,j}\left(\left(\sum\limits_{\substack{k=0}}^{r-2}\binom{r-2}{k}[r+p-1-i]_{r-2-k}[i]_kA
			_\alpha^{r-ip+kp+p-k-1}B_\alpha^{ip-kp+k}\right)-(r-1)!A_\alpha^{r+p-1-i}B_\alpha^i\right),
		\end{equation}
		since $A_\alpha^{r+p-1-i}B_\alpha^i$ is independent of $k$ and, by \eqref{sum},
		\[\sum\limits_{\substack{k=0}}^{r-2}\binom{r-2}{k}[r+p-1-i]_{r-2-k}[i]_k=[r+p-1]_{r-2}=(r-1)!\mod p.\]
		
		Consider now the expression \eqref{coefficient} as a polynomial in the variables $a_{i, j}$.
                We compute the coefficients of $a_{i,j}$ and $a_{i+p-1,j}$ in this polynomial for $1\leq i\leq r-1$. \\

                \noindent \textbf{Coefficient of $a_{i,j}$:} Note that for $1\leq i\leq r-1$ and $0\leq k\leq r-2,$ we have
		$[r+p-1-i]_{r-2-k}= 0$ 
                for $k<i-1,$ and $[i]_k=0$ for $i<k.$ Thus the coefficient of $a_{i,j}$ in \eqref{coefficient} 
{\small		\begin{eqnarray*}
			& = & \binom{r-2}{i-1}[r+p-1-i]_{r-i-1}[i]_{i-1}A_\alpha^{r-i}B_\alpha^{i+p-1} 
			       + \binom{r-2}{i}[r+p-1-i]_{r-i-2}[i]_iA_\alpha^{r+p-1-i}B_\alpha^{i}-(r-1)!A_\alpha^{r+p-1-i}B_\alpha^i\\
			& = & \binom{r-2}{i-1}(r-i-1)!i! A_\alpha^{r-i}B_\alpha^{i+p-1}
			      + \left(\binom{r-2}{i}(r-i-1)!i!-(r-1)!\right)A_\alpha^{r+p-1-i}B_\alpha^i\\
                        & = & i\left((r-2)!A_\alpha^{r-i}B_\alpha^{i+p-1}-(r-2)!A_\alpha^{r-i+p-1}B_\alpha^i\right).
		\end{eqnarray*}} 
		
		\noindent \textbf{Coefficient of $a_{i+p-1,j}$:} Similarly, in \eqref{coefficient}, the coefficient of $a_{i+p-1,j}$ is
{\small		\begin{align*}
			&\left(\sum\limits_{\substack{k=0}}^{r-2}\binom{r-2}{k}[r-i]_{r-2-k}[i+p-1]_kA
			_\alpha^{r-ip-p^2+p+kp+p-k-1}B_\alpha^{ip+p^2-p-kp+k}\right)-(r-1)!A_\alpha^{r-i}B_\alpha^{i+p-1}\\
			&\quad =\binom{r-2}{i-2}[r-i]_{r-i}[i+p-1]_{i-2}A
			_\alpha^{r-i}B_\alpha^{i+p-1} 
			  +\binom{r-2}{i-1}[r-i]_{r-i-1}[i+p-1]_{i-1}A
			_\alpha^{r-i+p-1}B_\alpha^{i}-(r-1)!A_\alpha^{r-i}B_\alpha^{i+p-1}\\
			&\quad =\left(\binom{r-2}{i-2}(r-i)!(i-1)!-(r-1)!\right)A_\alpha^{r-i}B_\alpha^{i+p-1}
                          +\binom{r-2}{i-1}(r-i)!(i-1)!A_\alpha^{r-i+p-1}B_\alpha^{i}\\
			&\quad =-(r-i)\left((r-2)!A_\alpha^{r-i}B_\alpha^{i+p-1}-(r-2)!A_\alpha^{r-i+p-1}B_\alpha^i\right).
		\end{align*}}
\!\!\!		The first equality holds because $[r-i]_{r-2-k}=0$ for $k<i-2$ and $[i+p-1]_k=0$ for $k>i-1$.
                The second equality follows because
		$[i+p-1]_{i-2}=(i-1)!=[i+p-1]_{i-1}$. \\
		
		For $1\leq i\leq r-1,$ substituting the coefficients of $a_{i,j}$ and $a_{i+p-1, j}$ in \eqref{coefficient}, we have 
{\small		\begin{align*}
			&  \sum\limits_{\substack{i=0}}^{r+p-1}a_{i,j}\left(\left(A_\alpha\dfrac{\partial}{\partial X}+B_\alpha\dfrac{\partial}{\partial Y}\right)^{r-2}\left(X^{r+p-1-i}Y^i\right)\Big\vert_{(A_\alpha, B_\alpha)}^{(A_\alpha^p, B_\alpha^p)}\right)\\
			&\quad =\sum\limits_{\substack{i=1}}^{r-1}\left(ia_{i,j}-(r-i)a_{i+p-1, j}\right)(r-2)!\left(A_\alpha^{r-i}B_\alpha^{i+p-1}-A_\alpha^{r-i+p-1}B_\alpha^i\right)\\
			&\quad {\hspace{12pt}}+	\sum\limits_{\substack{i=r}}^{p-1}a_{i,j}\left(\left(\sum\limits_{\substack{k=0}}^{r-2}\binom{r-2}{k}[r+p-1-i]_{r-2-k}[i]_kA
			_\alpha^{r-ip+kp+p-k-1}B_\alpha^{ip-kp+k}\right)-(r-1)!A_\alpha^{r+p-1-i}B_\alpha^i\right)\\
			&\quad =\sum\limits_{\substack{i=1}}^{r-1}\left(ia_{i,j}-(r-i)a_{i+p-1, j}\right)(r-2)!\left(A_\alpha^{r-i}B_\alpha^{i+p-1}-A_\alpha^{r-i+p-1}B_\alpha^i\right)\\
			&\quad {\hspace{12pt}}+	\sum\limits_{\substack{i=r}}^{p-1}a_{i,j}\left(\left(\sum\limits_{\substack{k=0}}^{r-2}(r-2)!\binom{r+p-1-i}{r-2-k}\binom{i}{k}
                           A_\alpha^{r-ip+kp+p-k-1}B_\alpha^{ip-kp+k}\right)-(r-1)!A_\alpha^{r+p-1-i}B_\alpha^i\right).
		\end{align*}}
\!\!\!		Thus \eqref{condition} shows that the following linear combination of functions in the induced space vanishes:
{\small		\begin{align}\label{main} \begin{split}
			&\sum\limits_{\substack{j=0}}^{p-1}\left(\sum\limits_{\substack{i=1}}^{r-1}\left(ia_{i,j}-(r-i)a_{i+p-1, j}\right)\left(A_\alpha^{r-i}B_\alpha^{i+p-1}-A_\alpha^{r-i+p-1}B_\alpha^i\right)\right)A_\alpha^{p^2-jp-p}B_\alpha^{jp}\> + \\
			&\sum\limits_{\substack{j=0}}^{p-1}\left(\sum\limits_{\substack{i=r}}^{p-1}a_{i,j}\left(\left(\sum\limits_{\substack{k=0}}^{r-2}\binom{r+p-1-i}{r-2-k}\binom{i}{k}A
			_\alpha^{r-ip+kp+p-k-1}B_\alpha^{ip-kp+k}\right)-(r-1)A_\alpha^{r+p-1-i}B_\alpha^i\right)\right) \cdot A_\alpha^{p^2-jp-p}B_\alpha^{jp} \> =\> 0, \end{split}
		\end{align}}
\!\!\!		after dividing by $(r-2)!$.
                                
		For simplicity, we make the change of variables 
		\[X_{i,j}:= ia_{i,j}-(r-i)a_{i+p-1, j},\] for $1\leq i\leq r-1$ and $0\leq j\leq p-1,$ 
		and,
		\[X_{i, j}:=a_{i, j}~ \text{ and } ~Z_{i, k}:= \binom{r+p-1-i}{r-2-k}\binom{i}{k},\]
		for $r\leq i\leq p-1$, $0\leq j\leq p-1$, $0\leq k\leq r-2.$  The $X_{i,j}$ are 
                variables
                (and the $Z_{i,k}$ constants).

                We claim that all the $X_{i,j}$ must be $0$. The proof of this is computational. It will occupy the next few pages.
                To reduce the length of the proof we have referred to some determinant computations which are in the
                longer arXiv version of this paper \url{https://arxiv.org/pdf/2308.10246}. The first time reader may skip to the end of the proof of
                Theorem~\ref{base case cuspidal}.\\
		
		\noindent \textbf{Claim:} For $1\leq i\leq p-1$ and $0\leq j\leq p-1,$ we have $X_{i, j}=0$.
		
		\noindent We prove the claim. Write \eqref{main} as
{\small		\begin{align*}
			&\sum\limits_{\substack{j=0}}^{p-1}\left(\sum\limits_{\substack{i=1}}^{r-1}X_{i, j}\left(A_\alpha^{r+p^2-i-jp-p}B_\alpha^{i+jp+p-1}-A_\alpha^{r+p^2-i-jp-1}B_\alpha^{i+jp}\right)\right)+\\
			&\sum\limits_{\substack{j=0}}^{p-1}\left(\sum\limits_{\substack{i=r}}^{p-1}X_{i, j}\left(\left(\sum\limits_{\substack{k=0}}^{r-2}Z_{i,k}A
			_\alpha^{r+p^2-ip-jp+kp-k-1}B_\alpha^{ip+jp-kp+k}\right)-(r-1)A_\alpha^{r+p^2-1-i-jp}B_\alpha^{i+jp}\right)\right)=0.\notag
		\end{align*}}
\!\!\!		Collecting terms in the same congruence class $n$ of $i+j$ modulo $(p-1)$, we have 
{\small		\begin{align*}
                  \sum\limits_{\substack{n=1}}^{p-1}& 
                                                      \sum\limits_{\substack{1\leq i\leq r-1\\0\leq j\leq p-1\\i+j\equiv n \mod (p-1)}}X_{i, j}\left(A_\alpha^{r+p^2-i-jp-p}B_\alpha^{i+jp+p-1}-A_\alpha^{r+p^2-i-jp-1}B_\alpha^{i+jp}\right) 
                  +\\
                  \sum\limits_{\substack{n=1}}^{p-1}&
                                                      \sum\limits_{\substack{r\leq i\leq p-1\\0\leq j\leq p-1\\i+j\equiv n \mod (p-1)}}X_{i, j}\left(\left(\sum\limits_{\substack{k=0}}^{r-2}Z_{i,k}A
                  _\alpha^{r+p^2-ip-jp+kp-k-1}B_\alpha^{ip+jp-kp+k}\right)-(r-1)A_\alpha^{r+p^2-1-i-jp}B_\alpha^{i+jp}\right) 
                  \> 	= \> 0.
		\end{align*}}
			%
			%
				%
				%
				%
\!\!\!\! Write the $n$-th summand above as $\mathcal{B}_n$ for $1 \leq n \leq p-1$. 
                Recall that we are viewing $A_\alpha$ and $B_\alpha$ as functions on $G({\mathbb F}_p)$, cf. \eqref{def of A alpha and B alpha base case}.
                By inspection, each of the functions in $\mathcal{B}_n$ 
                is of the form $A_\alpha^{r+p^2-1-l}B_\alpha^l$ for $l\equiv n$ modulo $(p-1)$.
                %
                Note that $A_\alpha^{r+p^2-1-l}B_\alpha^l$ 
                is an element of ${\rm ind}_{T(\mathbb{F}_p)}^{G(\mathbb{F}_p)}\omega_2^{r}$.
                 Using \eqref{flip} and \eqref{flop}, we may assume each of these functions belong to the basis ${\mathcal B}$,
                noting that these operations preserve the congruence class $n$. 
                By the linear independence of the basis $\mathcal{B}$ and the vanishing of the sum of the ${\mathcal B}_n$,
                we conclude that each $\mathcal{B}_n=0$. Thus fixing $1\leq n \leq p-1$ we have:					
{\small				\begin{align}\label{main1}
					&	\sum\limits_{\substack{1\leq i\leq r-1\\0\leq j\leq p-1\\i+j\equiv n \mod (p-1)}}X_{i, j}\left(A_\alpha^{r+p^2-i-jp-p}B_\alpha^{i+jp+p-1}-A_\alpha^{r+p^2-i-jp-1}B_\alpha^{i+jp}\right)+ \notag \\
					&\sum\limits_{\substack{r\leq i\leq p-1\\0\leq j\leq p-1\\i+j\equiv n \mod (p-1)}}X_{i, j}\left(\left(\sum\limits_{\substack{k=0}}^{r-2}Z_{i,k}A
					_\alpha^{r+p^2-ip-jp+kp-k-1}B_\alpha^{ip+jp-kp+k}\right)-(r-1)A_\alpha^{r+p^2-1-i-jp}B_\alpha^{i+jp}\right) = 0.
				\end{align}} 
                              
				
				\noindent  The possible pairs of $(i, j)$ such that $i+j\equiv n \mod (p-1)$ for $1\leq i\leq p-1$ and $0\leq j\leq p-1$ are:                              
				\begin{eqnarray}
                                    \label{i,j}
                                    (i, j)=\begin{cases}
					     \left(i, n-i\right) , &\text{if}~1\leq i\leq n,\\
					     \left(i, p-1+n-i\right), &\text{if}~n\leq i\leq p-1.
				           \end{cases}
                                \end{eqnarray} 
                                We analyze three cases: \\
				
\noindent        		\textbf{Case 1:} Suppose $1\leq n\leq r-1.$ 				
				Using \eqref{i,j}, equation \eqref{main1} becomes
				\begin{align}\label{case1}
					&\sum\limits_{\substack{i=1}}^nX_{i, n-i}\left(A_\alpha^{r+p^2-i-(n-i)p-p}B_\alpha^{i+(n-i)p+p-1}-A_\alpha^{r+p^2-i-(n-i)p-1}B_\alpha^{i+(n-i)p}\right) \notag \\
					&{\hspace{12pt}}+\sum\limits_{\substack{i=n}}^{r-1}X_{i, p-1+n-i}\left(A_\alpha^{r-i-np+ip}B_\alpha^{i+p^2+np-ip-1}-A_\alpha^{r-i-np+ip+p-1}B_\alpha^{i+p^2-p+np-ip}\right)\notag\\
					&+\sum\limits_{\substack{i=r}}^{p-1} X_{i,p-1+n-i}\left(\sum\limits_{\substack{k=0}}^{r-2}Z_{i,k}A
					_\alpha^{r+(p-1)k-np+p-1}B_\alpha^{p^2-p+np-(p-1)k}\right)\notag\\
					&{\hspace{12pt}}-\sum\limits_{\substack{i=r}}^{p-1}(r-1)X_{i,p-1+n-i}A_\alpha^{r-i-np+ip+p-1}B_\alpha^{i+p^2-p+np-ip}=0.
				\end{align}
				
				Looking at the indices occurring in $X_{i, j}$ in the above equation, we note for each $i \neq n$, there is a unique
                                $j\neq 0.$ But for $i=n,$ we get both $j=0$ and $j=p-1.$ For convenience, we drop the
                                index $j$ for $j\neq 0$: we write 
				\begin{eqnarray}
                                  \label{convention}
                                  X_{i, j} \> \text{ as } \>
                                      \begin{cases}
					X_i, & \text{if}~j\neq 0,\\
					X_0, & \text{if}~j=0.
                                      \end{cases}
                               \end{eqnarray}

                              Now, there is no flip or flop in the first summand of the above equation.
                              If $i=n,$ then there is one flip in the first component of the second summand and one flop in its second component.
                              Also, for $i=n+1$, there is one flop in the first component of the second summand if $n \leq r-2$ (and the $i = n+1$ term
                              is not there if $n = r-1$). Finally in the third summand,
                              there are flips for $0\leq k\leq n-1$ and there is one flop for $k=n$. 
			      Apart from these there are no flips and flops appearing in the above equation. 
				
                              Assume that $1 \leq n \leq r-2$. Changing the flips and flops appearing in \eqref{case1} to functions
                              in $\mathcal{B}$ and looking at the coefficient of $A_\alpha^{r+p^2-1-(n+l(p-1))}B_\alpha^{n+l(p-1)}$ for $0\leq l\leq p-1,$ we get the following system of equations:
                              {\small
                                        \begin{align} 
					\label{equations for small values of n in F_p case}
					&2X_n-X_{n+1}+\sum\limits_{\substack{i=r}}^{p-1}\left(Z_{i, n-1}-Z_{i, n}\right)X_i-X_0=0, & \text{if}~l=0, \nonumber \\ 
					&X_n-X_{n+1}+\sum\limits_{\substack{i=r}}^{p-1}\left(Z_{i, n-2}-Z_{i, n}\right)X_i-X_{n-1}+X_0=0, &\text{if}~l=1, \nonumber \\
					&X_n-X_{n+1}+\sum\limits_{\substack{i=r}}^{p-1}\left(Z_{i, n-l-1}-Z_{i, n}\right)X_i+X_{n-l+1}-X_{n-l}=0, &\text{if}~2\leq l\leq n-1, \nonumber\\
					&X_n-X_{n+1}-\sum\limits_{\substack{i=r}}^{p-1}Z_{i, n}X_i+X_{n-l+1}=0, &\text{if}~l=n, \nonumber \\
					&X_n-X_{n+1}-\sum\limits_{\substack{i=r}}^{p-1}Z_{i, n}X_i-(r-1)X_{n+p-l}=0, &\text{if}~n+1\leq l\leq p-r+n, \nonumber \\
					&X_n-X_{n+1}-\sum\limits_{\substack{i=r}}^{p-1}Z_{i, n}X_i-X_{n+p-l}=0, &\text{if}~ l=p-r+n+1,  \nonumber \\
					&X_n-X_{n+1}+\sum\limits_{\substack{i=r}}^{p-1}\left(Z_{i, p+n-l}-Z_{i, n}\right)X_i+X_{n+p-l+1}-X_{n+p-l}=0, &\text{if}~p-r+n+2\leq l\leq p-1.
				\end{align}}

\noindent \textbf{Case 2:} Suppose $r\leq n \leq p-2.$
	In this case, by \eqref{main1} we have
	
        {\small
                \begin{align*}
		&\sum\limits_{\substack{i=1}}^{r-1}X_{i, n-i }\left(A_\alpha^{r+p^2-i-(n-i)p-p}B_\alpha^{i+(n-i)p+p-1}-A_\alpha^{r+p^2-i-(n-i)p-1}B_\alpha^{i+(n-i)p}\right)\\
		&{\hspace{12pt}}+\sum\limits_{\substack{i=r}}^{n}X_{i, n-i}\left(\sum\limits_{\substack{k=0}}^{r-2}Z_{i,k}A
						_\alpha^{r+(p-1)k-np+p^2-1}B_\alpha^{np-(p-1)k}\right)\\
		&\quad {\hspace{12pt}}-\sum\limits_{\substack{i=r}}^{n}(r-1)
		X_{i, n-i}A_\alpha^{r+p^2-1-i-np+ip}B_\alpha^{i+np-ip}\\
		&{\hspace{12pt}} +\sum\limits_{\substack{i=n}}^{p-1}
		X_{i, p-1+n-i}\left(\sum\limits_{\substack{k=0}}^{r-2}Z_{i,k}A
						_\alpha^{r+(p-1)k-np+p-1}B_\alpha^{p^2-p+np-(p-1)k}\right)\\
		&\quad {\hspace{12pt}}-\sum\limits_{\substack{i=n}}^{p-1}(r-1)
		X_{i, p-1+n-i}A_\alpha^{r-i-np+ip+p-1}B_\alpha^{i+p^2-p+np-ip}=0.
		\end{align*}}

                In the second but last sum above there are flips for all $0 \leq k\leq r-2.$ Also, there is a flop in the last sum at $i = n .$ As before, by
                changing  the flips and flops appearing in the above equation to functions in $\mathcal{B}$ and then looking at the coefficient of
                $A_\alpha^{r+p^2-1-(n+l(p-1))}B_\alpha^{n+l(p-1)}$ for $0\leq l\leq p-1$ and using the convention \eqref{convention} for the variables,
                we get the following system of equations:
					
                {\small
                \begin{align}
		\label{equations for large values of n in F_p case}
			&(r-1)X_n-(r-1)X_0=0,       & \text{if}~l=0, \nonumber \\ 
			&(r-1)X_n-(r-1)X_{n-l}=0, &\text{if}~1\leq l\leq n-r, \nonumber \\
			&\left((r-1)+
			Z_{n, n-l-1}\right)X_n-X_{n-l}+\sum\limits_{\substack{i=n+1}}^{p-1}Z_{i, n-l-1}X_i=0,                          &\text{if}~l=n-r+1, \nonumber \\
			&\left((r-1)+Z_{n, n-l-1}\right)X_n+\sum\limits_{\substack{i=r}}^{n-1}Z_{i, n-l}X_i \nonumber \\
			&\quad -X_{n-l}+X_{n-l+1}+\sum\limits_{\substack{i=n+1}}^{p-1}
			Z_{i, n-l-1}X_i+Z_{n, n-l}X_0=0, 
			                                                &\text{if}~n-r+2\leq l\leq n-1, \nonumber\\
			& (r-1)X_n+X_{n-l+1}+\sum\limits_{\substack{i=r}}^{n-1}Z_{i, n-l}X_i+Z_{n, n-l}X_0=0,                             &\text{if}~ l=n, \nonumber \\
			&(r-1)X_n-(r-1)X_{n-l+p}=0, &\text{if}~n+1\leq l\leq p-1.
	\end{align} }

	\noindent \textbf{Case 3:} Suppose $n=p-1.$
	From \eqref{main1} we have,
                        \begin{align}
			\label{n=p-1 case}
				&\sum\limits_{\substack{i=1}}^{r-1}X_{i, p-1-i}\left(A_\alpha^{r-i+ip}B_\alpha^{i-ip+p^2-1}-A_\alpha^{r+p-1-i+ip}B_\alpha^{i+p^2-p-ip}\right) \nonumber \\ 
				&{\hspace{12pt}}+\sum\limits_{\substack{i=r}}^{p-1}X_{i, p-1-i}\left(\sum\limits_{\substack{k=0}}^{r-2}Z_{i,k}A
						_\alpha^{r+(p-1)k+p-1}B_\alpha^{p^2-p-(p-1)k}\right) \nonumber \\
				&\quad {\hspace{12pt}}-\sum\limits_{\substack{i=r}}^{p-1}(r-1)
				X_{i, p-1-i}A_\alpha^{r-i+ip+p-1}B_\alpha^{i+p^2-p-ip} \notag \\
				&{\hspace{12pt}}+X_{p-1, p-1} \left( \sum\limits_{\substack{k=0}}^{r-2}Z_{ p-1 ,k} A_\alpha^{r+(p-1)k-(p-1)^2}B_\alpha^{2p^2-2p-(p-1)k} \right) \notag \\
				&\quad {\hspace{12pt}}-(r-1)X_{p-1, p-1}A_\alpha^{r}B_\alpha^{p^2-1}=0.
		\end{align} 

                There are flops in the first component of the first sum for $i = 1$
                and in the second sum for $k = 0$. Also, there are flips in the second last line above for all
                $0\leq k \leq r-2$ and in the last line. Changing the flips and flops appearing in the above equation to functions in $\mathcal{B}$,
                looking at the coefficient of $A_\alpha^{r+p^2-1-l(p-1)}B_\alpha^{l(p-1)}$ for $0\leq l\leq p-1$ and using the convention \eqref{convention} for the variables, we have
	the following system of equations:
        {\small
	  \begin{align}
		\label{equations for n=p-1 case}
			& -X_1-\sum\limits_{\substack{i=r}}^{p-2}Z_{i, 0}X_i-(r-1)X_{p-1}-Z_{p-1, 0}X_0=0, & \text{if}~l=0, \nonumber \\ 
			& -X_1-\sum\limits_{\substack{i=r}}^{p-2}Z_{i, 0}X_i-(r-1)X_0-Z_{p-1, 0}X_0=0, & \text{if}~l=1, \nonumber \\
			& -X_1-\sum\limits_{\substack{i=r}}^{p-2}Z_{i, 0}X_i-(r-1)X_{p-l}-Z_{p-1, 0}X_0=0, & \text{if}~2\leq l\leq p-r, \nonumber \\
			& -X_1-\sum\limits_{\substack{i=r}}^{p-2}Z_{i, 0}X_i-X_{p-l}+Z_{p-1, r-2}X_{p-1}-Z_{p-1, 0}X_0=0, & \text{if}~l=p-r+1,  \nonumber \\
			& -X_1-\sum\limits_{\substack{i=r}}^{p-2}\left(Z_{i, 0}-Z_{i, p-l}\right) X_i+X_{p-l+1}-X_{p-l} \nonumber\\
			&\qquad +Z_{ p-1, p-l-1}X_{p-1}+Z_{ p-1, p-l}X_0-Z_{p-1, 0}X_0=0, & \text{if}~p-r+2\leq l\leq p-1.
	\end{align}}

        Let $M$ be the coefficient matrix of the above systems of equations in Cases 1, 2 and 3, respectively.
        A computation shows that $\det (M) \neq 0$. In fact, one can give
        a formula for $\det(M)$ in each case 
        but we do not need it, so to keep this paper a reasonable size, we omit it
        (details about this and other omitted arguments in this current abridged version of the paper may be
        found in an earlier version
        of the paper on the arXiv at  \url{https://arxiv.org/pdf/2308.10246.pdf}).
        It follows that in each case $X_i = 0$ for $0 \leq i \leq p-1$. 
	Thus we have $X_{i,j}=0$ for all $1\leq i\leq p-1$ and $0\leq j\leq p-1$, proving the Claim. \\

        Thus, for all $0\leq j\leq p-1$, we have
					\[ia_{i,j}-(r-i)a_{i+p-1, j}=0\] for $1\leq i\leq r-1$
					and \[a_{i,j}=0\] for $r\leq i\leq p-1$, where $a_{i,j}$ are
                                        the coefficients of $P \otimes Q$ in \eqref{poly}.

        Using these relations, we have
        \begin{align*}
	P\otimes Q&=\sum\limits_{\substack{i=0}}^{r+p-1}\sum\limits_{\substack{j=0}}^{p-1}a_{i,j}X^{r+p-1-i}Y^i\otimes S^{p-1-j}T^j\\
	&=\sum\limits_{\substack{j=0}}^{p-1}a_{0,j}X^{r+p-1}\otimes S^{p-1-j}T^j+\sum\limits_{\substack{j=0}}^{p-1}a_{r+p-1,j}Y^{r+p-1}\otimes S^{p-1-j}T^j\\
	&\quad {\hspace{12pt}}+\sum\limits_{\substack{i=1}}^{r-1}\sum\limits_{\substack{j=0}}^{p-1}a_{i,j}X^{r+p-1-i}Y^i\otimes S^{p-1-j}T^j+\sum\limits_{\substack{i=1}}^{r-1}\sum\limits_{\substack{j=0}}^{p-1}a_{i+p-1,j}X^{r-i}Y^{i+p-1}\otimes S^{p-1-j}T^j\\
	&=\sum\limits_{\substack{j=0}}^{p-1}a_{0,j}X^{r+p-1}\otimes S^{p-1-j}T^j+\sum\limits_{\substack{j=0}}^{p-1}a_{r+p-1,j}Y^{r+p-1}\otimes S^{p-1-j}T^j\\
	&\quad {\hspace{12pt}}+\sum\limits_{\substack{j=0}}^{p-1}\sum\limits_{\substack{i=1}}^{r-1}a_{i,j}\dfrac{1}{(r-i)}\left((r-i)X^{r+p-1-i}Y^i+iX^{r-i}Y^{i+p-1}\right)\otimes S^{p-1-j}T^j.
	\end{align*} 

	Each term in the last equality belongs to $D(V_r)\otimes V_{p-1},$ and hence, so does $P\otimes Q.$
        Thus $\ker\psi\subset D(V_r)\otimes V_{p-1}$ and so equality holds.
        Changing $r$ back to $r+2,$ we have $\ker\psi=D(V_{r+2})\otimes V_{p-1}$. 
        
        Thus $\psi$ induces an injective map $\dfrac{V_{r+p+1}}{D(V_{r+2})}\otimes V_{p-1}\rightarrow {\rm ind}_{T(\mathbb{F}_p)}^{G(\mathbb{F}_p)}\omega_2^{r+2}$. Since both sides
        have the same dimension $p^2-p$, we conclude that the induced map is an isomorphism. This finally proves Theorem~\ref{base case cuspidal}.
        %
					\end{proof}

          We have proved Theorem \ref{cuspidal} for $2\leq r\leq p-1.$ One might wonder what happens for boundary values of $r$.
          Theorem \ref{cuspidal} is also true if $r = 1$, though the definition of the map $\psi$ needs to be changed a bit.
	\begin{proposition}
		\label{r = 1}
		We have
		\[\dfrac{V_p}{D(V_1)} \otimes V_{p-1} \simeq {\rm ind}_{T(\mathbb{F}_p)}^{G(\mathbb{F}_p)}\omega_2.\]
	\end{proposition}
	
	\begin{proof}
	Define  $\psi :  V_p \otimes V_{p-1}\rightarrow  {\rm ind}_{T(\mathbb{F}_p)}^{G(\mathbb{F}_p)}\omega_2$ by
	\[P \otimes Q \mapsto \left( \psi_{P\otimes Q} : \left(\begin{matrix}
		                                                       	a & b\\
		                                                        c & d
	                                                 \end{matrix}\right)
   \mapsto D(P)\Big\vert_{(0,0)}^{(A_\alpha^p, B_\alpha^p)} \cdot Q\Big\vert_{(0,0)}^{(A_\alpha^p, B_\alpha^p)}\right).\]
    One checks that $\psi$ is a well-defined $G(\mathbb{F}_p)$-linear map and $\ker \psi = (\ker D \cap V_p) \otimes V_{p-1}$,
    which by \cite[Proposition 3.3]{red10} (which is due to Fakhruddin), equals
    \[ (\mathbb{F}_p[ X^p, Y^p, \theta] \cap V_p) \otimes V_{p-1}= ({\mathbb F}_p\text{-span of } X^p, Y^p) \otimes V_{p-1}= D(V_1) \otimes V_{p-1}.\]
    By comparing dimensions, the isomorphism in Proposition \ref{r = 1} follows.
     \end{proof}
    
    \begin{remark}
    \label{r=p-2,p-1}
     In fact, Proposition \ref{r = 1} is true `without tensoring with $V_{p-1}$'. That is, Reduzzi's result \eqref{reduzzi} even
    holds for $r = 1$. Indeed, one has
    \[\dfrac{V_p}{D(V_1)} \simeq V_{p-2} \otimes \det \simeq  \overline{\Theta(\omega_2)}, \]    
    by Diamond \cite[Proposition 1.3]{dia07} (this does not use crystalline cohomology, see also the material around Prasad
    \cite[Lemma 4.2]{dp10}
    for a survey: in fact, the reduction mod $p$ of the complex cuspidal representation $\Theta(\omega_2^r)$ for
    $1 \leq r \leq p-1$ of $G(\mathbb{F}_p)$ is irreducible if and only if $r =1$).
    On the other hand, if $r=p$, then Theorem \ref{cuspidal} is false for dimension reasons.
    Since ${\rm ind}_{T(\mathbb{F}_p)}^{G(\mathbb{F}_p)}\omega_2^p \simeq {\rm ind}_{T(\mathbb{F}_p)}^{G(\mathbb{F}_p)}\omega_2$,
    the right hand side of the isomorphism in the theorem reduces to the case $r = 1$. As for the left side, one easily checks
    \[\dfrac{V_{2p-1}}{D(V_p)}\simeq {\rm ind}_{B(\mathbb{F}_p)}^{G(\mathbb{F}_p)}d\] 
    is a principal series representation. Similarly, if $r=p+1,$ then $\omega_2^{p+1}$ is self-conjugate and the induction on the right side
    is not as interesting, whereas on the left side one checks 
	 \[\dfrac{V_{2p}}{D(V_{p+1})} \simeq V_{p-1}\otimes \det\] is a twist of the mod $p$ Steinberg representation. 
	 \end{remark}
%
%
%

    In view of Proposition \ref{r = 1} and Remark \ref{r=p-2,p-1}, we see  Theorem~\ref{base case cuspidal} holds for $-1 \leq r \leq p-3$ but not for $r = p-2$, $p-1$.
    However, by twisting, the theorem may be extended to the following higher symmetric powers:
    
	\begin{corollary}
		\label{bigger range}
	 If $-1\leq r\leq p-3-k$ for $0\leq k\leq p-2,$ then
					\[\dfrac{V_{r+(k+1)(p+1)}}{D(V_{r+2+k(p+1)})}\otimes V_{p-1}\simeq {\rm ind}_{T(\mathbb{F}_p)}^{G(\mathbb{F}_p)}\omega_2^{r+2+k(p+1)}.\]
	\end{corollary}
	\begin{proof}
		\label{proof of bigger range}
                We first show that \[\dfrac{V_{r+p+1}\otimes{\det}^k}{D(V_{r+2})\otimes{\det}^k}\simeq \dfrac{V^{(k)}_{r+(k+1)(p+1)}}{D(V^{(k)}_{r+2+k(p+1)})}\simeq \dfrac{V_{r+(k+1)(p+1)}}{D(V_{r+2+k(p+1)})}\]
         for $0\leq k\leq p-2$.         
         Define $\pi:V_{r+p+1}\otimes{\det}^k\rightarrow\frac{V^{(k)}_{r+(k+1)(p+1)}}{D(V^{(k)}_{r+2+k(p+1)})}$ by sending $P$ 
         to $\theta^kP$ for $P\in V_{r+p+1}.$ Let $Q\in V_{r+2}.$ Then $\pi(D(Q))=
                  \theta^k D(Q)=D(\theta^kQ),$
         where the last equality follows because $D(\theta)=-X^pY^p+X^pY^p=0.$
        Thus we have $\pi(D(V_{r+2})\otimes{\det}^k)\subset D(V^{(k)}_{r+2+k(p+1)}),$ and hence,
        \[\overline{\pi}:\dfrac{V_{r+p+1}\otimes{\det}^k}{D(V_{r+2})\otimes{\det}^k}\rightarrow \dfrac{V^{(k)}_{r+(k+1)(p+1)}}{D(V^{(k)}_{r+2+k(p+1)})}\] is a surjection.
        Also, both sides have dimension $p-1,$ so $\overline{\pi}$ is an isomorphism.
	
	For the second isomorphism, consider the composition
        \begin{eqnarray}
          \label{composition}
          V^{(k)}_{r+(k+1)(p+1)}\hookrightarrow V_{r+(k+1)(p+1)}\twoheadrightarrow \dfrac{V_{r+(k+1)(p+1)}}{D(V_{r+2+k(p+1)})}.
        \end{eqnarray}
	In the above, the first map is the natural inclusion map and the second one is the natural surjection map. Note that,
        the kernel of the map \eqref{composition}
        is $V^{(k)}_{r+(k+1)(p+1)}\cap D(V_{r+2+k(p+1)}).$ We show that 			  	\[V^{(k)}_{r+(k+1)(p+1)}\cap D(V_{r+2+k(p+1)})=D(V^{(k)}_{r+2+k(p+1)}).\] 
	Clearly, $D(V^{(k)}_{r+2+k(p+1)})\subset V^{(k)}_{r+(k+1)(p+1)}\cap D(V_{r+2+k(p+1)}).$ 
        The other containment is trivially true for $k = 0$. To establish it for $1 \leq k \leq p-2$, we need the following
        lemma which is easily proved by checking the two conditions in \cite[Lemma 2.7]{GV22}.
        \begin{lemma}
		\label{divide}
                If $0 \leq m \leq p-2$  and $p \nmid {r \choose m+1}$,
                then $\theta^{m+1} \mid D(Q) \iff \theta^{m+1} \mid Q.$
              \end{lemma}

  \begin{remark}
    By the lemma, $D$ induces an inclusion
    $$\dfrac{V_{r}}{V_{r}^{(m+1)}} \hookrightarrow \dfrac{V_{r+p-1}}{V_{r+p-1}^{(m+1)}},$$
    which is an isomorphism for dimension reasons. This provides
    another proof of \eqref{periodicity} using the $D$ map, under the slightly stronger assumptions 
    $0 \leq m \leq p-2$ and $p \nmid {r \choose m+1}$. The last condition is necessary (for instance, for $m = 0$,
    if $p \mid r$, then $D$ maps $X^r$, $Y^r$ to $0$, so the map above is not injective).
  \end{remark}

  Now,
	let $\theta^kP=D(Q),$ for some $P\in V_{r+p+1}$, $Q\in V_{r+2+k(p+1)}$ and $1 \leq k \leq p-2$. 
	By Lemma~\ref{divide}, 
        we have $\theta^k \mid D(Q)$ if and only if $\theta^k \mid Q$. Thus $\theta^kP=D(Q)\in D(V^{(k)}_{r+2+k(p+1)})$ and
        hence $V^{(k)}_{r+(k+1)(p+1)}\cap D(V_{r+2+k(p+1)})\subset D(V^{(k)}_{r+2+k(p+1)})$. So the kernel of \eqref{composition} is $D(V^{(k)}_{r+2+k(p+1)}).$
	Thus, there is an injection 
			\[\dfrac{V^{(k)}_{r+(k+1)(p+1)}}{D(V^{(k)}_{r+2+k(p+1)})}\hookrightarrow \dfrac{V_{r+(k+1)(p+1)}}{D(V_{r+2+k(p+1)})}.\]
                        If $-1\leq r\leq p-3-k$ and $0\leq k\leq p-2,$ then $D$ is injective on $V_{r+2+k(p+1)}.$ Otherwise, by \cite[Proposition 3.3]{red10} and
                        comparing degrees, we would have a relation of the form $ap+bp+c(p+1) = r+2+k(p+1)$ for some $a,b,c \geq 0$.
                        Comparing $p$-adic digits on both sides and noting they are in the range $[0,p-1]$, we have
                        $c = r+2+k$ and $a+b+c=k$. The first equality implies $c \geq k+1$, whereas the second implies $c \leq k$, a contradiction.    
                        Thus the dimension of each side of the inclusion above is $p-1.$ So it is an isomorphism.
						
	Now, by Theorem~\ref{base case cuspidal} and Remark~\ref{r=p-2,p-1},  for $-1 \leq r \leq p-3$, we have
		\[\dfrac{V_{r+p+1}}{D(V_{r+2})}\otimes V_{p-1}\simeq {\rm ind}_{T(\mathbb{F}_p)}^{G(\mathbb{F}_p)}\omega_2^{r+2}.\]
                Twisting both sides by $\det^k$ with $0\leq k \leq p-2$
        \begin{align*}
		&\implies \dfrac{V_{r+p+1}\otimes{\det}^k}{D(V_{r+2})\otimes{\det}^k}\otimes V_{p-1}\simeq {\rm ind}_{T(\mathbb{F}_p)}^{G(\mathbb{F}_p)}\left(\omega_2^{r+2}\otimes {\det}^k\vert_{T(\mathbb{F}_p)}\right)\\
		&\implies\dfrac{V^{(k)}_{r+(k+1)(p+1)}}{D(V^{(k)}_{r+2+(p+1)})}\otimes V_{p-1}\simeq {\rm ind}_{T(\mathbb{F}_p)}^{G(\mathbb{F}_p)}\left(\omega_2^{r+2}\otimes \omega_2^{k(p+1)}\right)\\
		&\implies\dfrac{V_{r+(k+1)(p+1)}}{D(V_{r+2+k(p+1)})}\otimes V_{p-1}\simeq {\rm ind}_{T(\mathbb{F}_p)}^{G(\mathbb{F}_p)}\omega_2^{r+2+k(p+1)}. \qedhere 
	\end{align*}
\end{proof}

	\begin{corollary}\label{higher m}
	Let $m \geq 0$ and $2m-1 \leq r \leq p-3.$ Then, we have
        \[\dfrac{V_{r+2+(m+1)(p-1)}}{D^{m+1}(V_{r+2})}\otimes V_{p-1}\simeq
          {\rm ind}_{T(\mathbb{F}_p)}^{G(\mathbb{F}_p)} \left( V_m\otimes \omega_2^{r+2-m} \right).\]
      \end{corollary}
      
      \begin{proof}
        Note that $V_m\big\vert_{T(\mathbb{F}_p)} \simeq \bigoplus_{j=0}^m \omega_2^{m+j(p-1)}.$
        Thus to prove the corollary, it suffices to prove that for $2m-1 \leq r \leq p-3$, we have
        $$\dfrac{V_{r+2+(m+1)(p-1)}}{D^{m+1}(V_{r+2})}\otimes V_{p-1}
	\simeq {\rm ind}_{T(\mathbb{F}_p)}^{G(\mathbb{F}_p)}\left(\bigoplus_{j=0}^{m}\omega_2^{r+2+j(p-1)}\right).
        $$
        For $m=0,$ this is Theorem \ref{base case cuspidal} and Remark~\ref{r=p-2,p-1}.
        The proof for $m>0$ is by induction and is omitted.
	\end{proof}

	 \subsection{The case of $\mathrm{GL}_2({\mathbb F}_q)$}
	 	\label{twisted}			
                We now prove Theorem~\ref{cuspidal-twisted}, which is a twisted version of Theorem \ref{base case cuspidal}.
                Recall that $V_{r_j}^{{\rm Fr}^j}:={\rm Sym}^{r_j}(\mathbb{F}_p^2)\circ {\rm Fr}^{j}$ and 
                $\{X_j^{r_j-i_j}Y_j^{i_j}\}_{0\leq i_j\leq r_j}$ is a basis of $V_{r_j}^{{\rm Fr}^j}$, for all $0\leq j\leq f-1.$ 
	\begin{lemma}
		\label{G-linearity of D_j}
                Let 
                $r_0\geq 1$. We define  
		\[D_0 = X_0^pX_1^{p-1}\cdots X_{f-1}^{p-1}\dfrac{\partial}{\partial X_0}+Y_0^pY_1^{p-1}\cdots Y_{f-1}^{p-1}\dfrac{\partial}{\partial Y_0}\]
	and
		\[ D_j = X_0^pX_1^{p-1}\cdots X_{j-1}^{p-1}\dfrac{\partial}{\partial X_j}+Y_0^pY_1^{p-1}\cdots Y_{j-1}^{p-1}\dfrac{\partial}{\partial Y_j}, \] for all $1\leq j\leq f-1.$ Then, the maps
		\begin{itemize}
                \item[(1)] \[D_0: V_{r_0}\otimes V_0^{\rm Fr}\otimes \dots \otimes V_0^{{\rm Fr}^{f-1}}\rightarrow \dfrac{V_{r_0+p-1}\otimes V_{p-1}^{\rm Fr}\otimes \dots \otimes V_{p-1}^{{\rm Fr}^{f-1}}}{\langle D_1,\dots,D_{f-1}\rangle} \> \text{ and}\]
		\item[(2)] \[D_j: V_{r_0-1}\otimes V_0^{\rm Fr}\otimes\dots \otimes V_0^{{\rm Fr}^{j-1}} \otimes V_p^{{\rm Fr}^{j}}\otimes V_{p-1}^{{\rm Fr}^{j+1}}\otimes\dots\otimes V_{p-1}^{{\rm Fr}^{f-1}}\rightarrow \dfrac{V_{r_0+p-1}\otimes V_{p-1}^{\rm Fr}\otimes \dots \otimes V_{p-1}^{{\rm Fr}^{f-1}}}{\langle D_1,\dots, D_{j-1}\rangle}\]
		\end{itemize} are $G(\mathbb{F}_q)$-linear.
	\end{lemma}
	\begin{proof}
          $(1)$. We show that $D_0$ is $G(\mathbb{F}_q)$-linear modulo $\langle D_1,\dots, D_{f-1} \rangle.$ By Bruhat decomposition, it is
          enough to check that $D_0(g\cdot X_0^{r_0-i_0}Y_0^{i_0})=g\cdot D_0(X_0^{r_0-i_0}Y_0^{i_0})$ modulo $\langle D_1,\dots, D_{f-1} \rangle$ for $g$
          either diagonal, the Weyl element $w =\left(\begin{smallmatrix}
		0 & 1\\
		1 & 0\\
	\end{smallmatrix}\right)$ or upper unipotent. 
          The case when $g$ diagonal is clear.
        If $g = w$, then 
\begin{eqnarray*}
 D_0(w\cdot X_0^{r_0-i_0}Y_0^{i_0})
 & = & D_0(Y_0^{r_0-i_0}X_0^{i_0})\\
 & = & i_0X_0^{i_0+p-1}Y_0^{r_0-i_0}X_1^{p-1}\cdots X_{f-1}^{p-1}+(r_0-i_0)X_0^{i_0}Y_0^{r_0+p-1-i_0}Y_1^{p-1}\cdots Y_{f-1}^{p-1}\\
 & = & w\cdot\left(i_0Y_0^{i_0+p-1}X_0^{r_0-i_0}Y_1^{p-1}\cdots Y_{f-1}^{p-1}+(r_0-i_0)Y_0^{i_0}X_0^{r_0+p-1-i_0}X_1^{p-1}\cdots X_{f-1}^{p-1}\right)\\
 & = & w\cdot D_0(X_0^{r_0-i_0}Y_0^{i_0}).
\end{eqnarray*}	
Finally, if $g=\left(\begin{smallmatrix}
 1 & a\\
 0 & 1
\end{smallmatrix}\right),$ for $a\in \mathbb{F}_q$, then
\begin{eqnarray}\label{unipotent1}
D_0(g\cdot X_0^{r_0-i_0}Y_0^{i_0})
& = & D_0\left(X_0^{r_0-i_0}(aX_0+Y_0)^{i_0}\right) \notag\\
& = & X_0^p\left((r_0-i_0)X_0^{r_0-i_0-1}(aX_0+Y_0)^{i_0}+X_0^{r_0-i_0}i_0a (aX_0+Y_0)^{i_0-1}\right)X_1^{p-1}\cdots X_{f-1}^{p-1} \notag \\
&   & +Y_0^pX_0^{r_0-i_0}i_0(aX_0+Y_0)^{i_0-1}Y_1^{p-1}\cdots Y_{f-1}^{p-1}
\end{eqnarray}
and
\begin{eqnarray}\label{unipotent2}
g\cdot D_0(X_0^{r_0-i_0}Y_0^{i_0}) 
& = & g\cdot \left(X_0^p(r_0-i_0)X_0^{r_0-i_0-1}Y_0^{i_0}X_1^{p-1}\cdots X_{f-1}^{p-1}+Y_0^pX_0^{r_0-i_0}i_0Y_0^{i_0-1}Y_1^{p-1}\cdots Y_{f-1}^{p-1}\right) \notag \\
& = & X_0^p(r_0-i_0)X_0^{r_0-i_0-1}(aX_0+Y_0)^{i_0}X_1^{p-1}\cdots X_{f-1}^{p-1}\notag\\
&   & +i_0X_0^{r_0-i_0}(aX_0+Y_0)^{p+i_0-1}(a^pX_1+Y_1)^{p-1}\cdots (a^{p^{f-1}}X_{f-1}+Y_{f-1})^{p-1},
\end{eqnarray}
so taking the difference of \eqref{unipotent1} and \eqref{unipotent2}, we have
\begin{align}
  \label{difference}
& \notag D_0(g\cdot X_0^{r_0-i_0}Y_0^{i_0})- g\cdot D_0(X_0^{r_0-i_0}Y_0^{i_0})\\\notag
&=i_0aX_0^pX_0^{r_0-i_0}(aX_0+i_0Y_0)^{i_0-1}X_1^{p-1}\cdots X_{f-1}^{p-1}+i_0Y_0^pX_0^{r_0-i_0}(aX_0+Y_0)^{i_0-1}Y_1^{p-1}\cdots Y_{f-1}^{p-1}\\\notag
&{\hspace{15pt}}-i_0X_0^{r_0-i_0}(aX_0+Y_0)^{p+i_0-1}(a^pX_1+Y_1)^{p-1}\cdots (a^{p^{f-1}}X_{f-1}+Y_{f-1})^{p-1}\\
&=i_0X_0^{r_0-i_0}(aX_0+Y_0)^{i_0-1}\left(aX_0^p\prod_{j=1}^{f-1}X_j^{p-1}+Y_0^p\prod_{j=1}^{f-1}Y_j^{p-1}-(a^pX_0^p+Y_0^p)\prod_{j=1}^{f-1}(a^{p^j}X_j+Y_j)^{p-1}\right).
 \end{align}
The term in the parentheses 
\begin{align*}
&\quad =\left(aX_0^p\prod_{j=1}^{f-1}X_j^{p-1}-a^pX_0^p\prod_{j=1}^{f-1}(a^{p^j}X_j+Y_j)^{p-1}\right)+\left(Y_0^p\prod_{j=1}^{f-1}Y_j^{p-1}-Y_0^p\prod_{j=1}^{f-1}(a^{p^j}X_j+Y_j)^{p-1}\right)\\
&\quad =-{\sum\limits_{i_1=0}^{p-1}\dots\sum\limits_{i_{f-1}=0}^{\>\>p-1_{\>\>'}}}X_0^pa^p\prod_{j=1}^{f-1}\left(\binom{p-1}{i_j}a^{p^j(p-1-i_j)}X_j^{p-1-i_j}Y_j^{i_j}\right)\\
&\quad {\hspace{15pt}}-{\sum\limits_{i_1=0}^{p-1}\dots\sum\limits_{i_{f-1}=0}^{\>\>p-1_{\>\>''}}}Y_0^p\prod_{j=1}^{f-1}\left(\binom{p-1}{i_j}a^{p^j(p-1-i_j)}X_j^{p-1-i_j}Y_j^{i_j}\right)\\
&\quad = -{\sum\limits_{i_1=0}^{p-1}\dots\sum\limits_{i_{f-1}=0}^{\>\>p-1_{\>\>'}}}a^{1-\sum\limits_{j=1}^{f-1}i_jp^j}(-1)^{\sum\limits_{j=1}^{f-1}i_j}X_0^p\prod_{j=1}^{f-1}X_j^{p-1-i_j}Y_j^{i_j}\\
&\quad {\hspace{15pt}}-{\sum\limits_{i_1=0}^{p-1}\dots\sum\limits_{i_{f-1}=0}^{\>\>p-1_{\>\>''}}}a^{1-p-\sum\limits_{j=1}^{f-1}i_jp^j}(-1)^{\sum\limits_{j=1}^{f-1}i_j}Y_0^p\prod_{j=1}^{f-1}X_j^{p-1-i_j}Y_j^{i_j},
\end{align*}
where $\sum\dots\sum'$ means $(0,\dots, 0)$ is omitted from the sum and $\sum\dots\sum''$ means $(p-1,\dots,p-1)$
is omitted from the sum. Writing $I_j=\sum\limits_{l=j}^{f-1}i_lp^l$ and $\sum_pI_j=\sum\limits_{l=j}^{f-1}i_l,$ the above expression
\begin{align*}
	&= -\sum\limits_{j=1}^{f-1}\left(\sum\limits_{i_1=0}^0\dots\sum\limits_{i_{j-1}=0}^{0}\sum\limits_{i_j=1}^{p-1}\sum\limits_{i_{j+1}=0}^{p-1}\dots\sum\limits_{i_{f-1}=0}^{p-1}(-1)^{\sum_pI_j}a^{1-I_j}X_0^p\left(\prod_{l=1}^{j-1}X_l^{p-1}\right)\prod_{l=j}^{f-1}X_l^{p-1-i_l}Y_l^{i_l}\right)\\
&\quad -\sum\limits_{j=1}^{f-1}\left(\sum\limits_{i_1=p-1}^{p-1}\dots\sum\limits_{i_{j-1}=p-1}^{p-1}\sum\limits_{i_j=0}^{p-2}\sum\limits_{i_{j+1}=0}^{p-1}\dots\sum\limits_{i_{f-1}=0}^{p-1}(-1)^{\sum_pI_j}a^{1-p^j-I_j}Y_0^p\left(\prod_{l=1}^{j-1}Y_l^{p-1}\right)\left(\prod_{l=j}^{f-1}X_l^{p-1-i_l}Y_l^{i_l}\right)\right)\\
        &=-\sum\limits_{j=1}^{f-1}\left( 
          \sum\limits_{i_j=1}^{p-1}\sum\limits_{i_{j+1}=0}^{p-1}\dots\sum\limits_{i_{f-1}=0}^{p-1}(-1)^{\sum_pI_j}a^{1-I_j}X_0^p\left(\prod_{l=1}^{j-1}X_l^{p-1}\right)X_j^{p-1-i_j}Y_j^{i_j}\left(\prod_{l=j+1}^{f-1}X_l^{p-1-i_l}Y_l^{i_l}\right)\right)\\
        &\quad -\sum\limits_{j=1}^{f-1}\left( 
          \sum\limits_{i_j=1}^{p-1}\sum\limits_{i_{j+1}=0}^{p-1}\dots\sum\limits_{i_{f-1}=0}^{p-1}(-1)^{\sum_pI_j-1}a^{1-I_j}Y_0^p\left(\prod_{l=1}^{j-1}Y_l^{p-1}\right)X_j^{p-i_j}Y_j^{i_j-1}\left(\prod_{l=j+1}^{f-1}X_l^{p-1-i_l}Y_l^{i_l}\right)\right),
\end{align*}
where the fourth sum above is obtained by the transformation $i_j\mapsto i_j-1$ in the second sum above, which, together with \eqref{difference}, gives
\begin{align*}
&D_0(g\cdot X_0^{r_0-i_0}Y_0^{i_0})-g\cdot D_0(X_0^{r_0-i_0}Y_0^{i_0})\\
&\quad =-\sum\limits_{j=1}^{f-1}\left(\sum\limits_{i_j=1}^{p-1}\sum\limits_{i_{j+1}=0}^{p-1}\dots\sum\limits_{i_{f-1}=0}^{p-1}\dfrac{(-1)^{\sum_pI_j-1}}{i_j}a^{1-I_j}D_j\left(i_0X_0^{r_0-i_0}(aX_0+Y_0)^{i_0-1}X_j^{p-i_j}Y_j^{i_j}\prod_{l=j+1}^{f-1}X_l^{p-1-i_l}Y_l^{i_l}\right)\right)\\
&\quad \in \langle D_1,\dots, D_{f-1}\rangle.
\end{align*}
Thus $D_0$ is $G({\mathbb F}_q)$-linear modulo $\langle D_1,\dots, D_{f-1}\rangle.$

(2). The proof 
is similar and is omitted.
\end{proof}

Let $\alpha \in \mathbb{F}_{q^2}$ be such that $\alpha^2\in \mathbb{F}_q$, $\alpha\notin\mathbb{F}_q.$ Fix an identification $i:\mathbb{F}_q^{\times}\simeq T(\mathbb{F}_q)\subset {\rm GL}_2(\mathbb{F}_q)$ given by 
	$u + v\alpha \mapsto \left(\begin{smallmatrix}
		u & v\alpha^2\\
		v & u 
	\end{smallmatrix}\right)$.				
					
    Let $r\geq 0$ and $0\leq i\leq r+q^2-1.$ Let $f_i:G(\mathbb{F}_q)\rightarrow \mathbb{F}_{q^2}$ be a function such that
    \[f_i\left(\left(\begin{matrix}
	a & b\\
	c & d
\end{matrix}\right)\right)=(a+c\alpha)^{r+q^2-1-i}(b+d\alpha)^i,\] 
for all $\left(\begin{smallmatrix}
	a & b\\
	c & d
\end{smallmatrix}\right)\in G(\mathbb{F}_q).$ Then $f_i\in {\rm ind}_{T(\mathbb{F}_q)}^{G(\mathbb{F}_q)}\omega_{2f}^{r}$ and the set
$\mathcal{B}_q=\{f_i| 0\leq i\leq q^2-q-1\}$ forms a basis of $ {\rm ind}_{T(\mathbb{F}_q)}^{G(\mathbb{F}_q)}\omega_{2f}^{r}.$ Indeed,
let $t=\left(\begin{smallmatrix}
	u  & v\alpha^2\\
	v & u
\end{smallmatrix}\right)\in T(\mathbb{F}_q)$ and $g=\left(\begin{smallmatrix}
a & b\\
c & d
\end{smallmatrix}\right)\in G(\mathbb{F}_q).$ Then we have
\[f_i\left(t\cdot g\right)=f_i\left(\left(\begin{matrix}
	ua+v\alpha^2c & ub+v\alpha^2d\\
	va+uc & vb+ud
\end{matrix}\right)\right)\\
=(u+v\alpha)^{r+q^2-1}(a+c\alpha)^{(r+q^2-1)-i}(b+d\alpha)^i\]
which equals
$(u+v\alpha)^r\cdot f_i(g)
=\omega_{2f}^r(t)\cdot f_i(g)
=t\cdot f_i(g)$.
One can check that the functions in $\mathcal{B}_q$ are linearly independent. Also, $T(\mathbb{F}_q)$ has index $q^2-q$ in $G(\mathbb{F}_q).$ So $\mathcal{B}_q$ forms a basis of ${\rm ind}_{T(\mathbb{F}_q)}^{G(\mathbb{F}_q)}\omega_{2f}^r.$ 

For $q^2-1\leq i\leq r+q^2-1,$ we observe that
	\begin{equation}\label{flip in general}
		f_i=f_{q^2-1+j}=f_j
	\end{equation}
	for some $0\leq j\leq r.$ 
        We say that $f_i$ is a \emph{flip}.
        We soon assume that $r \leq p-1 \leq q^2-q-1$, so $f_i \in \mathcal{B}_q$.
	
	On the other hand, for $q^2-q\leq i\leq q^2-2,$ we have $f_i=f_{q^2-q+j}$ for some $0\leq j\leq q-2.$
        We say that $f_i$ is a \emph{flop} since it satisfies the following relation:
	\begin{equation}\label{flop in general}
		f_j+f_{j+(q-1)}+f_{j+2(q-1)}+\dots +f_{j+(q-1)(q-1)}+f_{j+q^2-q}=0,
        \end{equation}
        where all terms but the last are in $\mathcal{B}_q$.
	Indeed, since
	$X^{q^2-1}-1=(X^{q-1}-1)(X^{(q-1)q}+X^{(q-1)(q-1)}+\dots+X^{q-1}+1)$,
	for $A\in \mathbb{F}_{q^2}^{\times}\setminus \mathbb{F}_q^{\times},$ we have
	$A^{(q-1)q}+A^{(q-1)(q-1)}+\dots+A^{q-1}+1=0$.				
	Thus for $\left(\begin{smallmatrix}
		a & b\\
		c & d
	\end{smallmatrix}\right)\in G(\mathbb{F}_q),$  
	\[\left(\dfrac{a+c\alpha}{b+d\alpha}\right)^{(q-1)q}+\left(\dfrac{a+c\alpha}{b+d\alpha}\right)^{(q-1)(q-1)}+\dots+\left(\dfrac{a+c\alpha}{b+d\alpha}\right)^{q-1}+1=0,\]
	which, after multiplying by $(a+c\alpha)^{(r+q-1)-j}(b+d\alpha)^{q^2-q+j}$, gives 
		$$(a+c\alpha)^{(r+q^2-1)-j}(b+d\alpha)^{j}+(a+c\alpha)^{(r+q^2-q)-j}(b+d\alpha)^{j+(q-1)}
                +\dots+(a+c\alpha)^{(r+q-1)-j}(b+d\alpha)^{j+q^2-q}=0,$$
	which shows 
        \eqref{flop in general}.

	Thus, by using \eqref{flip in general} and \eqref{flop in general}, any flip or flop can be changed to an easy linear combination of functions
        in $\mathcal{B}_q.$ We fix the basis $\mathcal{B}_q$ in the computations to follow. 				
	
	For any $\left(\begin{smallmatrix}
		a & b\\ c & d
              \end{smallmatrix}\right)\in G(\mathbb{F}_q),$ we denote $A_\alpha=a+c\alpha$ and $B_\alpha=b+d\alpha.$ For any polynomial $P(X, Y)$
            and $A, B, C, D\in \mathbb{F}_{q^2}$, we write $P(X, Y)\Big\vert_{(A, B)}^{(C, D)}=P(C, D)-P(A, B).$

	\begin{lemma}\label{T-linearity in general}
          Let $r=r_0+r_1p+\dots+r_{f-1}p^{f-1}$ with $2\leq r_0\leq p-1$ and $r_j\geq 0$ for all $1\leq j\leq f-1.$
          Let $P\otimes Q:=\bigotimes_{j=0}^{f-1}P_j\otimes \bigotimes_{j=0}^{f-1} Q_j
          \in \bigotimes_{j=0}^{f-1}V_{r_j+p-1}^{{\rm Fr}^j}\otimes \bigotimes_{j=0}^{f-1} V_{p-1}^{{\rm Fr}^j},$ with $P_j$ a homogeneous polynomial of
          degree $r_j+p-1$ in $X_j$, $Y_j,$ and $Q_j$ homogeneous of degree $p-1$ in $S_j$, $T_j.$ Define
	  $\psi_{P\otimes Q}:G(\mathbb{F}_q)\rightarrow \mathbb{F}_{q^2}$
	  by
	\[\left(\begin{matrix}
		a & b\\
		c & d
	\end{matrix}\right)\mapsto \nabla_{0}^{r_0-2}(P_0)\nabla_1^{r_1}(P_1)\cdots \nabla_{f-1}^{r_{f-1}}(P_{f-1})\Big\vert_{(A_\alpha, B_\alpha,\dots, A_\alpha^{p^{f-1}}, B_\alpha^{p^{f-1}})}^{(A_\alpha^{p^f}, B_\alpha^{p^f},\dots, A_\alpha^{p^{2f-1}}, B_\alpha^{p^{2f-1}})} \cdot \prod_{j=0}^{f-1}Q_j(A_\alpha^{p^{f+j}}, B_\alpha^{p^{f+j}}),\]
where $$\nabla_j = A_\alpha^{p^j}\frac{\partial}{\partial X_j}+B_\alpha^{p^j}\frac{\partial}{\partial Y_j}$$ for all $0\leq j\leq f-1.$
Then the map
\begin{itemize}
	\item[(1)] $\psi_{P\otimes Q}$ is $T({\mathbb{F}_q})$-linear.
	\item[(2)] $\psi:\bigotimes_{j=0}^{f-1}V_{r_j+p-1}^{{\rm Fr}^j}\otimes \bigotimes_{j=0}^{f-1} V_{p-1}^{{\rm Fr}^j}\rightarrow {\rm ind}_{T(\mathbb{F}_q)}^{G(\mathbb{F}_q)}\omega_{2f}^r$ such that $\psi(P\otimes Q)=\psi_{P\otimes Q}$ is $G(\mathbb{F}_q)$-linear.
\end{itemize}

\end{lemma}
	
\begin{proof}
	\noindent \textbf{$T(\mathbb{F}_q)$-linearity:}
Note $\nabla_{0}^{r_0-2}(P_0)\prod_{j=1}^{f-1}\nabla_j^{r_j}(P_j)\prod_{j=0}^{f-1}Q_j(S_j, T_j)$ is a linear combination of terms  
\[A_\alpha^{r-i-2}B_\alpha^{i}X_0^{p+1-i_0}Y_0^{i_0}\prod_{j=1}^{f-1}X_j^{p-1-i_j}Y_j^{i_j}\cdot \prod_{j=0}^{f-1}S_j^{p-1-k_j}T_j^{k_j},\]
where $r = \sum_{j=0}^{f-1}r_jp^j$ and $i=\sum_{j=0}^{f-1}i_jp^j.$
Now,
\begin{align*}
	& A_\alpha^{r-i-2}B_\alpha^{i}X_0^{p+1-i_0}Y_0^{i_0}\prod_{j=1}^{f-1}X_j^{p-1-i_j}Y_j^{i_j}\Big\vert_{(A_\alpha, B_\alpha,\dots, A_\alpha^{p^{f-1}}, B_\alpha^{p^{f-1}})}^{(A_\alpha^{p^f}, B_\alpha^{p^f},\dots, A_\alpha^{p^{2f-1}}, B_\alpha^{p^{2f-1}})}\prod_{j=0}^{f-1}S_j^{p-1-k_j}T_j^{k_j}\Big\vert_{\left(0, 0\right)}^{(A_\alpha^{p^{f+j}}, B_\alpha^{p^{f+j}})}\\
	&\quad =A_\alpha^{r-i-2}B_\alpha^i\left(A_\alpha^{ p^f\left(p+1-i_0+\sum\limits_{j=1}^{f-1}(p^{j+1}-p^{j}-i_jp^{j}) \right)}B_\alpha^{p^f\left( i_0+\sum\limits_{j=1}^{f-1}i_jp^{j} \right) }-A_\alpha^{p+1-i_0+\sum\limits_{j=1}^{f-1}(p^{j+1}-p^j-i_jp^j)} B_\alpha^{i_0+\sum\limits_{j=1}^{f-1}i_jp^{j}} \right)\\
	&\qquad \qquad \qquad \qquad \qquad \cdot A_\alpha^{p^f \left(  \sum\limits_{j=0}^{f-1}\left(p^{j+1}-p^j-k_jp^j\right)\right)}B_{\alpha}^{p^f\left(\sum\limits_{j=0}^{f-1}k_jp^j\right)}\\
	&\quad =A_\alpha^{r-i-2}B_\alpha^i\left(A_\alpha^{q+q^2-iq}B_\alpha^{iq}-A_\alpha^{1-i+q}B_\alpha^i\right)A_\alpha^{q^2-q-kq}B_\alpha^{kq}
               \> = \> A_\alpha^{r+q^2-1-i-(i+k)q}B_\alpha^{i+(i+k)q}-A_{\alpha}^{r+q^2-1-(2i+kq)}B_\alpha^{2i+kq},
	\end{align*}
	where in the last but one equality 
        $k=\sum_{j=0}^{f-1}k_jp^{j}.$ Thus $\psi_{P\otimes Q}$ is a linear combination of functions of the form $f_i$ above
        and hence is $T(\mathbb{F}_q)$-linear. \\

	\noindent \textbf{$G(\mathbb{F}_q)$-linearity:} Let 
	$g=\left(\begin{smallmatrix}
						u & v\\
						w & z
					\end{smallmatrix}\right)\in G(\mathbb{F}_q).$ Note that,
	\begin{eqnarray*}
	g\cdot \left(P\otimes Q\right) 
	& = &\bigotimes_{j=0}^{f-1}P_j(U_j, V_j)\otimes \bigotimes_{j=0}^{f-1}Q_j(U_j', V_j') =: \bigotimes_{j=0}^{f-1}P_j'\otimes \bigotimes_{j=0}^{f-1}Q_j',
	\end{eqnarray*}
where
$U_j=u^{p^j}X_j+w^{p^j}Y_j, V_j=v^{p^j}X_j+z^{p^j}Y_j, \> U_j'=u^{p^j}S_j+w^{p^j}T_j, V_j'=v^{p^j}S_j+z^{p^j}T_j$.
Now,
\begin{eqnarray}\label{glinear one side}
&& \psi\left(g\cdot P\otimes Q\right)\left(\left(\begin{matrix}
	a & b\\
	c & d
\end{matrix}\right)\right) \notag \\
&& \quad =\left(A_\alpha\dfrac{\partial}{\partial X_0}+B_\alpha\dfrac{\partial}{\partial Y_0}\right)^{r_0-2}(P_0')\prod_{j=1}^{f-1}\left(A_\alpha^{p^j}\dfrac{\partial}{\partial X_j}+B_\alpha^{p^j}\dfrac{\partial}{\partial Y_j}\right)^{r_j}(P_j')\Big\vert_{(A_\alpha, B_\alpha,\dots, A_\alpha^{p^{f-1}}, B_\alpha^{p^{f-1}})}^{(A_\alpha^{p^f}, B_\alpha^{p^f},\dots, A_\alpha^{p^{2f-1}}, B_\alpha^{p^{2f-1}})}\notag \\
&& {\hspace{260pt}}  \cdot  \prod_{j=0}^{f-1}Q_j'(A_\alpha^{p^{f+j}}, B_\alpha^{p^{f+j}}).
\end{eqnarray}
Applying Lemma~\ref{glineargeneral} (twice),
for $0\leq j\leq f-1$ and $k\geq 0$ we have 
\begin{align*}
	& \left(A_\alpha^{p^j}\dfrac{\partial}{\partial X_j}+B_\alpha^{p^j}\dfrac{\partial}{\partial Y_j}\right)^k(P_j')\Big\vert_{(A_\alpha^{p^j}, B_\alpha^{p^j})}^{(A_\alpha^{p^{f+j}}, B_\alpha^{p^{f+j}})}\\
	&\quad =\left((u^{p^j}A_\alpha^{p^j}+w^{p^j}B_\alpha^{p^j})\frac{\partial}{\partial X_j}+(v^{p^j}A_\alpha^{p^j}+z^{p^j}B_\alpha^{p^j})\frac{\partial}{\partial Y_j}\right)^k(P_j)\Bigg\vert_{(u^{p^j}A_\alpha^{p^j}+w^{p^j}B_\alpha^{p^j}, v^{p^j}A_\alpha^{p^j}+z^{p^j}B_\alpha^{p^j})}^{(u^{p^j}A_\alpha^{p^{f+j}}+w^{p^j}B_\alpha^{p^{f+j}}, v^{p^j}A_\alpha^{p^{f+j}}+z^{p^j}B_\alpha^{p^{f+j}})}\\
	&\quad =\left(A_\alpha'^{p^j}\dfrac{\partial}{\partial X_j}+B_\alpha'^{p^j}\dfrac{\partial}{\partial Y_j}\right)^k(P_j)\Big\vert_{(A_\alpha'^{p^j}, B_\alpha'^{p^j})}^{(A_\alpha'^{p^{f+j}}, B_\alpha'^{p^{f+j}})},
\end{align*}
where $A_\alpha'=uA_\alpha+wB_\alpha$ and $B_\alpha'=vA_\alpha+zB_\alpha.$ Observe that, in the last equality in writing the top limits in terms of $A_\alpha'$ and $B_\alpha',$ we use the fact that $u,v,w$ and $z$ are in $\mathbb{F}_q.$ 

Now taking, $k=r_0-2$ for $j=0$ and $k=r_j$ for $1\leq j\leq f-1,$ the expression in \eqref{glinear one side} equals
\begin{align*}
&\left(A_\alpha'\dfrac{\partial}{\partial X_0}+B_\alpha'\dfrac{\partial}{\partial Y_0}\right)^{r_0-2}(P_0)\prod_{j=1}^{f-1}\left(A_\alpha'^{p^j}\dfrac{\partial}{\partial X_j}+B_\alpha'^{p^j}\dfrac{\partial}{\partial Y_j}\right)^{r_j}(P_j)\Big\vert_{(A_\alpha', B_\alpha',\dots, A_\alpha'^{p^{f-1}}, B_\alpha'^{p^{f-1}})}^{(A_\alpha'^{p^f}, B_\alpha'^{p^f},\dots, A_\alpha'^{p^{2f-1}}, B_\alpha'^{p^{2f-1}})}\\
&\qquad \qquad \qquad \qquad \qquad \qquad \qquad \qquad \qquad \qquad {\hspace{20pt}} \cdot \prod_{j=0}^{f-1}Q_j(A_\alpha'^{p^{f+j}}, B_\alpha'^{p^{f+j}})\\
& \quad =\psi(P\otimes Q)\left(\left(\begin{matrix}
	ua+wb & va+zb\\
	uc+wd & vc+zd
\end{matrix}\right)\right)
\>  = \> g\cdot \left(\psi(P\otimes Q)\right)\left(\left(\begin{matrix}
	a & b\\
	c & d
\end{matrix}\right)\right).   \qedhere
\end{align*} 
\end{proof}

The  following lemmas will be useful in the proof of the main theorem.
\begin{lemma}
	\label{image of D_j}
	Let 
	\[ P=\sum\limits_{i_0=0}^{r_0+p-1}\sum\limits_{i_1=0}^{p-1}\dots\sum\limits_{i_{f-1}=0}^{p-1} b_{i_0, \dots, i_{f-1}} X_0^{r_0+p-1-i_0} Y_0^{i_0} \left(\prod_{l=1}^{f-1} X_l^{p-1-i_l} Y_l^{i_l} \right) \in V_{r_0+p-1}\otimes \bigotimes_{l=1}^{f-1} V_{p-1}^{{\rm Fr}^l}. \]
	
	Then,
	\begin{itemize}
		\item[(1)] we have $P\in {\rm Im}~D_0$ if and only if 
		\[i_0 b_{i_0, 0,\ldots, 0} = (r_0-i_0)   b_{i_0+p-1, p-1,\ldots, p-1}\]
		for $1\leq i_0 \leq r_0-1,$ and 
		\[b_{i_0,\dots, i_{f-1}}=0\] 
		for $r_0 \leq i_0 \leq p-1$ and all  $0 \leq i_j\leq p-1$ for $1 \leq j \leq f-1$.
                \item[(2)] for $1\leq j \leq f-1,$  we have $P\in {\rm Im}~D_j$ if and only if 
		\[b_{i_0, 0, \dots, 0, i_j, \dots, i_{f-1}} =- b_{i_0+p, p-1, \dots, p-1, i_j-1, i_{j+1}, \dots, i_{f-1}}\]
		where $0\leq i_0 \leq r_0-1,$ $1\leq i_j \leq p-1$ and $0\leq i_{j+1}, \dots, i_{f-1} \leq p-1,$ and in the remaining cases
		\[b_{i_0,\dots, i_{f-1}}=0.\]
	\end{itemize}
\end{lemma}
\begin{proof}
  The conditions are clearly necessary, and can be checked to be sufficient.
\end{proof}

\begin{lemma}
	\label{space generated by D_j's}
	Let $1 \leq r_0 \leq p-1$ and let 
	\[ P=\sum\limits_{i_0=0}^{r_0+p-1}\sum\limits_{i_1=0}^{p-1}\dots\sum\limits_{i_{f-1}=0}^{p-1} b_{i_0, \dots, i_{f-1}} X_0^{r_0+p-1-i_0} Y_0^{i_0} \left(\prod_{l=1}^{f-1} X_l^{p-1-i_l} Y_l^{i_l} \right) \in V_{r_0+p-1}\otimes\bigotimes_{l=1}^{f-1} V_{p-1}^{{\rm Fr}^l}. \]
	Then we have $P\in \langle D_0, \dots, D_{f-1} \rangle$ 
	if and only if the following hold:
	\begin{itemize}
		\item[(1)] for $1\leq i_0 \leq r_0-1,$ we have
		\[i_0b_{i_0, 0, \dots, 0}=(r_0-i_0) b_{i_0+p-1, p-1, \dots, p-1}\]
		\item[(2)] for $r_0\leq i_0 \leq p-1$ and all $0\leq i_j \leq p-1$ for $1\leq j\leq f-1,$ we have  \[b_{i_0, \dots, i_{f-1}}=0,\]
		\item[(3)] for $0\leq i_0\leq r_0-1$ and for all $1\leq t \leq f-1,$ $1\leq i_t \leq p-1$ and $0\leq i_{t+1}, \dots, i_{f-1}\leq p-1,$ we  have
		\[b_{i_0, 0,\dots, 0, i_t, \dots, i_{f-1}}=-b_{i_0+p, p-1, \dots, p-1, i_t-1, i_{t+1}, \dots, i_{f-1}}.\] 
	\end{itemize}
\end{lemma}

\begin{proof}
        Say $P\in \langle D_0, \dots, D_{f-1} \rangle.$ We show conditions $(1)$, $(2)$, $(3)$ are satisfied. Write 
	$P=\sum_{j=0}^{f-1} P_j$ where 
	\[P_j=\sum\limits_{i_0=0}^{r_0+p-1}\dots \sum\limits_{i_{f-1}=0}^{p-1} b_{i_0,\dots, i_{f-1}}^{(j)} X_0^{r_0+p-1-i_0} Y_0^{i_0}\prod_{l=1}^{f-1} X_l^{p-1-i_l} Y_l^{i_l} \in {\rm Im}~D_j\]
	for all $0\leq j\leq f-1.$
        Then  $$b_{i_0,\dots, i_{f-1}} = \sum_{j=0}^{f-1} b_{i_0,\dots, i_{f-1}}^{(j)}$$
        with each $b_{i_0,\dots, i_{f-1}}^{(j)}$ satisfying the conditions of Lemma~\ref{image of D_j}.
        Now (2) is clear since if $r_0 \leq i_0 \leq p-1$, each term on the right vanishes by (both parts of) Lemma~\ref{image of D_j}.
        Condition (1) is also clear, since if $1 \leq i_0 \leq r_0-1$ and the other $i_j$ are all $0$ or all $p-1$, then all
        the terms on the right vanish for $j \geq 1$, by the second part of Lemma~\ref{image of D_j}, and 
        the $j=0$ term on the right satisfies the desired identity by the first part. Similarly (3) holds, since
        if $0 \leq i_0 \leq r_0-1$ and the other $i_j$ are not all $0$ or not all $p-1$, then the $j = 0$
        term on the right vanishes by the first part of the lemma and the remaining terms satisfy the desired identity by the second part,
        whence so does their sum.
                
	For the converse, note that $P$ can be written as
{\small	\begin{eqnarray*}
		&& \sum_{t=1}^{f-1}\left( \sum_{i_0=0}^{r_0-1}\sum_{i_1=0}^{0}\dots\sum_{i_{t-1}=0}^{0}\sum_{i_t=1}^{p-1}\sum_{i_{t+1}=0}^{p-1}\dots \sum_{i_{f-1}=0}^{p-1} b_{i_0, 0, \dots, 0, i_t,\dots, i_{f-1}} X_0^{r_0+p-1-i_0} Y_0^{i_0}\prod_{l=1}^{t-1} X_l^{p-1}  \prod_{l=t}^{f-1} X_l^{p-1-i_l} Y_l^{i_l}\right)\\
		&& \quad + \sum_{i_0=1}^{r_0-1} b_{i_0,0,\dots, 0} X_0^{r_0+p-1-i_0} Y_0^{i_0} \prod_{l=1}^{f-1} X_l^{p-1}+b_{0,\dots, 0} X_0^{r_0+p-1} \prod_{l=1}^{f-1} X_l^{p-1}\\
		&& + \sum_{i_0=r_0}^{p-1}\sum_{i_1=0}^{p-1}\dots \sum_{i_{f-1}=0}^{p-1} b_{i_0, \dots, i_{f-1}} X_0^{r_0+p-1-i_0} Y_0^{i_0} \prod_{l=1}^{f-1}X_l^{p-1-i_l} Y_l^{i_l}\\
		&& + \sum_{t=1}^{f-1}\left( \sum_{i_0=p}^{r_0+p-1}\sum_{i_1=p-1}^{p-1}\dots \sum_{i_{t-1}=p-1}^{p-1}\sum_{i_t=0}^{p-2}\sum_{i_{t+1}=0}^{p-1}\dots \sum_{i_{f-1}=0}^{p-1} b_{i_0, p-1, \dots, p-1, i_t,\dots, i_{f-1}} X_0^{r_0+p-1-i_0} Y_0^{i_0}\prod_{l=1}^{t-1} Y_l^{p-1}  \prod_{l=t}^{f-1} X_l^{p-1-i_l} Y_l^{i_l}\right)\\
		&& \quad  + \sum_{i_0=p}^{r_0+p-2} b_{i_0,p-1,\dots, p-1} X_0^{r_0+p-1-i_0} Y_0^{i_0} \prod_{l=1}^{f-1} Y_l^{p-1}+b_{r_0+p-1, p-1, \dots, p-1} Y_0^{r_0+p-1} \prod_{l=1}^{f-1} Y_l^{p-1},
	\end{eqnarray*}}
\!\!\!\!      which, by the transformations $i_0\mapsto i_0+p$ and $i_t\mapsto i_t-1$ in the fourth sum above (and dropping the summations for $i_1,\dots, i_{t-1}$ in the first and fourth sum) and
      by the transformation $i_0 \mapsto i_0+p-1$ in the last sum, can be rewritten as 
{\small	\begin{eqnarray}
		\label{breaking the sum}
		&& \sum_{t=1}^{f-1}\left( \sum_{i_0=0}^{r_0-1}\sum_{i_t=1}^{p-1}\sum_{i_{t+1}=0}^{p-1}\dots \sum_{i_{f-1}=0}^{p-1} b_{i_0, 0, \dots, 0, i_t,\dots, i_{f-1}} X_0^{r_0+p-1-i_0} Y_0^{i_0}\left(\prod_{l=1}^{t-1} X_l^{p-1}\right)  X_t^{p-1-i_t} Y_t^{i_t} \prod_{l=t+1}^{f-1} X_l^{p-1-i_l} Y_l^{i_l}\right) \nonumber \\ 
		&& + \sum_{t=1}^{f-1}\left( \sum_{i_0=0}^{r_0-1} \sum_{i_t=1}^{p-1}\sum_{i_{t+1}=0}^{p-1}\dots \sum_{i_{f-1}=0}^{p-1} b_{i_0+p, p-1, \dots, p-1, i_t-1,\dots, i_{f-1}} X_0^{r_0-i_0-1} Y_0^{i_0+p}\prod_{l=1}^{t-1} Y_l^{p-1}  X_t^{p-i_t}Y_t^{i_t-1}\prod_{l=t+1}^{f-1} X_l^{p-1-i_l} Y_l^{i_l}\right) \nonumber \\
		&& \quad+ \sum_{i_0=1}^{r_0-1} b_{i_0,0,\dots, 0} X_0^{r_0+p-1-i_0} Y_0^{i_0} \prod_{l=1}^{f-1} X_l^{p-1}  +\sum_{i_0=1}^{r_0-1} b_{i_0+p-1,p-1,\dots, p-1} X_0^{r_0-i_0} Y_0^{i_0+p-1} \prod_{l=1}^{f-1} Y_l^{p-1} \nonumber \\ 
		&& + \sum_{i_0=r_0}^{p-1}\sum_{i_1=0}^{p-1}\dots \sum_{i_{f-1}=0}^{p-1} b_{i_0, \dots, i_{f-1}} X_0^{r_0+p-1-i_0} Y_0^{i_0} \prod_{l=1}^{f-1}X_l^{p-1-i_l} Y_l^{i_l} \nonumber \\ 
		&& \quad +b_{0,\dots, 0} X_0^{r_0+p-1} \prod_{l=1}^{f-1} X_l^{p-1} + b_{r_0+p-1, p-1, \dots, p-1} Y_0^{r_0+p-1} \prod_{l=1}^{f-1} Y_l^{p-1}. 
	\end{eqnarray} }
	
	Now suppose the conditions $(1)$, $(2)$, $(3)$ hold. Then by \eqref{breaking the sum}, we can write the polynomial $P$ as 
	\begin{eqnarray}
		&& \sum_{t=1}^{f-1}\left( - \sum_{i_0=0}^{r_0-1}\sum_{i_t=1}^{p-1}\sum_{i_{t+1}=0}^{p-1}\dots \sum_{i_{f-1}=0}^{p-1}  \dfrac{b_{i_0, 0, \dots, 0, i_t,\dots, i_{f-1}}}{i_t} D_t\left( X_0^{r_0-1-i_0} Y_0^{i_0}  X_t^{p-i_t} Y_t^{i_t} \prod_{l=t+1}^{f-1} X_l^{p-1-i_l} Y_l^{i_l}\right) \right) \nonumber \\ \nonumber
		&& \quad + \sum_{i_0=1}^{r_0-1} \dfrac{b_{i_0,0,\dots, 0}}{(r_0-i_0)}  D_0\left(X_0^{r_0-i_0} Y_0^{i_0}\right)+ \dfrac{1}{r_0} D_0 \left( b_{0,\dots, 0} X_0^{r_0} + b_{r_0+p-1, p-1, \dots, p-1} Y_0^{r_0} \right).
	\end{eqnarray}
	This implies that $P\in \langle D_0, \dots, D_{f-1} \rangle.$ 
\end{proof}	

\begin{remark}
	\label{dimension of the space generated by D_j's}
	In Lemma~\ref{space generated by D_j's}, there are no conditions on the coefficients $b_{0, \dots, 0}$ and $b_{r_0+p-1, p-1, \dots, p-1}$ of $P$.
        So if $1 \leq r_0 \leq p-1$, then the dimension of $\langle D_0, \dots, D_{f-1} \rangle$ over $\mathbb{F}_q$ is 
	\[2+ (r_0-1)+ \sum\limits_{t=1}^{f-1} r_0(p-1)p^{f-1-t}=r_0+1+ r_0(p^{f-1}-1)=r_0p^{f-1}+1.\] 
\end{remark}

The following theorem is Theorem \ref{cuspidal-twisted} from the introduction.
\begin{theorem}\label{main theorem cuspidal twisted}
Let $r=r_0+r_1p+\dots+r_{f-1}p^{f-1},$ where $2\leq r_0\leq p-1$ and $r_j=0$ for all $1\leq j\leq f-1.$ Recall that
 \[D_0 = X_0^pX_1^{p-1}\cdots X_{f-1}^{p-1}\dfrac{\partial}{\partial X_0}+Y_0^pY_1^{p-1}\cdots Y_{f-1}^{p-1}\dfrac{\partial}{\partial Y_0} \]
and
\[D_j = X_0^pX_1^{p-1}\cdots X_{j-1}^{p-1}\dfrac{\partial}{\partial X_j}+Y_0^pY_1^{p-1}\cdots Y_{j-1}^{p-1}\dfrac{\partial}{\partial Y_j}, \] for all $1\leq j\leq f-1.$ Then over ${\mathbb F}_{q^2}$ we have
\[\dfrac{\bigotimes_{j=0}^{f-1}V_{r_j+p-1}^{{\rm Fr}^j}}{\langle D_0,\dots, D_{f-1}\rangle}\otimes\bigotimes_{j=0}^{f-1} V_{p-1}^{{\rm Fr}^j}\simeq {\rm ind}_{T(\mathbb{F}_q)}^{G(\mathbb{F}_q)}\omega_{2f}^r.\]
\end{theorem}

\begin{proof}
Define $\psi:\bigotimes_{j=0}^{f-1}V_{r_j+p-1}^{{\rm Fr}^j}\otimes \bigotimes_{j=0}^{f-1}V_{p-1}^{{\rm Fr}^j}\rightarrow {\rm ind}_{T(\mathbb{F}_q)}^{G(\mathbb{F}_q)}\omega_{2f}^r$  as in Lemma \ref{T-linearity in general}. Recall that for $P\otimes Q:=\bigotimes_{j=0}^{f-1}P_j\otimes\bigotimes_{j=0}^{f-1}Q_j,$ we have 
$\psi(P\otimes Q)=\psi_{P\otimes Q},$
where now noting that $r_j=0$ for $1\leq j\leq f-1$, we have $\psi_{P\otimes Q}:G(\mathbb{F}_q)\rightarrow \mathbb{F}_{q^2}$ is defined by
	\[\left(\begin{matrix}
				a & b\\
				c & d
				\end{matrix}\right)\mapsto \nabla_{0}^{r_0-2}(P_0)P_1\cdots P_{f-1}\Big\vert_{(A_\alpha, B_\alpha,\dots, A_\alpha^{p^{f-1}}, B_\alpha^{p^{f-1}})}^{(A_\alpha^{p^f}, B_\alpha^{p^f},\dots, A_\alpha^{p^{2f-1}}, B_\alpha^{p^{2f-1}})} \cdot \prod_{j=0}^{f-1}Q_j(A_\alpha^{p^{f+j}}, B_\alpha^{p^{f+j}}),\]
                            where $\nabla_0 = A_\alpha\frac{\partial}{\partial X_0}+B_\alpha\frac{\partial}{\partial Y_0}$.
                            By Lemma \ref{T-linearity in general} (1), the map $\psi_{P\otimes Q}$ is $T(\mathbb{F}_q)$-linear and hence $\psi$ is well defined.
                            It is also $G(\mathbb{F}_q)$-linear by Lemma \ref{T-linearity in general} (2).
 
 We show that $\ker\psi=\langle D_0,\dots, D_{f-1}\rangle  \otimes\bigotimes_{j=0}^{f-1} V_{p-1}^{{\rm Fr}^j}$. The proof occupies the rest of this paper.
 First we show $\langle D_0,\dots, D_{f-1}\rangle \otimes\bigotimes_{j=0}^{f-1} V_{p-1}^{{\rm Fr}^j}
 \subset \ker\psi.$\\
 
	\noindent \textbf{Case 1:} Suppose $j\neq 0.$ We show that ${\rm Im}~D_j\subset \ker\psi.$ 
 	For $0\leq i_0\leq r_0-1,$ we have
 	\begin{align*}
 			& D_j\left(X_0^{r_0-1-i_0}Y_0^{i_0}X_j^{p-i_j}Y_j^{i_j}\prod_{l=j+1}^{f-1}X_l^{p-1-i_l}Y_l^{i_l}\right) \\ 
 			& \quad =X_0^{r_0+p-1-i_0}Y_0^{i_0}\left( \prod_{l=1}^{j-1}X_l^{p-1} \right)         (-i_j)X_j^{p-1-i_j}Y_j^{i_j}  \left( \prod_{l=j+1}^{f-1}X_l^{p-1-i_l}Y_l^{i_l}  \right) \\ 
 			& \quad {\hspace{12pt}}+X_0^{r_0-1-i_0}Y_0^{i_0+p}
 			\left( \prod_{l=1}^{j-1}Y_l^{p-1} \right) i_jX_j^{p-i_j}Y_j^{i_j-1} 
 			 \left( \prod_{l=j+1}^{f-1}X_l^{p-1-i_l}Y_l^{i_l}  \right).
	 \end{align*}
%
        We claim
        {\small 
          \begin{align}
                           \label{applying psi map D_j case} 
 			&- \nabla_{0}^{r_0-2}\left(X_0^{r_0+p-1-i_0}Y_0^{i_0}\right)\left( \prod_{l=1}^{j-1}X_l^{p-1} \right) X_j^{p-1-i_j}Y_j^{i_j}  \left( \prod_{l=j+1}^{f-1}X_l^{p-1-i_l}Y_l^{i_l}  \right)\Big\vert_{(A_\alpha, B_\alpha,\dots, A_\alpha^{p^{f-1}}, B_\alpha^{p^{f-1}})}^{(A_\alpha^{p^f}, B_\alpha^{p^f},\dots, A_\alpha^{p^{2f-1}}, B_\alpha^{p^{2f-1}})} \notag \\
 			&  +\nabla_{0}^{r_0-2}\left(X_0^{r_0-1-i_0}Y_0^{i_0+p}\right)\left( \prod_{l=1}^{j-1}Y_l^{p-1} \right)         X_j^{p-i_j}Y_j^{i_j-1}  \left( \prod_{l=j+1}^{f-1}X_l^{p-1-i_l}Y_l^{i_l}  \right)\Big\vert_{(A_\alpha, B_\alpha,\dots, A_\alpha^{p^{f-1}}, B_\alpha^{p^{f-1}})}^{(A_\alpha^{p^f}, B_\alpha^{p^f},\dots, A_\alpha^{p^{2f-1}}, B_\alpha^{p^{2f-1}})} \> = \> 0.
  \end{align} }
	Indeed, we have
{\small 		\begin{eqnarray}
 		\label{apply nabla first time}
			\nabla_{0}^{r_0-2}\left(X_0^{r_0+p-1-i_0}Y_0^{i_0}\right) 
 			& = &\left(A_\alpha\dfrac{\partial}{\partial X_0}+B_\alpha\dfrac{\partial}{\partial Y_0}\right)^{r_0-2}\left(X_0^{r_0+p-1-i_0}Y_0^{i_0}\right) \nonumber\\
    		& = &
   			 \sum\limits_{k_0=0}^{r_0-2}\binom{r_0-2}{k_0}A_\alpha^{r_0-2-k_0}B_\alpha^{k_0}[r_0+p-1-i_0]_{r_0-2-k_0}[i_0]_{k_0}X_0^{p+1-(i_0-k_0)}Y_0^{i_0-k_0} \nonumber \\
 			& = &
			\binom{r_0-2}{i_0-1}A_\alpha^{r_0-1-i_0}B_\alpha^{i_0-1}(r_0-1-i_0)!i_0!X_0^pY_0 \nonumber \\
 			&&{\hspace{20pt}}+\binom{r_0-2}{i_0}A_\alpha^{r_0-2-i_0}B_\alpha^{i_0}(r_0-1-i_0)!i_0!X_0^{p+1},
 	 \end{eqnarray} }
 	and similarly
{\small 		\begin{eqnarray}
   			\label{apply nabla second time}
				\nabla_{0}^{r_0-2}\left(X_0^{r_0-1-i_0}Y_0^{i_0+p}\right)
				& = &
				\binom{r_0-2}{i_0-1}A_\alpha^{r_0-1-i_0}B_\alpha^{i_0-1}(r_0-1-i_0)!i_0!Y_0^{p+1} \nonumber\\
				& &
				{\hspace{5pt}}+\binom{r_0-2}{i_0}A_\alpha^{r_0-2-i_0}B_\alpha^{i_0}(r_0-1-i_0)!i_0!X_0Y_0^p.
	  \end{eqnarray} }
          Ignoring the factor $(r_0-1-i_0)! i_0!$ and using \eqref{apply nabla first time},
          the first summand on the left side of \eqref{applying psi map D_j case} becomes
{\small 	\begin{align*}
		&- \binom{r_0-2}{i_0-1}A_\alpha^{r_0-1-i_0}B_\alpha^{i_0-1}X_0^pY_0\left( \prod_{l=1}^{j-1}X_l^{p-1} \right) X_j^{p-1-i_j}Y_j^{i_j}  \left( \prod_{l=j+1}^{f-1}X_l^{p-1-i_l}Y_l^{i_l}  \right)\Big\vert_{(A_\alpha, B_\alpha,\dots, A_\alpha^{p^{f-1}}, B_\alpha^{p^{f-1}})}^{(A_\alpha^{p^f}, B_\alpha^{p^f},\dots, A_\alpha^{p^{2f-1}}, B_\alpha^{p^{2f-1}})}\\
		&  - \binom{r_0-2}{i_0}A_\alpha^{r_0-2-i_0}B_\alpha^{i_0}X_0^{p+1}\left( \prod_{l=1}^{j-1}X_l^{p-1} \right) X_j^{p-1-i_j}Y_j^{i_j}  \left( \prod_{l=j+1}^{f-1}X_l^{p-1-i_l}Y_l^{i_l}  \right)\Big\vert_{(A_\alpha, B_\alpha,\dots, A_\alpha^{p^{f-1}}, B_\alpha^{p^{f-1}})}^{(A_\alpha^{p^f}, B_\alpha^{p^f},\dots, A_\alpha^{p^{2f-1}}, B_\alpha^{p^{2f-1}})},
	 \end{align*} }
 	which, by noting 
 	\[\sum\limits_{l=1}^{j-1}(p-1)p^{f+l}+(p-1-i_j)p^{f+j}+\sum\limits_{l=j+1}^{f-1}(p-1-i_l)p^{f+l}=p^{2f}-p^{f+1}-\sum\limits_{l=j}^{f-1}i_lp^{f+l}\]
 	and, dividing by $p^f$,
	\[\sum\limits_{l=1}^{j-1}(p-1)p^{l}+(p-1-i_j)p^{j}+\sum\limits_{l=j+1}^{f-1}(p-1-i_l)p^{l}=p^{f}-p-\sum\limits_{l=j}^{f-1}i_lp^{l},\] 
	equals
{\small 	\begin{align}
 		\label{first summand of the claim}
	 		&  - \binom{r_0-2}{i_0-1}A_\alpha^{r_0-1-i_0}B_\alpha^{i_0-1}\left(A_\alpha^{p^{f+1}+p^{2f}-p^{f+1}-\sum\limits_{l=j}^{f-1}i_lp^{f+l}}B_\alpha^{p^f+\sum\limits_{l=j}^{f-1}i_lp^{f+l}}-A_\alpha^{p+p^f-p-\sum\limits_{l=j}^{f-1}i_lp^{l}}B_\alpha^{1+\sum\limits_{l=j}^{f-1}i_lp^{l}}\right) \nonumber\\
			&
			-\binom{r_0-2}{i_0}A_\alpha^{r_0-2-i_0}B_\alpha^{i_0}\left(A_\alpha^{p^{f+1}+p^f+p^{2f}-p^{f+1}-\sum\limits_{l=j}^{f-1}i_lp^{f+l}}B_\alpha^{\sum\limits_{l=j}^{f-1}i_lp^{f+l}}-A_\alpha^{p+1+p^{f}-p-\sum\limits_{l=j}^{f-1}i_lp^{l}}B_\alpha^{\sum\limits_{l=j}^{f-1}i_lp^{l}}\right) \nonumber\\
 			&= -\binom{r_0-2}{i_0-1}\left(A_\alpha^{r_0-i_0-\sum\limits_{l=j}^{f-1}i_lp^{f+l}}B_\alpha^{i_0-1+p^f+\sum\limits_{l=j}^{f-1}i_lp^{f+l}}-A_\alpha^{r_0-1-i_0+p^f-\sum\limits_{l=j}^{f-1}i_lp^{l}}B_\alpha^{i_0+\sum\limits_{l=j}^{f-1}i_lp^{l}}\right) \nonumber\\
 			&{\hspace{15pt}}-\binom{r_0-2}{i_0}\left(A_\alpha^{r_0-1-i_0+p^f-\sum\limits_{l=j}^{f-1}i_lp^{f+l}}B_\alpha^{i_0+\sum\limits_{l=j}^{f-1}i_lp^{f+l}}-A_\alpha^{r_0-1-i_0+p^{f}-\sum\limits_{l=j}^{f-1}i_lp^{l}}B_\alpha^{i_0+\sum\limits_{l=j}^{f-1}i_lp^{l}}\right).
        \end{align} }
	\!\! \!\! \!\! Again, ignoring the common factor $(r_0-1-i_0)! i_0!$ and using \eqref{apply nabla second time}, a similar computation
        shows that the second summand on the left hand side
        of \eqref{applying psi map D_j case} is
        \begin{align}
           \label{second summand of the claim}
 			&\binom{r_0-2}{i_0-1}\left(A_\alpha^{r_0-i_0-\sum\limits_{l=j}^{f-1}i_lp^{f+l}}B_\alpha^{i_0-1+p^f+\sum\limits_{l=j}^{f-1}i_lp^{f+l}}-A_\alpha^{r_0-1-i_0+p^{f}-\sum\limits_{l=j}^{f-1}i_lp^{l}}B_\alpha^{i_0+\sum\limits_{l=j}^{f-1}i_lp^{l}}\right) \nonumber\\
 			&+\binom{r_0-2}{i_0}\left(A_\alpha^{r_0-1-i_0+p^f-\sum\limits_{l=j}^{f-1}i_lp^{f+l}}B_\alpha^{i_0+\sum\limits_{l=j}^{f-1}i_lp^{f+l}}-A_{\alpha}^{r_0-1-i_0+p^{f}-\sum\limits_{l=j}^{f-1}i_lp^{l}}B_\alpha^{i_0+\sum\limits_{l=j}^{f-1}i_lp^{l}}\right),
       \end{align}
       which is the negative of 
       \eqref{first summand of the claim}. This proves the claim \eqref{applying psi map D_j case}.
       Hence {\rm Im}~$D_j\subset \ker\psi$ for all $1\leq j\leq f-1.$ \\
 
       \noindent \textbf{Case 2:} The proof that ${\rm Im}~D_0\subset\ker\psi$ is similar and is omitted. \\

        Combining Cases $1$ and $2,$ we see that 
        $\langle D_0,\dots, D_{f-1}\rangle \otimes\bigotimes_{j=0}^{f-1} V_{p-1}^{{\rm Fr}^j} \subset \ker\psi$. \\
        
	Next we show that $\ker\psi\subset \langle D_0,\dots, D_{f-1}\rangle\otimes\bigotimes_{j=0}^{f-1} V_{p-1}^{{\rm Fr}^j}.$ 
	Let 
		\[P\otimes Q  = \sum\limits_{\vec{i}=\vec{0}}^{\vec{r+q-1}}\sum\limits_{\vec{j}=\vec{0}}^{\vec{q-1}}b_{\vec{i},\vec{j}}X_0^{r_0+p-1-i_0}Y_0^{i_0}\left(\prod_{l=1}^{f-1} X_l^{p-1-i_l}Y_l^{i_l}\right)\left(\prod_{l=0}^{f-1}S_l^{p-1-j_l}T_l^{j_l}\right) \in \ker \psi,\]
	where $\vec{i}=(i_0,\dots, i_{f-1}), \vec{j}=(j_0,\dots, j_{f-1})$ and 
        $\sum\limits_{\vec{i}=\vec{0}}^{\vec{r+q-1}}:=\sum\limits_{i_0=0}^{r_0+p-1}\sum\limits_{i_1=0}^{p-1}\dots\sum\limits_{i_{f-1}=0}^{p-1}$,
        $\sum\limits_{\vec{j}=\vec{0}}^{\vec{q-1}}:=\sum\limits_{j_0=0}^{p-1}\dots\sum\limits_{j_{f-1}=0}^{p-1}$.
	By definition of $\psi$, for all $\left( \begin{smallmatrix} a & b \\ c & d \end{smallmatrix} \right) \in G({\mathbb F}_q)$, we have 
	\begin{align}
		\label{after applying psi}
			&\sum\limits_{\vec{i}=\vec{0}}^{\vec{r+q-1}}\sum\limits_{\vec{j}=\vec{0}}^{\vec{q-1}}b_{\vec{i},\vec{j}}\nabla_0^{r_0-2}\left(X_0^{r_0+p-1-i_0}Y_0^{i_0}\right)\left(\prod_{l=1}^{f-1} X_l^{p-1-i_l}Y_l^{i_l}\right)\Big\vert_{(A_\alpha, B_\alpha,\dots, A_\alpha^{p^{f-1}}, B_\alpha^{p^{f-1}})}^{(A_\alpha^{p^f}, B_\alpha^{p^f},\dots, A_\alpha^{p^{2f-1}}, B_\alpha^{p^{2f-1}})} \nonumber \\
			&\qquad \qquad {\hspace{40pt}}\cdot
                          \prod_{l=0}^{f-1}\left(S_l^{p-1-j_l}T_l^{j_l}\Big\vert_{(0, 0)}^{(A_\alpha^{p^{f+l}}, B_\alpha^{p^{f+l}})}\right) = 0.
	\end{align}
	Note, for each $\vec{j}$, 
	\begin{equation}
		\label{jpart}
			\prod_{l=0}^{f-1}\left(S_l^{p-1-j_l}T_l^{j_l}\Big\vert_{(0, 0)}^{(A_\alpha^{p^{f+l}}, B_\alpha^{p^{f+l}})}\right)=A_\alpha^{\sum\limits_{l=0}^{f-1}(p-1-j_l)p^{f+l}}B_\alpha^{\sum\limits_{l=0}^{f-1}j_lp^{f+l}}=A_\alpha^{1-q-\sum\limits_{l=0}^{f-1}j_lp^{f+l}}B_\alpha^{\sum\limits_{l=0}^{f-1}j_lp^{f+l}}.\\
	\end{equation}
	Now by fixing $\vec{j}$ and only considering the sum over $\vec{i}$ in \eqref{after applying psi}, we have
	\begin{align*}
		& \sum\limits_{\vec{i}=\vec{0}}^{\vec{r+q-1}}b_{\vec{i},\vec{j}}\left(   A_\alpha\dfrac{\partial}{\partial X_0}+ B_\alpha\dfrac{\partial}{\partial Y_0} \right)^{r_0-2}\left(X_0^{r_0+p-1-i_0}Y_0^{i_0}\right)\left(\prod_{l=1}^{f-1} X_l^{p-1-i_l}Y_l^{i_l}\right)\Big\vert_{(A_\alpha, B_\alpha,\dots, A_\alpha^{p^{f-1}}, B_\alpha^{p^{f-1}})}^{(A_\alpha^{p^f}, B_\alpha^{p^f},\dots, A_\alpha^{p^{2f-1}}, B_\alpha^{p^{2f-1}})}\\
		&\quad =\sum\limits_{\vec{i}=\vec{0}}^{\vec{r+q-1}}\sum\limits_{k_0=0}^{r_0-2}b_{\vec{i},\vec{j}}\binom{r_0-2}{k_0}[r_0+p-1-i_0]_{r_0-2-k_0}[i_0]_{k_0}A_\alpha^{r_0-2-k_0}B_\alpha^{k_0}\\
		&\qquad \qquad \qquad {\hspace{15pt}}\cdot X_0^{p+1-(i_0-k_0)}Y_0^{i_0-k_0}\left(\prod_{l=1}^{f-1} X_l^{p-1-i_l}Y_l^{i_l}\right)\Big\vert_{(A_\alpha, B_\alpha,\dots, A_\alpha^{p^{f-1}}, B_\alpha^{p^{f-1}})}^{(A_\alpha^{p^f}, B_\alpha^{p^f},\dots, A_\alpha^{p^{2f-1}}, B_\alpha^{p^{2f-1}})},
	\end{align*}
	which by observing 
		$p^{f+1}+p^f-(i_0-k_0)p^f+\sum\limits_{l=1}^{f-1}(p-1-i_l)p^{f+l}=p^{2f}+p^f(1+k_0)-\sum\limits_{l=0}^{f-1}i_lp^{f+l}$,
	and writing
	\begin{eqnarray}
		\label{constant C}
			C_{i_0, k_0}= \binom{r_0-2}{k_0}[r_0+p-1-i_0]_{r_0-2-k_0}[i_0]_{k_0} = (r_0-2)! \binom{r_0+p-1-i_0}{r_0-2-k_0} \binom{i_0}{k_0},
	\end{eqnarray}
 	equals
{\small \begin{align}
                \label{fixed j}
		&\sum\limits_{\vec{i}=\vec{0}}^{\vec{r+q-1}}\sum\limits_{k_0=0}^{r_0-2}b_{\vec{i},\vec{j}}\> C_{i_0, k_0}A_\alpha^{r_0-2-k_0}B_\alpha^{k_0} \notag \\
		&\quad{\hspace{15pt}}\cdot \left(A_\alpha^{p^{2f}+p^f(1+k_0)-\sum\limits_{l=0}^{f-1}i_lp^{f+l}}B_\alpha^{(i_0-k_0)p^f+\sum\limits_{l=1}^{f-1}i_lp^{f+l}}-A_\alpha^{p^f+1+k_0-\sum\limits_{l=0}^{f-1}i_lp^l}B_\alpha^{i_0-k_0+\sum\limits_{l=1}^{f-1}i_lp^l} \right) \notag \\
		&\quad=\sum\limits_{\vec{i}=\vec{0}'}^{\vec{r+q-1}'}b_{\vec{i},\vec{j}} \left( \sum\limits_{k_0=0}^{r_0-2}C_{i_0, k_0}   A_\alpha^{r_0+q-1-k_0(1-q)-q\left( \sum\limits_{l=0}^{f-1}i_lp^l \right) } B_\alpha^{k_0(1-q)+ q\left(     \sum\limits_{l=0}^{f-1}i_lp^l  \right)     }\right) \notag \\
		&\quad{\hspace{15pt}}-\sum\limits_{\vec{i}=\vec{0}'}^{\vec{r+q-1}'}b_{\vec{i},\vec{j}} (r_0-1)!A_\alpha^{r_0+q-1-\sum\limits_{l=0}^{f-1}i_lp^l}
		B_\alpha^{\sum\limits_{l=0}^{f-1}i_lp^l}.
	\end{align}}
	The last equality holds since 
        by \eqref{sum}, we have
 		\[\sum\limits_{k_0=0}^{r_0-2}C_{i_0, k_0}=\sum\limits_{k_0=0}^{r_0-2}\binom{r_0-2}{k_0}[r_0+p-1-i_0]_{r_0-2-k_0}[i_0]_{k_0}=(r_0-1)! \mod p.\]
                Moreover, we have adorned the limits in the last two sums with $'$s to indicate that we drop the
                terms corresponding to $\vec{i}=\vec{0}$ and $\vec{i}=\vec{r+q-1}$. Indeed, if $\vec{i}=\vec{0},$ then $i_0=0$ and 
		\[C_{0, k_0}=\begin{cases}
			(r_0-1)!,&\text{if}~k_0=0,\\
			0, &\text{otherwise}.
		\end{cases}
		\]
	So the term for $\vec{i}=\vec{0}$ in \eqref{fixed j} is 
	\[b_{\vec{0}, \vec{j}}\left((r_0-1)!A_\alpha^{r_0+q-1}-(r_0-1)!A_\alpha^{r_0+q-1}\right)=0.\]
        Similarly, one may check that the term for $\vec{i}=\vec{r_0+q-1}$ is zero.  

  	For notational convenience, set
	\[(*) := \left( \sum\limits_{k_0=0}^{r_0-2}C_{i_0, k_0}   A_\alpha^{r_0+q-1-k_0(1-q)-q\left( \sum\limits_{l=0}^{f-1}i_lp^l \right) } B_\alpha^{k_0(1-q)+ q\left(     \sum\limits_{l=0}^{f-1}i_lp^l  \right)     }\right)-
	(r_0-1)!A_\alpha^{r_0+q-1-\sum\limits_{l=0}^{f-1}i_lp^l}
	B_\alpha^{\sum\limits_{l=0}^{f-1}i_lp^l}.\]
	Then \eqref{fixed j} decomposes as
{\small	\begin{eqnarray}
		\label{expanded sum}
			&&\sum\limits_{i_0=1}^{r_0-1}b_{i_0,0,\dots,0,\vec{j}}\>(*)+\sum\limits_{i_0=p}^{r_0+p-2}b_{i_0,p-1,\dots,p-1,\vec{j}}\>(*)\nonumber \\
			&&\quad +\sum\limits_{i_0=r_0}^{p-1}\sum\limits_{i_1=0}^{p-1}\dots\sum\limits_{i_{f-1}=0}^{p-1}b_{i_0,\dots,i_{f-1},\vec{j}}\>(*) \nonumber \\
			&&\quad+\sum\limits_{t=1}^{f-1}\left(\sum\limits_{i_0=0}^{r_0-1}\sum\limits_{i_1=0}^{0}\dots\sum\limits_{i_{t-1}=0}^{0}\sum\limits_{i_t=1}^{p-1}\sum\limits_{i_{t+1}=0}^{p-1}\dots\sum\limits_{i_{f-1}=0}^{p-1}b_{i_0,0,\dots,0,i_t,\dots,i_{f-1},\vec{j}}\>(*)\right) \nonumber \\
			&&\quad+\sum\limits_{t=1}^{f-1}\left(\sum\limits_{i_0=p}^{r_0+p-1}\sum\limits_{i_1=p-1}^{p-1}\dots\sum\limits_{i_{t-1}=p-1}^{p-1}\sum\limits_{i_t=0}^{p-2}\sum\limits_{i_{t+1}=0}^{p-1}\dots\sum\limits_{i_{f-1}=0}^{p-1}b_{i_0,p-1,\dots,p-1,i_t,\dots,i_{f-1},\vec{j}}\>(*)\right).  
	\end{eqnarray} }
	By taking $i_1=i_2=\dots=i_{f-1}=0$, the coefficient $(*)$ of $b_{i_0,0,\dots,0,\vec{j}}$ in the first sum above is
		\[ \left( \sum\limits_{k_0=0}^{r_0-2}C_{i_0, k_0}   A_\alpha^{r_0+q-1-k_0(1-q)-i_0q} B_\alpha^{k_0(1-q)+ i_0q}\right)-
		(r_0-1)!A_\alpha^{r_0+q-1-i_0}
		B_\alpha^{i_0}.\]
	Note that, by \eqref{constant C}, for $0\leq i_0\leq r_0-1,$ we have
	\begin{eqnarray*}
		C_{i_0, k_0}
		&=&\begin{cases}
		(r_0-2)!i_0,        & \text{if}~k_0=i_0-1,\\
		(r_0-2)!(r_0-1-i_0), & \text{if}~k_0=i_0,\\
		0, & \text{otherwise}.
				\end{cases}
	\end{eqnarray*}
	Thus the coefficient of $b_{i_0,0,\dots,0,\vec{j}}$ becomes
		\[i_0(r_0-2)!A_\alpha^{r_0-i_0}B_\alpha^{i_0+q-1}+\left((r_0-1-i_0)(r_0-2)!-(r_0-1)!\right)A_\alpha^{r_0+q-1-i_0}B_\alpha^{i_0},\]
	which further equals
	\begin{equation}
		\label{first kind 1}
			i_0(r_0-2)!\left(A_\alpha^{r_0-i_0}B_\alpha^{i_0+q-1}-A_\alpha^{r_0+q-1-i_0}B_\alpha^{i_0}\right).
	\end{equation}
	Now use the transformation $i_0\mapsto i_0+p-1$ in the second sum of \eqref{expanded sum}. By taking
        $i_1=i_2=\dots=i_{f-1}=p-1$, the  coefficient $(*)$ of $b_{i_0+p-1,p-1,\dots,p-1,\vec{j}}$ is given by
		\[ \left( \sum\limits_{k_0=0}^{r_0-2}C_{i_0+p-1, k_0}   A_\alpha^{r_0+q-1-k_0(1-q)-q(i_0+q-1)} B_\alpha^{k_0(1-q)+ q(i_0+q-1)}\right)-
			(r_0-1)!A_\alpha^{r_0-i_0}
				B_\alpha^{i_0+q-1}.\]
		For $0\leq i_0\leq r_0-1,$ we have
		\begin{eqnarray*}
		C_{i_0+p-1, k_0}&=&\binom{r_0-2}{k_0}[r_0-i_0]_{r_0-2-k_0}[i_0+p-1]_{k_0}
				\> = \>\begin{cases}
							(r_0-2)! (i_0-1), 	 & \text{if}~k_0=i_0-2,\\
							(r_0-2)! (r_0-i_0), & \text{if}~k_0=i_0-1,\\
								0,      					   &\text{otherwise}.
                       \end{cases}
	\end{eqnarray*}
	Thus the coefficient of $b_{i_0+p-1, p-1, p-1, \dots, p-1, \vec{j}}$ becomes
		\[\left((r_0-2)!(i_0-1)-(r_0-1)!\right) A_\alpha^{r_0-i_0}B_\alpha^{i_0+q-1} + 
			(r_0-2)! (r_0-i_0) A_\alpha^{r_0+q-1-i_0} B_\alpha^{i_0}, \]
	which equals
	\begin{equation}
		\label{first kind 2}
			-(r_0-2)!(r_0-i_0)\left( A_\alpha^{r_0-i_0}B_\alpha^{i_0+q-1} - A_\alpha^{r_0+q-1-i_0} B_\alpha^{i_0} \right).
	\end{equation}
	Finally, using $i_0\mapsto i_0+p$ and $i_t\mapsto i_t-1$ in the fifth sum of \eqref{expanded sum}, and noting
        that $\sum\limits_{l=0}^{f-1}i_lp^l$ changes to 
		\[i_0+p+\sum\limits_{l=1}^{t-1} (p-1) p^l +(i_t-1)p^t + \sum\limits_{l=t+1}^{f-1}i_lp^l=i_0+\sum\limits_{l=t}^{f-1}i_lp^l,\]
	the coefficient $(*)$ of $b_{i_0+p, p-1, \dots, p-1, i_t-1, i_{t+1}, \dots, i_{f-1}, \vec{j} }$ equals
	\begin{eqnarray}
		\label{third kind 2}
			\sum\limits_{k_0=0}^{r_0-2}C_{i_0+p, k_0}   A_\alpha^{r_0+q-1-k_0(1-q)-q\left( i_0 + \sum\limits_{l=t}^{f-1}i_lp^l \right) } B_\alpha^{k_0(1-q)+ q\left(i_0 +  \sum\limits_{l=t}^{f-1}i_lp^l  \right) } 
			-(r_0-1)!A_\alpha^{r_0+q-1-i_0-\sum\limits_{l=t}^{f-1}i_lp^l}
			B_\alpha^{i_0+\sum\limits_{l=t}^{f-1}i_lp^l}
	\end{eqnarray}
	which is exactly the coefficient of $b_{i_0, 0, \dots, 0, i_t, \dots, i_{f-1}, \vec{j}}$ obtained by taking $i_1=i_2=\dots=i_{t-1}=0$ in $(*)$, since
        $C_{i_0+p, k_0}=C_{i_0, k_0} \mod p$ for $0\leq i_0\leq r_0-1.$
	Summarizing, the transformation $i_0\mapsto i_0+p-1$ in the second sum of $\eqref{expanded sum}$ and the
        transformations $i_0\mapsto i_0+p$ and $i_t\mapsto i_t-1$ in the fifth sum
        of $\eqref{expanded sum}$, together with \eqref{first kind 1}, \eqref{first kind 2} and \eqref{third kind 2}, allow us to rewrite $\eqref{expanded sum}$ as
	\begin{eqnarray}
		\label{arranged expanded sum}
			&& \sum_{i_0=1}^{r_0-1} (r_0-2)! \left( i_0b_{i_0, 0, \dots, 0, \vec{j}}-(r_0-i_0) b_{i_0+p-1, p-1, \dots, p-1, \vec{j}}\right) \left(A_\alpha^{r_0-i_0}B_\alpha^{i_0+q-1} - A_\alpha^{r_0+q-1-i_0} B_\alpha^{i_0}\right) \nonumber \\ \nonumber
			&&\quad +\sum\limits_{i_0=r_0}^{p-1}\sum\limits_{i_1=0}^{p-1}\dots\sum\limits_{i_{f-1}=0}^{p-1}b_{i_0,\dots,i_{f-1},\vec{j}}\>(*) \\
			&& \quad + \sum_{t=1}^{f-1} \left(\sum\limits_{i_0=0}^{r_0-1}\sum\limits_{i_t=1}^{p-1}\sum\limits_{i_{t+1}=0}^{p-1}\dots\sum\limits_{i_{f-1}=0}^{p-1}\left( b_{i_0,0,\dots,0,i_t,\dots,i_{f-1},\vec{j}} + b_{i_0+p,p-1,\dots,p-1,i_t-1,\dots,i_{f-1},\vec{j}}\right) (*)\right). 
	\end{eqnarray}
        
	Now, for each 
        $\vec{j}=(j_0, \dots, j_{f-1})$, let $i=\sum_{l=0}^{f-1}i_lp^l$ and $j=\sum_{l=0}^{f-1}j_lp^l$, and set {\small
	\begin{eqnarray*}
		\label{variables}      
		X_{i, j}&=& \begin{cases}
                        i_0b_{i_0, 0, \dots, 0, \vec{j}} - (r_0-i_0) b_{i_0+p-1, p-1, \dots, p-1, \vec{j}}, & \text{if } 1 \leq i_0 \leq r_0-1, i_t = 0 \text{ for } t \neq 0, \\
		   	b_{i_0, \dots, i_{f-1}, \vec{j}}, & \text{if } r_0 \leq i_0 \leq p-1, 0 \leq i_t \leq p-1 \text{ for } t \neq 0,  \\ 
                        b_{i_0,0,\dots,0,i_t,\dots,i_{f-1},\vec{j}} + b_{i_0+p,p-1,\dots,p-1,i_t-1,\dots,i_{f-1},\vec{j}}, & \text{if } 0 \leq i_0 \leq r_0 -1, t \geq 1
                        \text{ smallest s.t. } i_t \neq 0.
	\end{cases}
	\end{eqnarray*}}
      \noindent Note that $1 \leq i \leq q-1$ and $0 \leq j \leq q-1$, so there are $(q-1)q$ variables $X_{i,j}$, with the first kind
      running in the range $1 \leq i = i_0 \leq r_0-1$, and the second and third
      kind running in the range $r_0 \leq i \leq q-1$.

      By \eqref{after applying psi}, and \eqref{jpart}, and \eqref{arranged expanded sum}  but with the variables $X_{i,j}$, we obtain
	\begin{eqnarray*}
		&& \sum_{i=1}^{r_0-1}  \sum_{j=0}^{q-1} (r_0-2)! X_{i, j} \left(A_\alpha^{r_0-i}B_\alpha^{i+q-1} - A_\alpha^{r_0+q-1-i} B_\alpha^{i}\right)\cdot A_\alpha^{1-q-jq}B_{\alpha}^{jq}\\
		&&\quad + \sum_{i=r_0}^{q-1} \sum_{j=0}^{q-1} X_{i, j} \left(\left( \sum\limits_{k_0=0}^{r_0-2}C_{i_0, k_0}   A_\alpha^{r_0+q-1-k_0(1-q)-iq } B_\alpha^{k_0(1-q)+ iq }\right)-
		(r_0-1)!A_\alpha^{r_0+q-1-i}
		B_\alpha^i \right)\\
		&&{\hspace{300pt}} \cdot A_\alpha^{1-q-qj}B_{\alpha}^{qj} = 0,
	\end{eqnarray*}
	which by dividing by $(r_0-2)!$ and setting (cf. \eqref{constant C})
		\[Z_{i_0, k_0}=\dfrac{C_{i_0, k_0}}{(r_0-2)!}=\binom{r_0+p-1-i_0}{r_0-2-k_0} \binom{i_0}{k_0}\] 
	yields the system of equations (that was obtained earlier in \eqref{main} for $q = p$) 
	\begin{eqnarray*}
		\label{arranged expanded sum with jpart}
			&& \sum_{i=1}^{r_0-1}  \sum_{j=0}^{q-1} X_{i, j} \left(A_\alpha^{r_0+1-q-i-jq}B_\alpha^{i+jq+q-1} - A_\alpha^{r_0-i-jq} B_\alpha^{i+jq}\right)\\
			&& \quad + \sum_{i=r_0}^{q-1} \sum_{j=0}^{q-1} X_{i, j} \left(\left( \sum\limits_{k_0=0}^{r_0-2}Z_{i_0, k_0}   A_\alpha^{r_0-k_0(1-q)-(i+j)q } B_\alpha^{k_0(1-q)+ (i+j)q }\right)-
			(r_0-1)A_\alpha^{r_0-i-jq}
			B_\alpha^{i+jq} \right) =0.
	\end{eqnarray*}
        As before, we separate the equations according to the congruence class $1\leq n \leq q-1$ of the sum $i+j$  so that we obtain $q-1$ separate systems
        of equations (each with $q$ distinct variables). Again in order to work in the basis ${\mathcal B}_q$, we convert all flips and flops
        to elements of ${\mathcal B}_q$ using \eqref{flip in general} and \eqref{flop in general}.
        The resulting equations and  corresponding coefficient matrices obtained and the computation using row operations
        to show that these matrices have non-zero determinant are identical to the case of $q = p$ treated earlier. The only real difference 
        is that in the formulas for the determinant (obtained earlier in three cases depending on the relative size of $n$),
        one needs to replace $p$ by $q$ everywhere. We conclude that all $X_{i,j} = 0$.
	That is, for each $\vec{j}$,  we have 
	\begin{itemize}
		\item[(1)] if $1\leq i_0 \leq r_0-1$ and $i_j = 0$ for $j \neq 0$, then
			\[i_0b_{i_0, 0, \dots, 0, \vec{j}}=(r_0-i_0) b_{i_0+p-1, p-1, \dots, p-1, \vec{j}},\]
		\item[(2)] if $r_0\leq i_0 \leq p-1$ and $0\leq i_j \leq p-1$ for $1\leq j\leq f-1,$ then  \[b_{i_0, \dots, i_{f-1}, \vec{j}}=0,\]
		\item[(3)] if $0\leq i_0\leq r_0-1$ and there is a (smallest)  $1\leq t \leq f-1$ with $i_t \neq 0$ (so $i_j= 0$ for $1 \leq j \leq t-1$ and
                  $0 \leq i_{t+1}, \ldots, i_{f-1}\leq p-1),$ then
		\[b_{i_0, 0,\dots, 0, i_t, \dots, i_{f-1}, \vec{j}}=-b_{i_0+p, p-1, \dots, p-1, i_t-1, i_{t+1}, \dots, i_{f-1}, \vec{j}}.\] 
		\end{itemize}
	Then,
	\begin{eqnarray*}
		P\otimes Q	
		&=&  \sum\limits_{\vec{j}=\vec{0}}^{\vec{q-1}}   \left( \sum\limits_{\vec{i}=\vec{0}}^{\vec{r+q-1}}b_{\vec{i},\vec{j}}X_0^{r_0+p-1-i_0}Y_0^{i_0}\prod_{l=1}^{f-1} X_l^{p-1-i_l}Y_l^{i_l} \right) \otimes \prod_{l=0}^{f-1}S_l^{p-1-j_l}T_l^{j_l},
	\end{eqnarray*}
	where each polynomial in the parentheses satisfies $(1)$, $(2)$, $(3)$. 
        So by Lemma \ref{space generated by D_j's}, we conclude $P\otimes Q \in \langle D_0, \dots, D_{f-1} \rangle \otimes \bigotimes_{j=0}^{f-1} V_{p-1}^{{\rm Fr}^j}$, showing $\ker\psi\subset \langle D_0, \dots, D_{f-1} \rangle\otimes\bigotimes_{j=0}^{f-1} V_{p-1}^{{\rm Fr}^j}.$
        We finally have
		\[\ker\psi=\langle D_0, \dots, D_{f-1} \rangle\otimes\bigotimes_{j=0}^{f-1} V_{p-1}^{{\rm Fr}^j}.\]
                By Remark \ref{dimension of the space generated by D_j's}, the dimension of $\langle D_0, \dots, D_{f-1} \rangle$ over 
                ${\mathbb F}_q$ equals $r_0p^{f-1}+1,$
 		so after tensoring with ${\mathbb F}_{q^2}$, 
		\[\dim_{\mathbb{F}_{q^2}} \left(\dfrac{\bigotimes_{j=0}^{f-1}V_{r_j+p-1}^{{\rm Fr}^j}}{\langle D_0,\dots, D_{f-1}\rangle}\right) =(r_0+p)p^{f-1}- r_0p^{f-1}-1 =q-1.\]
	Thus, over ${\mathbb F}_{q^2}$ we have
		$\dim_{\mathbb{F}_{q^2}}\left(\dfrac{\bigotimes_{j=0}^{f-1}V_{r_j+p-1}^{{\rm Fr}^j}}{\langle D_0,\dots, D_{f-1}\rangle}\otimes\bigotimes_{j=0}^{f-1} V_{p-1}^{{\rm Fr}^j} \right)=q(q-1)=\dim_{\mathbb{F}_{q^2}} {\rm ind}_{T(\mathbb{F}_q)}^{G(\mathbb{F}_q)}\omega_{2f}^r$,
                and so $\psi$ must be an isomorphism. This completes the proof of Theorem~\ref{main theorem cuspidal twisted}
                (which is also Theorem~\ref{cuspidal-twisted}).
	\end{proof}

        \vspace{.03cm}
        {\noindent \bf Acknowledgements:} We thank A. Chitrao, C. Khare and S. Varma for useful  discussions. Both authors thank the referee for many useful comments and suggestions. The first author thanks ANU, Canberra for
        its hospitality during March 2024.

	\end{document}